\documentclass{amsart}

\usepackage[utf8]{inputenc}
\usepackage{color}

\usepackage{hyperref}
\usepackage{cleveref}
\usepackage[all]{xy}
\xyoption{arc}
\usepackage{amssymb,amsmath}
\usepackage[a4paper,left=20mm,top=5mm,right=20mm,bottom=5mm,includeheadfoot,headheight=20mm, headsep=5mm, footskip=15mm]{geometry}

\usepackage{stackrel}
\usepackage{csquotes}
\usepackage[english]{babel}
\usepackage{enumerate}
\usepackage[final]{pdfpages}
\usepackage{graphicx}
\usepackage{imakeidx}
\usepackage{latexsym,keyval,ifthen}
\usepackage[makeroom]{cancel}
\usepackage{mathalfa}
\usepackage{mathdots}
\usepackage{multicol}
\usepackage{prerex}
\usepackage{randomwalk}
\usepackage{soul}
\usepackage{tgpagella}
\usepackage{tikz}
\usepackage{tikz-cd}
\usetikzlibrary{arrows,arrows.meta}
\usepackage{wrapfig}
\usepackage{xgalley}
\usepackage{calc}
\usepackage{xcolor}

\newtheorem{theorem}{Theorem}[section]
\newtheorem{lemma}[theorem]{Lemma}
\newtheorem{proposition}[theorem]{Proposition}
\newtheorem{conjecture}[theorem]{Conjecture}
\newtheorem{corollary}[theorem]{Corollary}

\theoremstyle{definition}
\newtheorem{definition}[theorem]{Definition}

\newtheorem{assumption}[theorem]{Assumption}

\theoremstyle{remark}
\newtheorem{remark}[theorem]{Remark}

\numberwithin{equation}{section}

\newcommand{\ipa}[1]{\left(#1\right)}
\newcommand{\Gal}{{\mathrm {Gal}}}

\newcommand{\ord}{\mathrm{ord}}
\newcommand{\Pic}{\mathrm{Pic}}

\newcommand{\M}{\mathrm{M}}

\newcommand{\sign}{\mathrm{sign}}
\newcommand{\Ind}{{\mathrm{Ind}}}
\newcommand{\PGL}{{\mathrm{PGL}}}

\newcommand{\GL}{{\mathrm{GL}}}
\newcommand{\SU}{{\mathrm{SU}}}

\newcommand{\SO}{{\mathrm {SO}}}

\newcommand{\SL}{{\mathrm {SL}}}

\newcommand{\Z}{{\mathbb Z}}
\newcommand{\bH}{{\mathbb H}}
\newcommand{\mfM}{{\mathfrak M}}
\newcommand{\A}{{\mathbb A}}
\newcommand{\Q}{{\mathbb Q}}
\newcommand{\C}{{\mathbb C}}
\newcommand{\R}{{\mathbb R}}

\newcommand{\N}{{\mathbb N}}
\newcommand{\G}{{\mathbb G}}

\newcommand{\I}{{\mathbb I}}
\newcommand{\PP}{{\mathbb P}}

\newcommand{\cA}{{\mathcal A}}

\newcommand{\cR}{{\mathcal R}}

\newcommand{\cS}{{\mathcal S}}
\newcommand{\cD}{{\mathcal D}}

\newcommand{\cB}{{\mathcal B}}

\newcommand{\cK}{{\mathcal K}}
\newcommand{\cG}{{\mathcal G}}

\newcommand{\cM}{{\mathcal M}}

\newcommand{\cP}{{\mathcal P}}
\newcommand{\cO}{{\mathcal O}}

\newcommand{\mfg}{{\mathfrak g}}
\newcommand{\mfgltwo}{{\mathfrak gl}_2}

\newcommand{\ra}{{\rightarrow}}

\newcommand{\Hom}{{\mathrm {Hom}}}

\newcommand{\sgn}{\mathrm{sgn}}
\newcommand{\fp}{{\mathfrak{p}}}
\newcommand{\fq}{{\mathfrak{q}}}

\title[Periods of modular forms and applications to the conjectures of Oda and of Prasanna--Venkatesh]{Comparison of periods of modular forms and  applications to the conjectures of Oda and of Prasanna--Venkatesh}

\author{Xavier Guitart}
\address{Departament de Matemàtiques i Informàtica, Universitat de Barcelona}
\email{xevi.guitart@gmail.com}
\author{Santiago Molina}
\address{Departament de Matemàtica, Universitat de Lleida}
\email{santiago.molina@udl.cat} 

\begin{document}
\maketitle

\begin{abstract}
  We establish several formulas relating periods of modular forms on quaternion algebras over number fields to special values of $L$-functions. Our main inputs are the cohomological techniques for working with periods introduced in \cite{preprintsanti2}, along with explicit versions of the Waldspurger formula due to Cai--Shu--Tian \cite{CST}. We work in general even positive weights; when specialized to parallel weight 2, our formulas provide partial evidence for the conjectures of Oda and of Prasanna--Venkatesh in the case of forms associated to elliptic curves.
\end{abstract}

\section{Introduction}

 Periods of modular forms play an important role in number theory due to their connection with special values of $L$-functions and their appearance in several landmark conjectures, such as the Birch and Swinnerton-Dyer Conjecture or Deligne's conjecture on critical values of $L$-functions.  The goal of the present article is to prove several formulas relating periods of modular forms with special values of $L$-functions, and to provide applications to the conjectures of Oda and of Prasanna--Venkatesh.

\subsection{Statement of the main results}\label{sec:main results}
 
The setting we consider is the following. Let $F$ be a number field of degree $d$ and let $B$ be a quaternion algebra over $F$. Denote by $G$ the algebraic group associated to $B^\times/F^\times$. Let $\pi$ be an automorphic representation of $G$ of weight $\underline{k}=(k_\nu)_{\nu:F\hookrightarrow \bar \Q}\in (2\N)^d$ and conductor $N$. Let $L_{\underline k}$ be the number field fixed by ${\{\tau\in\Gal(\bar \Q/\Q)\colon k_{\tau\nu}=k_\nu \text{ for all }\nu \}}$, and let $L_\pi$ be the coefficient field of $\pi$, defined as the smallest extension of $L_{\underline k}$ that contains the Hecke eigenvalues of $\pi$  (in particular, for parallel weight, $L_\pi$ is the field of Hecke eigenvalues). Let $\Sigma_B$ be the set of infinite places of $F$ where $B$ splits and let $\varepsilon  \in \{\pm 1\}^{\Sigma_B}$ be a \emph{sign vector}; that is, a choice of sign for each place in $\Sigma_B$.  We say that $\varepsilon$ is of \emph{lowest degree} if it takes value $+1$ at all complex places, and of \emph{highest degree} if it takes value $-1$ at all complex places.  We define the periods of $\pi$ attached to a sign vector of lowest or highest degree as in \cite{preprintsanti2}. More precisely, inspired by the approach of \cite{harder} (see also \cite{ESsanti}), for any sign vector $\varepsilon$ we define an Eichler--Shimura morphism ${\rm ES}_\varepsilon$ that associates to any modular form for $G$ a certain cohomology class.  If $\varepsilon$ is of lowest or highest degree, the class corresponding  under $\mathrm{ES}_\varepsilon$ to a normalized newform for $\pi$ can be divided by a period, that we call $\Omega_\varepsilon^\pi$, so that it becomes $L_\pi$-rational. This determines the period $\Omega_\varepsilon^\pi$ up to multiplication by an element of $L_\pi^\times$. See \S \ref{cohoAG} for the construction of the Eichler--Shimura morphisms and \S \ref{sec:periods} for the definition of the periods.

We now state the three main results of this note, in some cases under simplifying assumptions that, while not strictly necessary, allow for a clearer presentation in the introduction. The most general versions are given in the main body of the text, with references to their location indicated in parentheses. The first result is in the particular case where $G=\PGL_2$. For consistency with the notation that we will use later on, let us denote in this case by $\Pi$ an automorphic representation of $\PGL_2$ and  by $\Omega_\varepsilon^\Pi$ the periods associated to a sign vector $\varepsilon \in \{\pm 1\}^{\Sigma_F}$  of lowest or highest degree (here $\Sigma_F$ denotes the set of infinite places of $F$).

\begin{proposition}[Corollary \ref{ShimRel}]\label{prop1}
    Let $\rho:\I_F/F^\times\rightarrow \{\pm 1\}$ be a quadratic Hecke character and let $E_\rho/F$ be the associated quadratic extension. Let $\varepsilon$ be the sign vector of lowest degree defined by $\varepsilon_\sigma=\rho_{\sigma}(-1)$ for all $\sigma\in\Sigma_F$. Then 
    \[
    \frac{L(1/2,\Pi, \rho)}{|d_F|^{\frac{1}{2}}\cdot\pi^{\frac{\underline k}{2}}\cdot i^{s_\varepsilon}\cdot |D_{\rho}|^{\frac{1}{2}}\cdot \Omega_\varepsilon^\Pi}\,\text{ belongs to }\, L_{\Pi},
    \]
    where $|d_F|$ (resp. $|D_\rho|$) denotes the norm of the different ideal of $F$ (resp. the norm of the the relative discriminant of $E_\rho/F$), $s_\varepsilon=\#\{\sigma\in\Sigma_F\colon \varepsilon_\sigma=-1\}$, and $\pi^{\frac{\underline{k}}{2}}$  stands for the real number $\pi$ raised to the half-sum of the components of the weigh vector \underline{k}.
\end{proposition}

This result is to be expected for a reasonable notion of periods. Indeed, it is of the same type as those appearing in Shimura's seminal works on periods \cite{Shimura76, Shimura78},  and it aligns with Hida's results \cite{hida-duke} and with Blasius's conjecture \cite{blasius} (see \cite{JST} for some recent results in this direction). The main significance of this result is that it confirms the meaningfulness of the periods we define: they essentially coincide (up to algebraic or transcendental but controlled factors) with the periods considered elsewhere. Thus, it can be interpreted as a validation that our techniques are effective in proving results concerning the relationship between periods and special values of $L$-functions.

For the second and third main results we return to the more general situation where $\pi$ is an automorphic representation of a general  $G$; that is, associated to any quaternion algebra $B/F$. Denote by $\Pi$ the Jacquet-Langlands lift of $\pi$ to $\PGL_2$. Let $E$ be a quadratic étale $F$-algebra such that there exists an embedding of $F$-algebras $E\hookrightarrow B$. Assume that the relative discriminant $D$ of $E/F$ is coprime to the conductor $N$ of $\pi$, and that $\Sigma_B$ coincides with the set of infinite places where $E/F$ is split.

\begin{proposition}[Corollary \ref{AlgebraicityofL}]\label{prop2}
    Let  $\chi:\I_E/E^\times\I_F\rightarrow \{\pm 1\}$ be an anticyclotomic character of conductor $c$ coprime to $N$. Suppose also that $\chi$  satisfies the Heegner-type hypothesis described in Assumption \ref{Assonchi}. Let $\varepsilon$ be the sign vector   of lowest degree  such that $\varepsilon_\sigma=\chi_{\sigma}(-1)$ for all $\sigma \in \Sigma_B$, and let $\alpha\in F^\times$ be any element such that $E=F(\sqrt{\alpha})$. Denote by $r_1^B$ the number of real places in $\Sigma_F\setminus\Sigma_B$. Then, the quantity 
    \begin{equation}\label{eq:quantity is square}
\frac{L(1/2,\Pi,\chi)}{(-1)^{\left(\sum_{\sigma\not\in\Sigma_B}\frac{k_\sigma-2}{2}\right)}\cdot\pi^{r_1^B}\cdot\pi^{\underline k}\cdot(\Omega^\pi_\varepsilon)^{2}\cdot|cd_F|\cdot|D|^{-\frac{1}{2}} \cdot \alpha^{\frac{2-\underline{k}}{2}}}\,
\end{equation}
is a square in $L_\Pi$.
\end{proposition}
This statement is a consequence of the main results of \cite{preprintsanti2}. However, we give a different, more direct proof, based on $(\mathfrak{g},K)$-cohomology techniques rather than on group cohomology as in  \cite{preprintsanti2}.
\begin{remark}
Assume that $G=\PGL_2$ so that, in particular, $E/F$ splits at all archimedean places. Suppose also that $\chi$ is of the form $\chi=\rho\circ{\rm N}_{E/F}$ for some Hecke character $\rho:\I_F/F^\times\rightarrow\{\pm 1\}$. Using Artin formalism and Proposition \ref{prop1},  and denoting by $\psi_E$ the quadratic character associated to $E/F$ we deduce that
\[
\frac{L(1/2,\Pi,\chi)}{\pi^{\underline k}(\Omega_\varepsilon^\Pi)^{2}|cd_F||D|^{\frac{-1}{2}} \alpha^{\frac{2-\underline{k}}{2}}}=\frac{(|D||D_\rho||D_{\rho\psi_E}|)^{\frac{1}{2}}\alpha^{\frac{{\underline k}-2}{2}}}{(-1)^{s_\varepsilon}|c|}\frac{L(1/2,\Pi, \rho)}{|d_F|^{\frac{1}{2}}\pi^{\frac{\underline k}{2}} |D_\rho|^{\frac{1}{2}} i^{s_\varepsilon} \Omega_\varepsilon^\Pi}\frac{L(1/2,\Pi, \rho\psi_E)}{|d_F|^{\frac{1}{2}}\pi^{\frac{\underline k}{2}} |D_{\rho\psi_E}|^{\frac{1}{2}}i^{s_\varepsilon} \Omega_\varepsilon^\Pi}\in L_\Pi,
\]
because $|D|^{\frac{1}{2}}|D_\rho|^{\frac{1}{2}}|D_{\rho\psi_E}|^{\frac{1}{2}}\in\Q^\times$. Therefore, in this situation the fact that the quantity \eqref{eq:quantity is square} belongs to $L_\pi$ can be easily deduced from Proposition \ref{prop1}. But Proposition \ref{prop2} is a stronger result even in this situation, since it shows that it is a square in $L_\pi$.
\end{remark}
The third main result of this note relates the periods attached to a  lowest degree sign vector $\varepsilon$ and its  highest degree opposite $-\varepsilon$ with the special value of the adjoint $L$-function of $\pi$. In the particular case where $F$ is totally real and $\pi$ is of parallel weight $2$, this amounts to the classical Riemann--Hodge period relation for Hilbert modular forms of \cite[Theorem 2.4]{oda}.

\begin{proposition}[Corollary \ref{cor3}]\label{prop3}
    Let $\varepsilon\in \{\pm1\}^{\Sigma_B}$ be a sign vector  of lowest degree and denote by $r_{1,B}$ the number of real places in $\Sigma_B$, by  $r_1^B$ the number of real places in $\Sigma_F\setminus\Sigma_B$, and by $r_2$ the number of complex places. We have that
    \begin{align}\label{eq:prop3}
\frac{L(1,\Pi,{\rm ad})}{\Omega^\pi_\varepsilon\cdot \Omega^\pi_{-\varepsilon}\cdot\pi^{2r_1^B+r_2}\cdot(\pi i)^{r_{1,B}}\cdot\pi^{\underline k} }\quad\text{belongs to}\; L_\Pi^\times.
\end{align}
\end{proposition}

 Hida proved in \cite{hida1} a similar algebraicity result for the adjoint $L$-function for the case where  $B=\GL_2$ over a general base field $F$, and for the case where $B$ is a quaternion algebra over a quadratic number field in \cite{hida2}.\footnote{We have been informed by Hida that he has recently established the result for quaternion algebras $B$ over totally real number fields.} The periods appearing in Hida's results are associated to the base change lift of $\Pi$ to an auxiliary quadratic extension $K/F$, while our periods are associated directly to $\Pi$.

In the remainder of the introduction we illustrate some applications of our formulas. We  show that they provide supporting evidence --occasionally conditional on well-established conjectures-- for the conjectures of Oda and Prasanna--Venkatesh, in the specific case where $\pi$ arises from an elliptic curve over $F$. We begin by giving a formulation of these conjectures in this particular setting.

\subsection{The conjectures of Oda and Prasanna-Venkatesh for elliptic curves}\label{sec: oda and PV} For the remainder of the introduction, we assume that $\Pi$ corresponds to an elliptic curve $A/F$; that is, $L(s, \Pi) = L(s, A)$. This implies, in particular, that  $\underline{k}=(2,\dots,2)$  and that $L_\Pi=\Q$. We fix an invariant differential $\omega_A\in H^0(A,\Omega_A^1)$. For any real place $\sigma\in\Sigma_F$, we will also denote by $\sigma\colon F\hookrightarrow\R$ the corresponding embedding. We can describe $A_{\sigma}:=A\times_{\sigma}\C$ as a complex torus
\[
A_\sigma(\C)\simeq \C/\Lambda_\sigma
\]
is such a way that $\omega_A$ corresponds under this identification to the differential $dz$. Since we are interested in $\Lambda_\sigma\otimes\Q$ rather than $\Lambda_\sigma$, without loss of generality we can, and do, assume that $ \Lambda_\sigma=\Z\Omega_{\sigma,1}+\Z\Omega_{\sigma,2}$ with $\Omega_{\sigma,1}\in\R$ and $\Omega_{\sigma,2}\in \R i$. Let $\tau_\sigma=\Omega_{\sigma,1}/\Omega_{\sigma,2}$ be the period of $A_\sigma$. The following conjecture, formulated by Oda \cite{oda83} in the case where $F$ is totally real, predicts that the geometric period $\tau_\sigma$ can be calculated from the automorphic periods of Jacquet--Langlands lifts of $\Pi$ to quaternion algebras. See \cite{DL} for an inspiring formulation of the conjecture with applications to the computation of algebraic points on elliptic curves, and \cite{GMS} for the more general case of arbitrary signature base field.
\begin{conjecture}[Oda]\label{Odaconj}
    Let $\sigma\in\Sigma_F$ be a real place. Let $B$ be any quaternion algebra that splits at $\sigma$, and let $G$ be the algebraic group associated to $B^\times/F^\times$. Suppose that $\Pi$ admits a Jacquet-Langlands lift $\pi$ to $G$. For $\varepsilon\in \{\pm1\}^{\Sigma_B\setminus\{\sigma\}}$  of lowest or highest degree, denote by $(\varepsilon,+)$ the sign vector of $\Sigma_B$ that coincides with $\varepsilon$ at the places different from $\sigma$, and it is $+1$ at $\sigma$; define analogously $(\varepsilon,-)$. Then,
    \begin{align}\label{eq:oda conj}
    \Omega^\pi_{(\varepsilon,+)}/\Omega^\pi_{(\varepsilon,-)}\equiv\tau_\sigma\mod\Q^\times.
    \end{align}
\end{conjecture}

A different conjecture associated with the periods $\Omega_{\varepsilon}^\Pi$ is a conjecture of Prasanna and Venkatesh \cite{PV}. To state the conjecture, we will need some ingredients: Let $M={\rm Ad}(h^1(A)_\Q)$ be the weight zero adjoint motive over $\Q$ associated with the elliptic curve $A$. We will denote by $\SL_1(B)\subseteq B$ the elements of norm 1, namely, the dual group of $G(F)=B^\times/F^\times$. If we write $\M_2(\bullet)_0$ for the subspace of matrices in $\M_2(\bullet)$ with zero trace, then we have an embedding
\begin{equation}\label{defvarphiBG}
\varphi:M_{B}\stackrel{\kappa}{\simeq}\bigoplus_{\nu:F\hookrightarrow\C}\M_2(\Q)_0\hookrightarrow \hat\mfg\otimes_\Q\bar\Q=B_0\otimes_\Q\bar\Q=\bigoplus_{\nu:F\hookrightarrow\C}\M_2(\bar\Q)_0,
\end{equation}
where $\hat\mfg=B_0=\{b\in B;\;{\rm Tr}(b)=0\}$ is the Lie algebra of $\SL_1(B)$, and $M_{B}$ is the Betti realization of $M$. Notice that $\kappa$ is provided by an identification $\Lambda_{\nu}\simeq\Z^2$, where $A_\nu(\C)\simeq\C/\Lambda_\nu$ for $\nu:F\hookrightarrow\C$. The Weil group of $\R$,  
\[
W_\R=\C^\times\rtimes\langle j\rangle,\qquad j^2=-1,\qquad j^{-1}zj=\bar z,
\]
acts naturally on both $\hat\mfg\otimes_\Q\C$ and $M_B\otimes_\Q\C$. Indeed, $W_\R$ acts on $\hat \mfg\otimes_\Q\C$ via composition of the natural conjugation and the archimedean parameter $W_\R\rightarrow \,^LG$ provided by $\pi$. Moreover, if we consider the Hodge decomposition $M_B\otimes_\Q\C=\bigoplus_{p+q=0}M_B^{pq}$, the action of $z\in \C^\times\subset W_\R$ is given by $z^{p}\bar z^{q}$ on $M_B^{pq}$, while $j$ acts as the complex conjugation involution $C_\infty$. Under the isomorphism $\hat \mfg\otimes_\Q\C\simeq M_B\otimes_\Q\C$ induced by $\varphi$, both actions are conjugated; namely, there exists $\delta\in {\rm Aut}(M_B\otimes_\Q\C)$ such that $\gamma\varphi(m)=\varphi(\delta^{-1}\gamma \delta m)$, for all $\gamma\in W_\R$ and $m\in M_B$. Moreover, we can choose $\delta$ so that it is an isometry with respect to the natural pairing induced by
\[
\langle\;,\;\rangle_\Q:B_0\times B_0\longrightarrow \Q;\qquad \langle b_1,b_2\rangle={\rm Tr}_{F/\Q}{\rm Tr}(b_1\cdot \bar b_2),
\]
where $\overline{(\cdot)}$ denotes the non-trivial conjugation on $B$. Hence, $\varphi$ induces an isomorphism
\[
\kappa_\delta:(M_B\otimes\C)^{W_\R}\longrightarrow (\hat\mfg\otimes_\Q\C)^{W_\R}=:\mathfrak{a};\qquad \kappa_\delta(m)=\varphi(\delta^{-1}m).
\]
On the one hand, Prasanna and Venkatesh define in \cite{PV} an action of $\bigwedge^\ast\mathfrak{a}^{\vee}$ on the automorphic cohomology. To describe such an action, write $\Sigma_{B}=\Sigma_B^\R\sqcup\Sigma_B^\C$, where $\Sigma_B^\R$ (resp. $\Sigma_B^\C$) is the subset of real places (resp. complex places), and notice that $\mathfrak{a}=(\hat\mfg\otimes_\Q\C)^{W_\R}=\bigoplus_{\Sigma_B^\C}\C\hat H_\sigma$, where $\hat H_\sigma$ corresponds to the matrix $\big(\begin{smallmatrix}
    1&\\&-1
\end{smallmatrix}\big)\in \M_2(\C)_0\simeq\hat \mfg_\sigma={\rm Lie}(\widehat{G(F_\sigma)})$. Write $\{\hat H_\sigma^\ast\}\subset\mathfrak{a}^\vee$ for the dual basis. Given $\varepsilon_\C=(\varepsilon_\sigma)\in \{\pm 1\}^{\Sigma_B^\C}$, we define
\[
n_{\varepsilon_\C}:=\#\{\sigma\in \Sigma_B^\C\colon \varepsilon_\sigma=-1\};\qquad \hat H_{\varepsilon_\C}^\ast:=\bigwedge_{\varepsilon_\sigma=-1}\hat H_\sigma^\ast\in \bigwedge^{n_{\varepsilon_\C}}\mathfrak{a}^\vee.
\]
As we will see in Lemma \ref{LemPraVen}, all the Eichler--Shimura morphisms we construct can be recovered from the \emph{lowest-degree} Eichler--Shimura morphisms (those for which $n_{\varepsilon_\C} = 0$), together with the Prasanna--Venkatesh action of $\hat H_{\varepsilon_\C}^\ast$. More precisely, for any $\varepsilon_\R \in \{\pm1\}^{\Sigma_B^\R}$, let $(\varepsilon_\R, \underline{1}) \in \{\pm1\}^{\Sigma_B}$ denote the corresponding lowest-degree sign vector, obtained by extending $\varepsilon_\R$ with ones at the complex places. Then we have 
\begin{equation}\label{P-Vaction}
\hat H_{\varepsilon_\C}^\ast \circ {\rm ES}_{(\varepsilon_\R, \underline{1})} = i^{n_{\varepsilon_\C}} \cdot {\rm ES}_{(\varepsilon_\R, \varepsilon_\C)}.
 \end{equation} 
This describes how the Prasanna--Venkatesh action reconstructs the full $\pi$-isotypic cohomology from the lowest-degree contributions.

On the other hand, one can identify $(M_B\otimes\C)^{W_\R}$ with the extension of Deligne cohomology $H^1_{\cD}(M,\R(1))\otimes_\R\C$. A conjecture of Beilinson predicts that the natural regulator map $r$ between motivic and Deligne cohomologies provides an isomorphism 
\[
r:H^1_{\cM}(M,\Q(1))\rightarrow H^1_{\cD}(M,\R(1));\qquad r\left(H^1_{\cM}(M,\Q(1))\right)\otimes_\Q\R\simeq H^1_{\cD}(M,\R(1)).
\]
In particular, it defines a $\Q$-structure on $(M_B\otimes\C)^{W_\R}$. Prasanna and Venkatesh conjecture in \cite{PV} that the induced $\Q$-structure on $\bigwedge^\ast\mathfrak{a}^\vee$ transferred via $\kappa_\delta$ preserves $\pi$-isotypic components in cohomology with coefficients in $\Q$. Recall that, both in lowest and highest degree, the periods are chosen in such a way that, after applying the Eichler--Shimura maps ${\rm ES}_\varepsilon$ to the normalized eigenform, the resulting cohomology classes become $\Q$-rational once divided by $\Omega_\varepsilon^\pi$. Thus,  the following conjecture is an immediate consequence of the Prasanna--Venkatesh conjectures together with \eqref{P-Vaction}.
\begin{conjecture}\label{PVconj}
We have that $r(H^1_{\cM}(M,\Q(1)))\otimes_\Q\C=\mathfrak{a}$. Moreover, if $\Pi$ admits a Jacquet-Langlands lift $\pi$ to $G$ then, for all $\varepsilon_\R\in \{\pm 1\}^{\Sigma_B^\R}$,
    \[
    \frac{\Omega_{(\varepsilon_\R,\underline 1)}^\pi}{\Omega_{(\varepsilon_\R,-\underline 1)}^\pi}i^{r_2}\kappa_\delta^{-1}\left(\hat H_{-\underline 1}^\ast\right)\in \bigwedge^{{r_2}}r\left(H^1_{\cM}(M,\Q(1))\right)^\vee.
    \]
\end{conjecture}
Some of the applications of our formulas to the conjectures of Oda and Prasanna--Venkatesh rely on assuming Beilinson's conjectures for the motives $A$ and $M$. While the specialization of Beilinson's conjecture to these specific motives is likely well known to experts, we include it in Appendix \ref{ApdxBC} due to lack of a suitable reference. In the next \S, we summarize the relevant statements; namely, the portion of Beilinson's conjecture that we assume at various points, in the specific context of the motives $A$ and $M$.

\subsection{Beilinson's conjecture for $A$ and $M$} We continue with the same notation: $A$ is an elliptic curve over $F$, and $M = \mathrm{Ad}(h^1(A)_\Q)$. We decompose the set $\Sigma_F$ of archimedean places into real and complex ones, writing $\Sigma_F = \Sigma_F^\R \sqcup \Sigma_F^\C$. For any $\sigma \in \Sigma_F$, we fix an embedding $\sigma: F \hookrightarrow \C$ representing it; if $\sigma \in \Sigma_F^\C$, we denote by $\bar\sigma: F \hookrightarrow \C$ its composition with complex conjugation.
Such an abuse of notation of conflating embeddings and places will be used occasionally, but the general rule will be to denote places by $\sigma$ and embeddings by $\nu$.

For any embedding $\nu:F\hookrightarrow\C$ we can describe $A_{\nu}=A\times_{\nu}\C$ as a complex torus $$A_{\nu}(\C)\simeq \C/\Lambda_{\nu}; \ \  \Lambda_\nu=\Z\Omega_{{\nu},1}+\Z\Omega_{{\nu},2},$$ in such a way that the invariant differential $\omega_A$ corresponds to $dz$ under this identification. As in \S\ref{sec: oda and PV}, if $\sigma\in \Sigma_F^\R$ we assume that $\Omega_{\sigma,1}\in\R$ and $ \Omega_{\sigma,2}\in \R i$. Similarly, if  $\sigma\in \Sigma_F^\C$, we assume that $\overline{\Omega_{\sigma,1}}=\Omega_{\bar\sigma,1}$ and $\overline{\Omega_{\sigma,2}}=\Omega_{\bar\sigma,2}$. We put $\tau_\nu=\Omega_{\nu,1}/\Omega_{\nu,2}$.

The following conjectures are deduced in Appendix \S\ref{ApdxBC} from Beilinson's conjectures. The first conjecture can be regarded as an unrefined version of the leading term formula in the Birch and Swinnerton-Dyer conjecture in rank zero. The second conjecture predicts the term $ \bigwedge^{{r_2}}r\left(H^1_{\cM}(M,\Q(1))\right)^\vee$ of Conjecture \ref{PVconj}.
\begin{conjecture}\label{BSDBeilinson}
    Assume that $L(1,A)\neq0$. Then
    \begin{equation*}
    L(1,A)\in|d_F|^{1/2}\prod_{\sigma\in\Sigma_F^\R}\Omega_{\sigma, 1}\prod_{\sigma\in\Sigma_F^\C}{\rm Im}(\Omega_{\sigma, 1}\overline{\Omega_{\sigma,2}})\Q^\times.
    \end{equation*}
    \end{conjecture}

The identifications $\Lambda_\nu\simeq \Z^2$ provided by the periods $\Omega_{\nu,i}$ induce the isomorphism $M_B=\bigoplus_{\nu:F\hookrightarrow\C}\M_2(\Q)_0$  described in \eqref{defvarphiBG}. If we consider the natural monomorphism
\[
\imath:\bigoplus_{\sigma\in \Sigma_F^\C}\M_2(\R)_0\hookrightarrow M_B\otimes \C\simeq \bigoplus_{\nu:F\hookrightarrow\C}\M_2(\C)_0;\qquad (\imath(m_\sigma))_\nu=\left\{\begin{array}{ll}1;&\nu\neq \sigma,\bar\sigma\\m_\sigma;&\nu=\sigma,\bar\sigma,\end{array}\right. 
\]
then it can be deduced from the Hodge decomposition \eqref{HdgDeceq} that $(M_B\otimes \C)^{W_\R}\simeq H^1_{\cD}(M,\R(1))$ is generated by 
\begin{equation}\label{eqHsigma}
H_\sigma=\imath\left(\begin{array}{cc}\frac{-{\rm Re}(\tau_\sigma)}{{\rm Im}(\tau_\sigma)}&\frac{|\tau_\sigma|^2}{{\rm Im}(\tau_\sigma)}\\\frac{-1}{{\rm Im}(\tau_\sigma)}&\frac{{\rm Re}(\tau_\sigma)}{{\rm Im}(\tau_\sigma)}\end{array}\right).
\end{equation}
In particular $\dim(H^1_{\cD}(M,\R(1)))=r_2$.
\begin{conjecture}\label{R2calc} 
We have that $\dim_\Q r\left(H^1_{\cM}(M,\Q(1))\right)=r_2$. Moreover,
\begin{equation*}
  \bigwedge^{r_2}r\left(H^1_{\cM}(M,\Q(1))\right)=L(1,M)(2\pi i)^{-r_1-r_2}\prod_{\sigma\in\Sigma_F^\R}\Omega_{{\sigma},1}^{-1}\Omega_{{\sigma},2}^{-1}\prod_{\sigma\in\Sigma_F^\C}{\rm Im}(\Omega_{\sigma,1}\overline{\Omega_{\sigma,2}})^{-2}\left(\bigwedge_{\sigma\in\Sigma_F^\C} H_\sigma\right)\Q.
\end{equation*}
\end{conjecture}

\subsection{Applications to Oda's conjecture} In this section, we use the formulas of Propositions \ref{prop1}, \ref{prop2}, and \ref{prop3} to provide several pieces of evidence in support of Oda's conjecture. We follow a standard strategy that relates automorphic periods to geometric ones via $L$-functions. To carry this out, we need to ensure the existence of enough quadratic twists of $A$ whose $L$-functions do not vanish at the central point. To this end, we use the results in Appendix B, under the assumption that $N$, the conductor of $A$, is not a square. This assumption could be replaced by the \emph{strong admissibility} condition described in the appendix.

The first result shows that Conjecture \eqref{BSDBeilinson} implies Oda's conjecture for $G=\PGL_2$ and lowest-degree sign vectors $\varepsilon\in\{\pm 1\}^{\Sigma_F\setminus\{\sigma\}}$.

\begin{proposition}\label{Oda1}
Assume that $N$ is not a square and  Conjecture \eqref{BSDBeilinson} holds. For any $\sigma\in\Sigma_F^\R$ and any $\varepsilon\in\{\pm 1\}^{\Sigma_F\setminus\{\sigma\}}$ of lowest degree, we have 
\[
 \Omega_{(\varepsilon,+)}^\Pi/\Omega_{(\varepsilon,-)}^\Pi\equiv \tau_\sigma\mod\Q^\times.
\]
\end{proposition} 
\begin{proof}
Let $\rho:\I_F^\times/F^\times\rightarrow \{\pm 1\}$ be the quadratic Hecke character associated with a quadratic extension $E_\rho=F(\sqrt{\alpha})$. Then the motive $h^1(A)_\Q(\rho)$ equals $h^1(A^\rho)_\Q$, where the elliptic curve $A^\rho$ satisfies
\[
A^\rho_\sigma(\C)=(A^\rho\times_\sigma\C)(\C)\simeq\C/\Lambda_\sigma^\rho,\quad \Lambda_\sigma^\rho=\Z\Omega_{\sigma,1}^\rho+\Z\Omega_{\sigma,2}^\rho;\qquad \Omega^\rho_{\sigma,i}=\sqrt{\sigma(\alpha)}\Omega_{\sigma,\lambda_\sigma i+3\frac{1-\lambda_\sigma}{2}}.
\]
where $\lambda=(\lambda_\sigma)\in \{\pm 1\}^{\Sigma_F}$, the sign vector of $\rho$, is given by $\lambda_\sigma=\rho_\sigma(-1)$. Thus,  by Conjecture  \ref{BSDBeilinson}, when $L(1,A^\rho)\neq0$
    \begin{equation*}\label{resultBSD}
    L(1,A^\rho)\in |d_F|^{1/2}|\alpha|^{1/2}i^{s_\lambda}\prod_{\sigma\in\Sigma_F^\R}\Omega_{\sigma,\frac{3-\lambda_\sigma}{2}}\prod_{\sigma\in\Sigma_F^\C}{\rm Im}(\Omega_{\sigma,1}\overline{\Omega_{\sigma,2}})\Q^\times.
    \end{equation*}
Hence, combining this with Proposition \ref{prop1}, we deduce
\begin{equation}\label{relOmegaOmega}
  \Omega_\lambda^\Pi\equiv\frac{L(1/2,\Pi,\rho)}{|d_F|^{\frac{1}{2}}\pi^{d}|\alpha|^{\frac{1}{2}}i^{s_\lambda}}\equiv\frac{L(1,A^\rho)}{|d_F|^{\frac{1}{2}}\pi^{d}|\alpha|^{\frac{1}{2}}i^{s_\lambda}}\equiv\pi^{-d}\prod_{\sigma\in\Sigma_F^\R}\Omega_{\sigma,\frac{3-\lambda_\sigma}{2}}\prod_{\sigma\in\Sigma_F^\C}{\rm Im}(\Omega_{\sigma,1}\overline{\Omega_{\sigma,2}})\mod\Q^\times.
\end{equation}

Since $N$ is not a square, Proposition \ref{strongadm} ensures the existence of $\rho_1$ and $\rho_2$ with sign vectors $\lambda_1=(\varepsilon,+)$ and $\lambda_2=(\varepsilon,-)$, such that $L(1,A^{\rho_i})\neq0$. Thus, we conclude
\begin{equation*}\label{OdasupBSD}
    \Omega_{(\varepsilon,+)}^\Pi/\Omega_{(\varepsilon,-)}^\Pi\equiv\Omega_{\sigma, 1}/\Omega_{\sigma, 2}=\tau_\sigma\mod\Q^\times,
\end{equation*}
and the result follows. 
\end{proof}

The second result provides evidence for Oda's conjecture in the case of general $G$ and a lowest-degree sign vector, again under the assumption that Conjecture \eqref{BSDBeilinson} holds. In the general case, however, we are not able to prove the full conjecture, but only a weaker statement: namely, that the congruence holds modulo $(\Q^\times)^{\frac{1}{2}}$. The argument is based in the following relation between periods of $\pi$ and its Jacquet--Langlands lift $\Pi$ to $\PGL_2$. 
\begin{lemma}\label{relperiodsPGLG}
  Assume that $N$ is not a square and that $\Pi$ admits a Jacquet--Langlands lift $\pi$ to $G$.  Then, for any $\varepsilon\in\{\pm 1\}^{\Sigma_B}$ of lowest degree and any $\lambda\in\{\pm 1\}^{\Sigma_F\setminus\Sigma_B}$, we have
  \begin{align*}
\Omega_{(\varepsilon,\lambda)}^\Pi\Omega_{(\varepsilon,-\lambda)}^\Pi\equiv(\pi i)^{r_1^B}(\Omega^{\pi}_{\varepsilon})^{2}\mod \Q^\times,\qquad\qquad \mbox{ and }\qquad\qquad
     \frac{\Omega_{(\varepsilon,\lambda)}^\Pi}{\Omega_{(-\varepsilon,\lambda)}^\Pi}\equiv\frac{\Omega^{\pi}_{\varepsilon}}{\Omega^{\pi}_{-\varepsilon}}\mod \Q^\times.
  \end{align*}
\end{lemma}
\begin{proof}
  By Proposition \ref{prop:apendix second} there exist quadratic Hecke characters $\rho_1,\rho_2:\I_F^\times/F^\times\rightarrow \{\pm 1\}$ with sign vectors $(\varepsilon,\lambda)$ and $(\varepsilon,-\lambda)$, respectively, of conductor coprime to $N$, and such that $L(1, \Pi, {\rho_i})\neq 0$. Denote by $E/F$ the quadratic extension associated to $\rho_1\cdot\rho_2$, which admits an embedding into $B$. Observe that, by Remark \ref{rk:ass} plus the fact that $L(1/2,\Pi, \rho_1\circ \mathrm{N}_{E/F})\neq 0$, the character $\rho_1\circ \mathrm{N}_{E/F}$ satisfies Assumption \ref{Assonchi}.
By Propositions \ref{prop1}, \ref{prop2} and Artin formalism,
\begin{equation*}
 \Omega_{(\varepsilon,\lambda)}^\Pi\Omega_{(\varepsilon,-\lambda)}^\Pi\equiv \frac{L(1/2,\Pi,\rho_1)L(1/2,\Pi,\rho_2)}{\pi^{2d} i^{r_1^B} |D_{\rho_1}D_{\rho_2}|^{\frac{1}{2}}}\equiv \frac{L(1/2,\Pi, \rho_1\circ \mathrm{N}_{E/F})}{\pi^{2d} i^{r_1^B} |D_{\rho_1\rho_2}|^{\frac{1}{2}}}\equiv(\pi i)^{r_1^B}(\Omega^{\pi}_{\varepsilon})^{2}\mod \Q^\times.
\end{equation*}
On the other side, applying  Proposition \ref{prop3} twice, we obtain
\begin{equation*}
 \Omega_{(\varepsilon,-\lambda)}^\Pi\Omega_{(-\varepsilon,\lambda)}^\Pi\equiv \frac{L(1,\Pi,{\rm ad})}{\pi^{2r_1^B+r_2}(\pi i)^{r_{1,B}}\pi^{2d} } \equiv(\pi i)^{r_1^B}\Omega^{\pi}_{\varepsilon}\Omega^{\pi}_{-\varepsilon}\mod \Q^\times.
\end{equation*}
Dividing both sides of the two equalities yields the desired result.
\end{proof}

\begin{proposition}\label{Oda2}
 Assume that $N$ is not a square and that $\Pi$ admits a Jacquet--Langlands lift $\pi$ to $G$. If  Conjecture \ref{BSDBeilinson} holds then, for any $\sigma\in\Sigma_B^\R$ and any $\varepsilon\in\{\pm 1\}^{\Sigma_B\setminus\{\sigma\}}$ of lowest degree, we have 
\[
 \Omega^\pi_{(\varepsilon,+)}/\Omega^\pi_{(\varepsilon,-)}\equiv \tau_\sigma\mod(\Q^\times)^{\frac{1}{2}}.
\]
\end{proposition} 
\begin{proof}
For any $\lambda\in\{\pm 1\}^{\Sigma_B}$ and $\gamma\in\{\pm 1\}^{\Sigma_F\setminus\Sigma_B}$, we apply Lemma \ref{relperiodsPGLG} and Equation \eqref{relOmegaOmega} to obtain
\begin{equation}\label{RelOmegaGOmega}
    (\Omega_\lambda^\pi)^2\equiv (\pi i)^{-r_1^B}\Omega_{(\lambda,\gamma)}^\Pi\Omega_{(\lambda,-\gamma)}^\Pi\equiv\prod_{\sigma\not\in\Sigma_F^\R\setminus\Sigma_B^\R}\frac{\Omega_{1,\sigma}\Omega_{2,\sigma}}{i\pi^{3}}\prod_{\sigma\in\Sigma_B^\R}\frac{\Omega^2_{\frac{3-\varepsilon_\sigma}{2},\sigma}}{\pi^2}\prod_{\sigma\in\Sigma_F^\C}\frac{{\rm Im}(\Omega_{1,\sigma}\overline{\Omega_{2,\sigma}})^2}{\pi^4}\mod \Q^\times.
\end{equation}
Thus, considering $\lambda=(\varepsilon,+)$ or $\epsilon=(\varepsilon,-)$, and dividing the corresponding right-hand-sides, we obtain the desired result.
\end{proof}

So far, we have provided evidence for Oda's conjecture only in the case of lowest-degree sign vectors. We now show that Prasanna--Venkatesh's conjecture implies the following: if Oda's conjecture holds for lowest degree sign vectors, then it holds for highest degree sign vectors.

More precisely, let $\sigma\in\Sigma_B^\R$ be a real place where $B$ splits and let   $(\varepsilon_\R,\underline 1) \in\{\pm 1\}^{\Sigma_B\setminus\{\sigma\}}$ be a lowest degree sign vector at the places different from $\sigma$, where $\varepsilon_\R\in\{\pm 1\}^{\Sigma_B^\R\setminus\{\sigma\}}$.  Denote by $((\varepsilon_\R,\underline{1}),\pm)$ the extensions of $(\varepsilon_\R,\underline 1)$ by either $+$ or $-$ at $\sigma$, and by $((\varepsilon_\R,-\underline 1),\pm)$ the associated highest degree sign vector.
\begin{proposition}
 Assume that Conjecture \ref{PVconj} holds and that
  $
\Omega^\pi_{((\varepsilon_\R,\underline 1),+)}/\Omega^\pi_{((\varepsilon_\R,\underline 1),-)}\equiv \tau_\sigma\mod \Q^\times.
    $
  Then
\[
\Omega^\pi_{  ((\varepsilon_\R,-\underline 1),+)}/\Omega^\pi_{  ((\varepsilon_\R,-\underline 1),-)}\equiv \tau_\sigma\mod \Q^\times.
\]
\end{proposition}
\begin{proof}
By assumption (Conjecture \ref{PVconj}), we have that  $\bigwedge^{r_2}r\left(H^1_{\cM}(M,\Q(1))\right)^\vee\simeq \Q$. Moreover, we will see in Appendix \ref{ApdxBC} that $\{\kappa_\delta^{-1}\left(\hat H_{-\underline 1}^\ast\right)\}$ defines a basis for $\bigwedge^{r_2}\mathfrak{a}^\vee\simeq \C$. Since
\[
\Omega^\pi_{((\varepsilon_\R,\underline 1),\pm)}/\Omega^\pi_{  ((\varepsilon_\R,-\underline 1),\pm)} i^{-r_2}\kappa_\delta^{-1}\left(\hat H_{{-\underline{1}}}^\ast\right)\in \bigwedge^{{r_2}}r\left(H^1_{\cM}(M,\Q(1))\right)^\vee,
 \]
 we deduce that $\Omega_{((\varepsilon_\R,\underline 1),+)}^\pi/\Omega_{{((\varepsilon_\R,-\underline 1),+)}}^\pi\equiv  \Omega_{((\varepsilon_\R,\underline 1),-)}^\pi/\Omega_{{((\varepsilon_\R,-\underline 1),-)}}^\pi\mod\Q^\times$ and the result follows.
 \end{proof}

Under the assumption that $F$ is totally real, we can prove Oda's conjecture in certain cases and give evidence for it in others, without invoking Conjecture \ref{BSDBeilinson}. More precisely, this applies when $\pi$ admits a Jacquet--Langlands lift to a quaternion algebra $B_0$ that splits at a single archimedean  place. In such cases, there exists a morphism $X_{B_0}\rightarrow A$ from the Shimura curve associated with $B_0$ to the elliptic curve $A$. Since $F$ is totally real, our Eichler--Shimura maps $\mathrm{ES}_\pm$ coincide with the classical ones given by integration of differential forms (see \cite{ESsanti}). This implies that Oda's conjecture holds for $G_0$, the multiplicative group of $B_0$, for geometric reasons. Finally, we use Propositions \ref{prop1}, \ref{prop2}, and \ref{prop3} to transfer the result to $\PGL_2$ and any general $G$. The precise formulation is given in the following statement.
\begin{proposition}\label{oldOda1}
Assume that $F$ is totally real, $N$ is not a square and $\Pi$ admits a Jacquet-Langlands lift  to a quaternion algebra that splits at a single archimedean place.  For any $\sigma\in\Sigma_F^\R$ and any $\lambda\in\{\pm 1\}^{\Sigma_F\setminus\{\sigma\}}$, we have 
\begin{equation*}
 \Omega^\Pi_{(\lambda,+)}/\Omega^\Pi_{(\lambda,-)}\equiv \tau_\sigma\mod\Q^\times.
\end{equation*}
If, moreover, $\Pi$ admits a Jacquet-Langlands lift $\pi$  to $G$ then, for any $\sigma\in\Sigma_B^\R$ and any $\varepsilon\in\{\pm 1\}^{\Sigma_B\setminus\{\sigma\}}$, we have 
\begin{equation}\label{eq:oda}
 \Omega^\pi_{(\varepsilon,+)}/\Omega^\pi_{(\varepsilon,-)}\equiv \tau_\sigma\mod(\Q^\times)^{\frac{1}{2}}.
\end{equation}
\end{proposition}
\begin{proof}
Fix $\sigma\in\Sigma_B$. Since $\Pi$ admits a Jacquet-Langlands  lift  to a quaternion algebra that splits at a single archimedean place,  $\Pi$ admits a Jacquet-Langlands  lift $\pi_0$ to a quaternion algebra $B_0$ that only splits at $\sigma$. Let $G_0$ be the algebraic group associated with $B_0^\times /F^\times$. Since there exists a modular parametrization $X_{B_0}\ra A$, applying \cite{ESsanti} we have that $\Omega^{\pi_0}_{+}/\Omega^{\pi_0}_{-}\equiv \tau_\sigma\;{\rm mod}\;\Q^\times$.
By Lemma \ref{relperiodsPGLG}, for any $\lambda\in\{\pm 1\}^{\Sigma_F\setminus\{\sigma\}}$,
\begin{equation*}
  \frac{\Omega_{(\lambda,+)}^\Pi}{\Omega_{(\lambda,-)}^\Pi}\equiv\frac{\Omega^{\pi_0}_+}{\Omega^{\pi_0}_-}\equiv \tau_\sigma \mod\Q^\times,
\end{equation*}
which proves the first part of the proposition. For the second part, we apply again Lemma \ref{relperiodsPGLG} to obtain
\begin{equation*}
\frac{(\Omega^{\pi}_{(\varepsilon,+)})^2}{(\Omega^{\pi}_{(\varepsilon,-)})^2}\equiv\frac{\Omega_{(\varepsilon,+,\gamma)}^\Pi\Omega_{(\varepsilon,+,-\gamma)}^\Pi}{\Omega_{(\varepsilon,-,\gamma)}^\Pi\Omega_{(\varepsilon,-,-\gamma)}^\Pi}\equiv \tau_\sigma^2\mod\Q^\times,
\end{equation*}
for any $\gamma\in\Sigma_F\setminus\Sigma_B$, and the result follows.
\end{proof}
If $\Pi$ does not admit a Jacquet–Langlands lift to a quaternion algebra that splits at a single archimedean place, one can still prove that the Hodge structure defined by the periods $\Omega_\varepsilon^\Pi$  is the tensor product of the Hodge structures of certain elliptic curves defined over $\C$. In terms of periods, this  can be formulated as follows.
\begin{proposition}\label{oldOda2}
Assume that $F$ is totally real and $N$ is not a square. Then, for any $\sigma\in\Sigma_F$, the $\Q^\times$-equivalence class of the quotient $ \Omega^\Pi_{(\lambda,+)}/\Omega^\Pi_{(\lambda,-)}$ is independent of $\lambda\in\{\pm 1\}^{\Sigma_F\setminus\{\sigma\}}$.
\end{proposition}
\begin{proof}
If $\#\Sigma_F$ is odd, then $\Pi$ admits a Jacquet–Langlands lift to a quaternion algebra that splits at a single archimedean place, and the result follows from Proposition \ref{oldOda1}. If $\#\Sigma_F$ is even, then $\Pi$  admits a Jacquet–Langlands lift $\pi$ to a quaternion algebra $B$  that splits at two archimedean places $\Sigma_{B}=\{\sigma,\tau\}$. We write $\lambda=(\gamma,s)$, where $s\in\{\pm 1\}$ and $\gamma\in \{\pm 1\}^{\Sigma_F\setminus\{\sigma,\tau\}}$. By Lemma  \ref{relperiodsPGLG}, 
\[
 \frac{\Omega^\Pi_{(\lambda,+)}}{\Omega^\Pi_{(\lambda,-)}}= \frac{\Omega^\Pi_{(\gamma,s,+)}}{\Omega^\Pi_{(\gamma,-s,-)}} \frac{\Omega^\Pi_{(\gamma,-s,-)}}{\Omega^\Pi_{(\gamma,s,-)}}\equiv  \frac{\Omega^{\pi}_{(s,+)}}{\Omega^\pi_{(-s,-)}} \frac{\Omega^\pi_{(-s,-)}}{\Omega^\pi_{(s,-)}}= \frac{\Omega^{\pi}_{(s,+)}}{\Omega^\pi_{(s,-)}}\mod\Q^\times.
\]
Thus, it only remains to check that $\Omega^{\pi}_{(+,+)}/\Omega^\pi_{(+,-)}\equiv \Omega^{\pi}_{(-,+)}/\Omega^\pi_{(-,-)}\mod\Q^\times$, which follows directly from Proposition \ref{prop3}.
\end{proof}

\begin{remark}

  The results stated in \cite{oda83} are similar to Propositions \ref{oldOda1} and \ref{oldOda2}, though under the assumptions that $\pi$ is strongly admissible and that $F$  has narrow class number one. In fact, display \eqref{eq:oda} is claimed in \cite{oda83} even modulo $\mathbb{Q}^\times $. However, to the best of our knowledge, the proofs provided in \cite{oda83} appear to be incomplete. For this reason, we believe that the results presented here have independent value, even if some may not seem as strong as those announced in \cite{oda83}. In any case, when $F$ has narrow class number greater than one, Propositions \ref{oldOda1} and \ref{oldOda2} are not covered by \cite{oda83}
\end{remark}

\begin{remark}
 Relations such as \eqref{relOmegaOmega} and \eqref{RelOmegaGOmega}, along with Proposition \ref{oldOda2}, also follow from conjectures of Shimura \cite{sh83}, which predict a factorization of the periods in terms of invariants indexed by the archimedean places. These conjectures were established up to algebraic factors by Yoshida in the case where $F$ is totally real, as shown in \cite{yoshida} and \cite{yoshida-duke}. Hida has informed us that a proof modulo the minimal field of rationality is achievable and is likely to appear in his forthcoming work.
\end{remark}
\subsection{Applications to Prasanna--Venkatesh's conjecture}
In this final part of the introduction, we explain how one of our main results, Proposition \ref{prop3}, provides evidence for Conjecture \ref{PVconj}. To this end, we assume that Beilinson conjectures for $h^1(A)$ and
${\rm Ad}(h^1(A)_\Q$ hold, and we will also need the previously used notion of strong admissibility in order to use Oda's conjecture in lowest degree.
\begin{proposition}
Assume that $N$ is not a square and $\Pi$ admits a Jacquet-Langlands lift $\pi$ to $G$. If Conjectures \ref{BSDBeilinson} and \ref{R2calc} hold, and if Conjecture \ref{Odaconj} holds for lowest degree sign vectors, then for any $\varepsilon_\R\in\{\pm 1\}^{\Sigma_B^\R}$,
\[
 \frac{\Omega_{(\varepsilon_\R,\underline 1)}^\pi}{\Omega_{(\varepsilon_\R,\underline{-1})}^\pi}i^{-r_2}\kappa_\delta^{-1}\left(\hat H_{{\underline -1}}^\ast\right)\in \bigwedge^{r_2}r\left(H^1_{\cM}(M,\Q(1))\right)^\vee.
\]
\end{proposition} 
\begin{proof}
 Let us consider the motive $M={\rm Ad}(h^1(A)_\Q)$. On the one hand, by \eqref{HdgDeceq} in the Appendix:
\begin{equation*}
  \begin{small}
    M_B^{0,0}=\bigoplus_{{\nu}:F\hookrightarrow\C}\C\left(\begin{array}{cc}-{\rm Re}(\tau_\nu)&|\tau_\nu|^2\\-1&{\rm Re}(\tau_\nu)\end{array}\right);\,  M_B^{1,-1}=\bigoplus_{{\nu}:F\hookrightarrow\C}\C\left(\begin{array}{cc}1&-\tau_\nu\\\tau_\nu^{-1}&-1\end{array}\right); \,  M_B^{-1,1}=\bigoplus_{{\nu}:F\hookrightarrow\C}\C\left(\begin{array}{cc}1&-\bar\tau_\nu\\\bar\tau_\nu^{-1}&-1\end{array}\right).\end{small}
    \end{equation*}
On the other hand, associated to $h^1(A)$, we have a Weil representation 
\begin{equation*}\label{Weilreprho}
    \begin{small}
\rho:W_\R=\C^\times\rtimes\langle j\rangle\longrightarrow \;^LG=\left(\prod_{\sigma\mid\infty}\prod_{\nu\mid\sigma}\SL_2( F_\sigma)\right)\rtimes\Gal(\C/\R);\ \ \rho(z\in \C^\times)=\left(\left(\begin{array}{cc}
   \nu^{-1}\left(\frac{z}{\bar z}\right)^{1/2}  &  \\
     & \nu^{-1}\left(\frac{\bar z}{ z}\right)^{1/2}
\end{array}\right)_{\nu}\right),  \end{small}
\end{equation*}
where $\Gal(\C/\R)$ acts trivially on $\SL_2(\bar F_\sigma)$, if $F_\sigma=\R$, and switches the two components of $\prod_{\nu\mid\sigma}\SL_2(\bar F_\sigma)$ if $F_\sigma=\C$. Finally, if we denote $c\in\Gal(\C/\R)$ the complex conjugation element we have that
\[
\rho(j)=\left(\left(\omega_{\nu}\right)_{\nu},c\right);\qquad w_{\sigma}=\left(\begin{array}{cc}
     & 1 \\
    -1 & 
\end{array}\right);\quad F_\sigma=\R;\quad w_{\nu}=\left(\begin{array}{cc}
    i &  \\
     & -i
\end{array}\right);\quad \nu\mid\sigma,\;F_\sigma=\C.
\]
Let us consider the isomorphism $\varphi:M_B\otimes_\Q\C\stackrel{\simeq}{\rightarrow}\hat\mfg\otimes_\Q\C$ induced by \eqref{defvarphiBG}, and for any $t\in\C^\times$ write $\delta_{\nu,t}=\big(\begin{smallmatrix}t\bar\tau_\nu&\tau_\nu\\t&1\end{smallmatrix}\big)$. An easy computation shows that 
\[
\delta_{\nu,t}\big(\begin{smallmatrix}0&1\\0&0\end{smallmatrix}\big)\delta_{\nu,t}^{-1}=\frac{-t\bar\tau_\nu}{\bar\tau_\nu-\tau_\nu}\big(\begin{smallmatrix}1&-\bar\tau_\nu\\\bar\tau_\nu^{-1}&-1\end{smallmatrix}\big);\qquad \delta_{\nu,t}\big(\begin{smallmatrix}0&0\\1&0\end{smallmatrix}\big)\delta_{\nu,t}^{-1}=\frac{t^{-1}\tau_\nu}{\bar\tau_\nu-\tau_\nu}\big(\begin{smallmatrix}1&-\tau_\nu\\\tau_\nu^{-1}&-1\end{smallmatrix}\big);\qquad \delta_{\nu,t}\big(\begin{smallmatrix}1&0\\0&-1\end{smallmatrix}\big)\delta_{\nu,t}^{-1}=H_\sigma,
\]
where $H_\sigma$ was defined in \eqref{eqHsigma}.
Moreover, the subspaces of $\hat\mfg\otimes_\Q\C$ where $z\in\C^\times$ acts as $1$, $z/\bar z$ and $\bar z/z$ are, respectively, 
\[
\bigoplus_{{\nu}:F\hookrightarrow\C}\C\left(\begin{smallmatrix}1&0\\0&-1\end{smallmatrix}\right),\qquad\bigoplus_{\sigma\in\Sigma_F^\R}\C\left(\begin{smallmatrix}0&0\\1&0\end{smallmatrix}\right)\oplus\bigoplus_{\sigma\in\Sigma_F^\C}\C\left(\begin{smallmatrix}0&0\\1&0\end{smallmatrix}\right)\oplus\C\left(\begin{smallmatrix}0&1\\0&0\end{smallmatrix}\right),\qquad\bigoplus_{\sigma\in\Sigma_F^\R}\C\left(\begin{smallmatrix}0&1\\0&0\end{smallmatrix}\right)\oplus\bigoplus_{\sigma\in\Sigma_F^\C}\C\left(\begin{smallmatrix}0&1\\0&0\end{smallmatrix}\right)\oplus\C\left(\begin{smallmatrix}0&0\\1&0\end{smallmatrix}\right).
\]
By \eqref{HdgDeceq}, the unique isometries $\delta_t\in {\rm Aut}(M_B\otimes_\Q\C)$ that satisfy $z\varphi(m)=\varphi(\delta_t^{-1}z \delta_tm)$, for all $z\in \C^\times\subset W_\R$ and $m\in M_B$, are induced by conjugating by $\delta_{\sigma,t_\nu}^{-1}$ at the component indexed by $\nu\mid\sigma$, for any choice of $t=(t_\nu)_\nu\in(\C^\times)^{[F:\Q]}$.

Moreover, the induced isomorphism
\[
\kappa_\delta:(M_B\otimes\C)^{W_\R}\longrightarrow (\hat\mfg\otimes_\Q\C)^{W_\R};\qquad \kappa_\delta(m)=\varphi(\delta_t^{-1}m),
\]
is independent of $t$. By the above computations, $\kappa_\delta$ maps $H_\sigma\in (M_B\otimes\C)^{W_\R}$ to $\hat H_\sigma\in (\hat\mfg\otimes_\Q\C)^{W_\R}$ (corresponding to $\left(\big(\begin{smallmatrix}1&0\\0&-1\end{smallmatrix}\big),\big(\begin{smallmatrix}1&0\\0&-1\end{smallmatrix}\big)\right)$ at the place $\sigma\in \Sigma_F^\C$). Hence, by Conjecture \ref{R2calc},
\[
 \bigwedge^{r_2}r\left(H^1_{\cM}(M,\Q(1))\right)^{\vee}=L(1,\Pi,{\rm ad})^{-1}(2\pi i)^{r_1+r_2}\prod_{\sigma\in\Sigma_F^\R}\Omega_{{\sigma},1}\Omega_{{\sigma},2}\prod_{\sigma\in\Sigma_F^\C}{\rm Im}(\Omega_{\sigma,1}\overline{\Omega_{\sigma,2}})^{2}\kappa_\delta^{-1}\left(\hat H_{{\underline -1}}^\ast\right)\Q.
\]
Applying Conjecture \ref{Odaconj} in lowest degree inductively, we obtain
\[
\Omega^\pi_{(\varepsilon_\R,\underline 1)}/\Omega^\pi_{(-\varepsilon_\R,\underline 1)}\equiv \prod_{{\varepsilon_{\R,\sigma}=1}}\tau_{{\sigma}}\prod_{\varepsilon_{\R,\sigma}=-1}\tau_{{\sigma}}^{-1}=\prod_{\sigma\in\Sigma_B^\R}\tau_{{\sigma}}^{\varepsilon_{\R,\sigma}}\mod \Q^\times.
\]
and, by Equation \eqref{RelOmegaGOmega},
\[
(\Omega^\pi_{(-\varepsilon_\R,\underline{1})})^2\equiv i^{r_1^B}\pi^{-r_1^B-2d}\prod_{\sigma\in\Sigma_F^\R}\Omega_{1,\sigma}\Omega_{2,\sigma}\prod_{\sigma\in\Sigma_B^\R}\tau_\sigma^{-\varepsilon_{\R,\sigma}}\prod_{\sigma\in\Sigma_F^\C}{\rm Im}(\Omega_{1,\sigma}\overline{\Omega_{2,\sigma}})^2\mod\Q^\times
\]
Thus, by Proposition \ref{prop3},
\begin{eqnarray*}
\frac{\Omega_{(\varepsilon_\R,\underline 1)}^\pi}{\Omega_{(\varepsilon_\R,\underline{-1})}^\pi}i^{-r_2}\kappa_\delta^{-1}\left(\hat H_{{\underline -1}}^\ast\right)&=&\frac{\Omega_{(\varepsilon_\R,\underline 1)}^\pi}{\Omega_{(-\varepsilon_\R,\underline{1})}^\pi}\frac{(\Omega_{(-\varepsilon_\R,\underline{1})}^\pi)^2}{\Omega_{(\varepsilon_\R,\underline{-1})}^\pi\Omega_{(-\varepsilon_\R,\underline{1})}^\pi}i^{-r_2}\kappa_\delta^{-1}\left(\hat H_{{\underline -1}}^\ast\right)\\
&\equiv& L(1,\Pi,{\rm ad})^{-1}(\pi i)^{r_1+r_2}\prod_{\sigma\in\Sigma_F^\R}\Omega_{1,\sigma}\Omega_{2,\sigma}\prod_{\sigma\in\Sigma_F^\C}{\rm Im}(\Omega_{1,\sigma}\overline{\Omega_{2,\sigma}})^2\kappa_\delta^{-1}\left(\hat H_{{\underline -1}}^\ast\right)\mod \Q^\times,
\end{eqnarray*}
and the result follows.
\end{proof}
Since, by Proposition \ref{Oda1}, in case $G=\PGL_2$ Oda's conjecture for lowest degree follows from Conjectures  \ref{BSDBeilinson} and \ref{R2calc}, we have the following corollary.
\begin{corollary}
Assume that $N$ is not a square and Conjectures \ref{BSDBeilinson} and \ref{R2calc} hold. Then, for any $\varepsilon_\R\in\{\pm 1\}^{\Sigma_F^\R}$,
\[
 \frac{\Omega_{(\varepsilon_\R,\underline 1)}^\Pi}{\Omega_{(\varepsilon_\R,\underline{-1})}^\Pi}i^{-r_2}\kappa_\delta^{-1}\left(\hat H_{{\underline -1}}^\ast\right)\in \bigwedge^{r_2}r\left(H^1_{\cM}(M,\Q(1))\right)^\vee.
\]
\end{corollary} 

The rest of the article is devoted to prove the main results stated in \S\ref{sec:main results}. Section \S\ref{sec: preliminaries} sets up the necessary notations and conventions for Haar measures, Gauss sums, and finite dimensional representations. Sections \ref{LocArchRep} and \ref{fundClasses}   are also preparatory: we develop the theory of local archimedean representations and we define certain fundamental classes. Finally, Section \ref{sec: global formulas} constitutes the technical core of the article: we define the Eichler--Shimura morphisms and the automorphic periods, and we prove the main global formulas. In Appendix \ref{ApdxBC} we derive Conjectures \ref{BSDBeilinson} and \ref{R2calc} from Beilinson's conjecture, and in Appendix \ref{sec:apendix B} we show that a theorem of Waldspurger guarantees the existence of sufficiently many non-vanishing twists of the automorphic representations we consider.

\subsubsection*{Acknowledgments} We are grateful to Henri Darmon for valuable discussions during the early stages of this project. We also thank Haruzo Hida for drawing our attention to relevant results in the literature on periods, which are closely related to the results presented in this work.  This work is partially supported by grants PID2022-137605NB-I00 and 2021 SGR 01468, and also by the María de Maeztu Program CEX2020-001084-M.
\section{Preliminaries}\label{sec: preliminaries}

In this section we set up some of the notation that will be in force throughout the article and we present some preliminary material on Haar measures, Gauss sums, and finite dimensional representations that we will use later on.

Throughout the article, we fix an algebraic closure $\bar\Q$ of $\Q$ inside $\C$.
We write $\hat\Z$ for the profinite completion of $\Z$ and if $R$ is a ring we put $\hat R:=R\otimes\hat\Z$. For a number field $F$, we denote by $\cO_F$ its ring of integers and by $\A_F$ and $\A_F^\infty:=\hat\cO_F\otimes\Q$ the rings of adeles and finite adeles, respectively. Similarly, $\I_F$ and $\I_F^\infty$ denote the ideles and finite ideles. For any place $v$ of $F$, we write $F_v$ for the completion of $F$ at $v$. If $v$ is non-archimedean, namely $v\mid p$ for some prime number $p$, we denote by $\cO_{F,v}$ the integer ring, $\kappa_v$ the residue field, $q_v:=\#\kappa_v$, ${\rm ord}_v(\cdot)$ the valuation, $|x|_v = q_v^{-\mathrm{ord}_v(x)}$ the normalized absolute value, $\varpi_v$ a uniformizer, and $d_{F_v}=\delta_v\cO_{F_v}$ the different over $\Q_p$. We also denote by $\Sigma_F$ the set of archimedean places of $F$. We will frequently use the letter $\sigma$ to denote an archimedean place and for an embedding $\nu \colon F\hookrightarrow \C$ we will write $\nu \mid \sigma$ to indicate that $\nu$ belongs to the class of $\sigma$. Sometimes we will use $\sigma\mid\infty$ to denote $\sigma\in\Sigma_F$.

In all the text, $B$ will denote a quaternion algebra over $F$ and $\Sigma_B$ the set of infinite places of $F$ that split in $B$. We denote by $G$ the algebraic group over $F$ associated with the group of units of $B$ modulo scalars; that is, for any $F$-algebra $R$ we have that
\[
G(R)=(B\otimes_{F}R)^\times/R^\times.
\]

Let $E$ be an étale $F$-algebra such that $[E\colon F]=2$ and there exists an embedding of $F$-algebras $ E\hookrightarrow B$. Observe that $E$ is either a quadratic field extension of $F$ or $E=F\times F$. We fix from now on one such embedding and we use it to identify $E$ with a subalgebra of $B$. We denote by $T$ the algebraic subgroup of $G$ such that for any $F$-algebra $R$
\[
T(R):=(E \otimes_{F}R)^\times/R^\times.
\]

\subsection{Haar Measures}\label{haarmeasures}
For any number field $F$, let us consider the additive character $\psi:\A_F/F\rightarrow\R$ defined as
\begin{equation}\label{defpsiad}
    \psi=\prod_v\psi_v,\qquad  \psi_v(a)=\left\{\begin{array}{lc}
       e^{2\pi i a},  &\text{ if } F_v=\R  \\
        e^{4\pi i{\rm Re}(a)}, &\text{ if } F_v=\C\\
        e^{-2\pi i[{\rm Tr}_{F_v/\Q_p}(a)]}&\text{ if } v\mid p,
    \end{array}\right.
\end{equation}
where $[\cdot]:\Q_p \rightarrow\Q$ is the map that sends $x\in\Q_p$ to its $p$-adic fractional part.
Let $dx_v$ be the Haar measure of $F_v$ normalized so that it is self-dual with respect to $\psi_v$; namely, $\hat{\hat\phi}(x_v)=\phi(-x_v)$, where $\hat\phi$ is the Fourier transform $\hat\phi(y_v)=\int_{F_v}\phi(x_v)\psi_v(x_vy_v)dx_v$.
Notice that if $v$ is archimedean, $dx_v$ is $[F_v:\R]$ times the usual Lebesgue measure; and if $v$ is non-archimedean, then $dx_v$ is the Haar measure satisfying ${\rm vol}(\cO_{F,v})=|d_{F_v}|_v^{1/2}$. Define $$d^\times x_v=\zeta_v(1)|x_v|_v^{-1}dx_v,$$
where 
$$
\zeta_v(s)=\begin{cases}
			(1-q_v^{-s})^{-1}, & \text{if $v\nmid\infty$}\\
         \pi^{-s/2}\Gamma(s/2)   , & \text{if $F_v=\R$}\\
         2(2\pi)^{-s}\Gamma(s), & \text{if $F_v=\C$}.
		 \end{cases}
$$
One easily checks that if $v$ is non-archimedean then ${\rm vol}(\cO_{F_v}^\times)=|d_{F_v}|_v^{1/2}$. The product of $d^\times x_v$ over all places provides a Tamagawa measure $d^\times x$ on $\A_F^\times/F^\times$. In fact, such Haar measure satisfies
\[
{\rm Res}_{s=1}\int_{x\in\A_F^\times/F^\times,\;|s|\leq 1}|x|^{s-1}d^\times x={\rm Res}_{s=1}\Lambda_F(s),
\]
where $\Lambda_F(s)=\zeta_F(s)\prod_{v\mid\infty}\zeta_{v}(s)$ is the completed zeta function of $F$. This implies that, if we choose $d^\times t$ to be the quotient measure for $T(\A_F)/T(F)=\A_E^\times/\A_F^\times E^\times$, then one has that ${\rm vol}(T(\A_F)/T(F))=2L(1,\psi_T)$, where $\psi_T$ is the quadratic character associated to the extension $E/F$.

Let us consider the Haar measure $dg_v$ of $B_v:=B\otimes_F F_v$ which is self-dual with respect to $\psi_v$, namely,
\[
\hat{\hat\phi}(g_v)=\phi(-g_v),\quad\mbox{ where $\hat\phi$ is the Fourier transform }\quad\hat\phi(a_v)=\int_{B_v}\phi(g_v)\psi_v(a_v\bar g_v+g_v\bar a_v)dg_v,
\]
and $(g_v\mapsto\bar g_v)$ is the usual involution on $B_v$. We define similarly as above
$d^\times g_v=\zeta_v(1)|g_v\bar g_v|_v^{-2}dg_v$. The product of such $d^\times g_v$ over all places provides a Tamagawa measure for $G$ satisfying ${\rm vol}(G(F)\backslash G(\A_F))=2$ and (see \cite[Lemma 3.5]{CST})
\[
\begin{array}{ll}
   {\rm vol}(\PGL_2(\cO_{F_v}))=\zeta_v(2)^{-1}|d_{F_v}|_v^{3/2},  &\text{ if }v\nmid\infty\text{ and } B_v=\M_2(F_v),  \\
   {\rm vol}(\cO_{B_v}^\times/\cO_{F_v}^\times)=\zeta_v(2)^{-1}(q_v-1)^{-1}|d_{F_v}|_v^{3/2},  & \text{ if }v\nmid\infty \text{ and } B_v\neq\M_2(F_v),\\
   {\rm vol}(B_\sigma^\times/F_\sigma^\times)=2\pi^2,&\text{ if }\sigma\mid\infty\text{ and } B_\sigma\neq\M_2(F_\sigma).
\end{array}
\]
If $\sigma\in\Sigma_B$ and $F_\sigma=\R$, then the measure $d^\times g_\sigma$ corresponds to 
\begin{equation}\label{Haararch}
d^\times g_\sigma=\frac{dxdyd\theta}{y^2},\quad\text{ where }\quad g_\sigma=\big(\begin{smallmatrix}1&x\\&1\end{smallmatrix}\big) \big(\begin{smallmatrix}y^{1/2}&\\&y^{-1/2}\end{smallmatrix}\big)\big(\begin{smallmatrix}\cos\theta&\sin\theta\\-\sin\theta&\cos\theta\end{smallmatrix}\big),    
\end{equation}
for $x\in \R$, $y\in \R_+$ and $\theta\in [0,\pi)$. Finally, if $\sigma\in\Sigma_B$ and $F_\sigma=\C$, then the measure $d^\times g_\sigma$ corresponds to 
\begin{equation}\label{Haararch2}
d^\times g_\sigma=16\frac{\sin2\theta drds_1ds_2dadbd\theta}{r^3},\quad\text{ where }\quad g_\sigma=\left(\begin{smallmatrix}r&s_1+is_2\\&r^{-1}\end{smallmatrix}\right)\left(\begin{smallmatrix}\cos\theta e^{ai}&\sin\theta e^{bi}\\-\sin\theta e^{-bi}&\cos\theta e^{-ai}\end{smallmatrix}\right),
\end{equation}
for $s_1,s_2\in \R$, $r\in \R_+$, $\theta\in [0,\pi/2)$,  $b\in[0,2\pi)$ and $a\in[0,\pi)$.
In the displays above and in the whole article, we adopt the convention that a blank entry on a matrix denotes a 0 in that entry.

\subsection{Gauss sums}\label{Gausssumssec} In this section we introduce the definition of Gauss sums that we will use and prove some standard results which we have not found in the literature stated in the precise form we will need. Let  $$\chi=\prod_v\chi_v:\I_F/F^\times\rightarrow\C^\times$$ be a locally constant character. For a non-archimedean place $v$, write $\mathfrak{c}(\chi_v)$ for the conductor of $\chi_v$. 
 We fix $y_v\in F_v^\times$ such that $d_{F_v}^{-1}=y_v\cdot\mathfrak{c}(\chi_v)$.

The local Gauss sum $\mathfrak{g}(\chi_v,y_v)$ is defined as
\[
\mathfrak{g}(\chi_v,y_v)=\frac{1}{{\rm vol}(\cO_{F_v}^\times)}\int_{\cO_{F_v}^\times}\chi_v(x_v)^{-1}\psi_v(y_vx_v)d^\times x_v,
\]
where $\psi_v$  is the additive character of \eqref{defpsiad}.
Notice that when $\chi_v$ is unramified, $\mathfrak{g}(\chi_v,y_v)=1$.
We write $y=(y_v)_v\in \I_F^\infty$, and we define \emph{the Gauss sum of $\chi$} as
\[
\mathfrak{g}(\chi,y)=\prod_{v\nmid\infty}\mathfrak{g}(\chi_v,y_v).
\]
\begin{remark}
    The Gauss sum $\mathfrak{g}(\chi,y)$ depends on $y$, but a different choice of $y$ has the effect of scaling  $\mathfrak{g}(\chi,y)$ by an element of $\Q(\chi)^\times$, the field generated by the values of $\chi$. When we write equalities ${\rm mod}\; \Q(\chi)^\times$, we will simply write $\mathfrak{g}(\chi)$ instead of $\mathfrak{g}(\chi,y)$.
\end{remark}

\begin{proposition}\label{propertiesGaussS}
    We have that 
    \[
    \mathfrak{g}(\chi,y)\cdot \mathfrak{g}(\chi^{-1},y)=|c(\chi)|\prod_{v\mid c(\chi)}\zeta_v(1)^{2}\chi_v(-1).
    \]
	In particular, if $\chi$ is quadratic associated to the quadratic extension $E/F$
    \[
	  \mathfrak{g}(\chi,y)\equiv i^{\#\{\sigma\mid\infty\colon E_\sigma=\C\}} |D|^{\frac{1}{2}}\mod\Q^\times,	
    \]
where $D$ is the relative discriminant of $E/F$.
\end{proposition}
\begin{proof}

    We compute 
    \begin{eqnarray*}
    \mathfrak{g}(\chi_v,y_v)\cdot \mathfrak{g}(\chi_v^{-1},y_v)&=&\frac{1}{{\rm vol}(\cO_{F_v}^\times)^2}\int_{\cO_{F_v}^\times}\int_{\cO_{F_v}^\times}\chi_v(x_v^{-1}t_v)\psi_v(y_v(x_v+t_v))d^\times x_v d^\times t_v\\
    &=&\frac{1}{{\rm vol}(\cO_{F_v}^\times)^2}\int_{\cO_{F_v}^\times}\int_{\cO_{F_v}^\times}\chi_v(z_v)\psi_v(y_vx_v(1+z_v))d^\times z_v d^\times x_v.
    \end{eqnarray*}
    By \cite[Lemma 2.1]{Spiess}, we have that
    \[
    \int_{\cO_{F_v}^\times}\psi_v(ax_v)d^\times x_v=\left\{\begin{array}{ll}
         {\rm vol}(\cO_{F_v}^\times)& \text{ if } a\in d_F^{-1}, \\
         {\rm vol}(\cO_{F_v}^\times)(1-q_v)^{-1}& \text{ if } {\rm ord}(\delta_va)=-1 ,\\
         0&\mbox{otherwise}.
    \end{array}\right.
    \]
    Then, if we write $c_v=\delta_v^{-1}y_v^{-1}$ we have
    \begin{eqnarray*}
    \mathfrak{g}(\chi_v,y_v)\cdot \mathfrak{g}(\chi_v^{-1},y_v)&=&\frac{1}{{\rm vol}(\cO_{F_v}^\times)}\left(\int_{\cO_{F_v}^\times}\chi_v(z_v)1_{d_F^{-1}}(y_v(1+z_v))d^\times z_v+ \int_{\cO_{F_v}^\times}\frac{\chi_v(z_v)}{(1-q_v)}1_{\cO_{F_v}^\times}(\varpi_vc_v^{-1}(1+z_v))d^\times z_v\right)\\
    &=&\frac{1}{{\rm vol}(\cO_{F_v}^\times)}\left(\int_{\cO_{F_v}^\times}\chi_v(z_v)1_{c(\chi_v)}(1+z_v)d^\times z_v+ \int_{\cO_{F_v}^\times}\frac{\chi_v(z_v)}{(1-q_v)}1_{\varpi_v^{-1}c_v\cO_{F_v}^\times}(1+z_v)d^\times z_v\right).
    \end{eqnarray*}
   Notice that $\mathfrak{c}(\chi_v)=c_v\cO_{F_v}$. Hence, when $c_v\in\cO_{F_v}^\times$ ($\chi_v$ is unramified), then  $\mathfrak{g}(\chi_v,y_v)\cdot \mathfrak{g}(\chi_v^{-1},y_v)=1$. On the other hand, when $c(\chi_v)= \varpi_v\cO_{F_v}$, we have that
   \begin{eqnarray*}
     \mathfrak{g}(\chi_v,y_v)\cdot \mathfrak{g}(\chi_v^{-1},y_v)&=&\frac{{\rm vol}(1+c(\chi_v))}{{\rm vol}(\cO_{F_v}^\times)}\chi_v(-1)+\frac{{\rm vol}(1+c(\chi_v))}{(1-q_v){\rm vol}(\cO_{F_v}^\times)}\left(\sum_{a\in(\cO_{F_v}/\varpi_v\cO_{F_v})\setminus\{0,-1\}}\chi_v(a)\right)\\
    &=&\frac{1}{q_v-1}\chi_v(-1)\left(1+\frac{1}{q_v-1}\right)=\chi_v(-1)\frac{q_v}{(q_v-1)^2}.
   \end{eqnarray*}
    Finally, when ${\rm ord}(c_v)\geq 2$, then
    \begin{eqnarray*}
     \mathfrak{g}(\chi_v,y_v)\cdot \mathfrak{g}(\chi_v^{-1},y_v)&=&\frac{{\rm vol}(1+c(\chi_v))}{{\rm vol}(\cO_{F_v}^\times)}\chi_v(-1)+\frac{(1-q_v)^{-1}}{{\rm vol}(\cO_{F_v}^\times)}\sum_{a\in (\cO_{F_v}/\varpi_v\cO_{F_v})^\times}\int_{-1+a\varpi_v^{-1}c_v+c(\chi_v)}\chi_v(z_v)d^\times z_v\\
     &=&\frac{{\rm vol}(1+c(\chi_v))}{{\rm vol}(\cO_{F_v}^\times)}\chi_v(-1)+\frac{{\rm vol}(1+c(\chi_v))}{{\rm vol}(\cO_{F_v}^\times)(1-q_v)}\chi_v(-1)\sum_{a\in (1+\varpi_v^{-1}c(\chi_v))/(1+c(\chi_v)\setminus \{1\}}\chi_v(a)\\
     &=&\frac{{\rm vol}(1+c(\chi_v))}{{\rm vol}(\cO_{F_v}^\times)}\chi_v(-1)+\frac{{\rm vol}(1+c(\chi_v))}{{\rm vol}(\cO_{F_v}^\times)(q_v-1)}\chi_v(-1)=|c_v|_v(1-q_v^{-1})^{-2}\chi_v(-1),
   \end{eqnarray*}
   and the result follows. 
\end{proof}

\subsection{Finite dimensional representations}\label{findimreps}

Let $k$ be a positive even integer. For any field $L$,
let $\cP(k)_L={\rm Sym}^k(L^2)$ be the space of polynomials in 2 variables over $L$ homogeneous of degree $k$ with $\PGL_2(L)$-action:
\[
\left(\left(\begin{array}{cc}a&b\\c& d\end{array}\right)P\right)(X,Y)=(ad-bc)^{-\frac{k}{2}}P(aX+cY,bX+dY),\qquad P\in\cP(k)_L.
\]  
Let us denote $V(k)_L=\cP(k)_L^\vee$ with dual $\PGL_2(L)$-action:
\[
(g\mu)(P)=\mu(g^{-1}P),\qquad \mu\in V(k)_L, \ g \in \PGL_2(L).
\]  
Notice that $V(k)_L\simeq \cP(k)_L$ by means of the isomorphism
\begin{equation}\label{dualVP}
V(k)_L\longrightarrow\cP(k)_L,\qquad\mu\longmapsto\mu((Xy-Yx)^k).
\end{equation}

\subsubsection{Polynomials and torus}\label{PolTorus}
From the fixed embedding $E\hookrightarrow B$ we can define an isomorphism $B\otimes_FE\simeq{\rm M}_2(E)$. In fact, we have that $B=E\oplus EJ$, where $J$ normalizes $E$ and $J^2\in F^\times$. Hence, for this fixed choice of $J$ we have the embedding $ \iota\colon   B\hookrightarrow\M_2(E)$ given by
\begin{equation}\label{embEinB}
\iota( e_1+e_2J)= \left(\begin{array}{cc}
    e_1 & J^2e_2 \\
    \bar e_2 & \bar e_1
\end{array}\right),
\end{equation}
where here the bar denotes the non-trivial automorphism in $\Gal(E/F)$.
For a given embedding $\nu:F\hookrightarrow\C$ we fix an extension $\nu_E\colon E\hookrightarrow\C$. The composition $\nu_E\circ \iota$ gives rise to an embedding  $G(F_\sigma)\hookrightarrow\PGL_2(\C)$. This induces an action of $G(F_\infty)=G(F\otimes_\Q\C)$ on the spaces
\[
V(\underline k):=\bigotimes_{\nu}V(k_{\nu})_\C \qquad\text{ and  }\qquad\cP(\underline k):=\bigotimes_{\nu}\cP(k_{\nu})_\C,
\]
where  $\underline k=(k_{\nu})\in (2\N)^d$, $d = [F\colon \Q]$, and $\nu$ runs over the embeddings of $F$ in $\C$. The natural embedding $G(F)\subseteq G(F_\infty)$ induces a structure of $G(F)$-representation to $V(\underline k)$ and to $ \cP(\underline k)$.

The embedding $  E\hookrightarrow B \stackrel{\iota}{\hookrightarrow}{\rm M}_2(E)$ maps $e$ to $\big(\begin{smallmatrix} e&\\&\bar e\end{smallmatrix}\big)$. This implies that we have a $T(F_\infty)$-equivariant morphism
\begin{equation}\label{PtoCoverC}
    \cP(\underline k)\longrightarrow C(T(F_\infty),\C);\qquad \bigotimes_{\nu} P_{\nu}\longmapsto\ipa{(t_\sigma)_{\sigma\mid\infty}\mapsto\prod_{\sigma\mid\infty}\prod_{\nu\mid\sigma}P_{\nu}\ipa{1,\nu_E\ipa{\frac{t_\sigma}{\bar t_\sigma}}}\nu_E\ipa{\frac{t_\sigma}{\bar t_\sigma}}^{-\frac{k_{\nu}}{2}}},
\end{equation}
where $ C(T(F_\infty),\C)$ denotes the set of continuous functions from $T(F_\infty)$ to $\C$.

\subsubsection{Invariant polynomials}\label{invpoly} For $\underline{m}=(m_{\nu})_{\nu}\in\Z^d$, $\lambda=(\lambda_\sigma)_\sigma\in F_\infty$ and $t=(t_\sigma)_\sigma\in T(F_\infty)$, we write 
\[
\lambda^{\underline{m}}:=\prod_{\sigma\in\Sigma_F}\prod_{\nu\mid\sigma} \nu(\lambda_\sigma)^{m_{\nu}};\qquad t^{\underline{m}}:=\prod_{\sigma\in\Sigma_F}\prod_{\nu\mid\sigma} \nu_E\ipa{\frac{t_\sigma}{\bar t_\sigma}}^{m_{\nu}}.
\]

For reasons that will become apparent later, it is more convenient to consider degrees of the form $(\underline k-2)=(k_{\nu}-2)$, for some $\underline k=(k_{\nu})\in (2\N)^d$ with $k_{\nu}\geq 2$.
If $\frac{2-k_{\nu}}{2}\leq m_{\nu}\leq\frac{k_{\nu}-2}{2}$ for all $\nu:F\hookrightarrow\C$, the character $t\mapsto t^{\underline m}$ corresponds by means of the morphism \eqref{PtoCoverC} to an element $\mu_{\underline m}=\bigotimes_{\nu}\mu_{m_{\nu}}\in V(\underline{k}-2)$ given by
\begin{equation}\label{defmumglobal}
    \mu_{m_{\nu}}\left(\left|\begin{array}{cc}X& Y\\  x&y \end{array}\right|^{k_{\nu}-2}\right)=x^{\frac{k_{\nu}-2}{2}-m_{\nu}}y^{\frac{k_{\nu}-2}{2}+m_{\nu}},\quad\mbox{or simply}\quad \mu_{\underline m}\left(\left|\begin{array}{cc}X& Y\\  x&y \end{array}\right|^{\underline k-2}\right)=x^{\frac{\underline k-2}{2}-\underline m}y^{\frac{\underline k-2}{2}+\underline m}.
\end{equation}
Hence $\mu_{\underline m}\in V(\underline{k}-2)$ is the unique element, up to constant, such that $t \mu_{\underline m}=t^{-\underline m}\mu_{\underline m}$, where on the left-hand side $t$ acts via the  action of $T(F_\infty)\subset G(F_\infty)$ on $V(\underline{k}-2)$ and on the right-hand side the complex number $t^{-\underline m}$ acts by multiplication. 

\subsubsection{Other finite dimensional representations}\label{otherfdreps}

Write ${\rm Tr}$ for the reduced trace on $B$ and let us consider the finite dimensional $F$-vector space
$B_0=\{b\in B\mid {\rm Tr}(b)=0\}$,
endowed with a left action of $B^\times$ given by conjugation. Let us consider the non-degenerate $B^\times$-invariant symmetric pairing
\[
\langle\;,\;\rangle:B_0\times B_0\longrightarrow F;\qquad \langle b_1,b_2\rangle={\rm Tr}(b_1\cdot \bar b_2),
\]
where $\overline{(\cdot)}$ denotes the non-trivial conjugation on $B$. For any even integer $k\geq 4$, let us consider the morphism
\[
\Delta_k:{\rm Sym}^{\frac{k}{2}}(B_0)\longrightarrow{\rm Sym}^{\frac{k}{2}-2}(B_0);\qquad \Delta_k(b_1\cdot b_2\cdots b_{\frac{k}{2}})=\sum_{i<j}\langle b_i,b_j\rangle\; b_1\cdots \hat b_i\cdots\hat b_j\cdots b_{\frac{k}{2}}
\]
We define
\[
V_0=F;\qquad V_2=B_0;\qquad V_k=\ker(\Delta_k)\subset {\rm Sym}^{\frac{k}{2}}(B_0),\quad \mbox{if }\;k\geq 4.
\]
Notice that the action of $B^\times$ on $B_0$ induces an action of $B^\times$ on $V_k$.
\begin{lemma}\label{lemaVkVk}
Given an extension $L/F$ admitting an embedding $\imath:B\hookrightarrow \M_2(L)$, the following morphism
    \[
    \kappa:V_k\otimes_F L\longrightarrow \cP(k)_L;\qquad \kappa\left(b_1\cdots b_{\frac{k}{2}}\right)(X,Y)=\prod_{i=1}^{\frac{k}{2}}{\rm Tr}\left(\left(\begin{array}{r}  Y\\-X\end{array}\right)(X\; Y)\;\imath(b_i)\right),
    \]
    is an isomorphism of $B^\times$-modules over $L$.
\end{lemma}
\begin{proof}
    This result is fairly standard, but we will provide a proof due to the absence of a suitable reference.
    The fact that $\kappa$ is $B^\times$-equivariant follows from a simple calculation. Moreover, it is clear that it is an isomorphism for $k=2$, hence we will identify $B_0\otimes_F L$ with $\cP(2)_L=:\cP_2$.

    The morphism $\kappa$ of the statement, for general $k$, comes from the natural surjective morphism 
    \begin{equation}\label{auxeqsymPk}
    \kappa_n:{\rm Sym}^n(\cP_2)\longrightarrow\cP(2n)_L;\qquad \kappa_n\left(p_1\cdots p_{n}\right)=\prod_{i=1}^{n}p_n;\qquad n:=\frac{k}{2}.
    \end{equation}
    Moreover, if we consider $a=(X^2\cdot Y^2)-(XY\cdot XY)\in \ker\kappa_2\subset{\rm Sym}^2(\cP_2)$, we can define an injective morphism
    \[
    \iota_n:{\rm Sym}^{n-2}(\cP_2)\hookrightarrow {\rm Sym}^n(\cP_2);\qquad \iota_n(q_1\cdots q_{n-2})=(a\cdot q_1\cdots q_{n-2}),
    \]
    that provides the isomorphism $\ker\kappa_n\simeq \operatorname{Im}\iota_n$ because $\dim(\cP(2n)_L)=2n+1$ and $\dim({\rm Sym}^n(\cP_2))=\frac{(n+2)(n+1)}{2}$. Notice that the symmetric pairing $\langle\;,\;\rangle$ corresponds to the paring on $\cP_2$ provided by \eqref{dualVP}, and gives rise to a perfect symmetric pairing $\langle\;,\;\rangle_n$ on ${\rm Sym}^n(\cP_2)\otimes{\rm Sym}^n(\cP_2)$ that identifies ${\rm Sym}^n(\cP_2)$ with ${\rm Sym}^n(\cP_2)^\vee$
    \[
    \langle\;,\;\rangle_n:{\rm Sym}^n(\cP_2)\times {\rm Sym}^n(\cP_2)\longrightarrow L;\qquad \langle\;(p_1,\cdots p_n),(q_1,\cdots,q_n)\;\rangle_n=\sum_{\sigma\in S_n}\prod_{i=1}^n\langle p_i,q_{\sigma(i)}\rangle.
    \]
    Hence, the result will follow if we prove that the following diagram is commutative:
    \[
    \xymatrix{
    0\ar[r]&V(2n)_L\ar[r]^{\kappa_n^\vee}&{\rm Sym}^n(\cP_2)^{\vee}\ar[r]^{\iota_n^\vee}&{\rm Sym}^{n-2}(\cP_2)^\vee\ar[r]&0\\
    &&{\rm Sym}^n(\cP_2)\ar[r]^{\Delta_{2n}}\ar[u]^{\simeq}&{\rm Sym}^{n-2}(\cP_2)\ar[u]^{\simeq}.&
    }
    \]
    Indeed, it is easy to check that $\langle a,(p_1\cdot p_2)\;\rangle_2=\langle p_1,p_2\rangle$ for all $p_1,p_2\in\cP_2$. Hence, for any $\underline p=(p_1\cdots p_n)\in {\rm Sym}^n(\cP_2)$ and $\underline q=(q_1\cdots q_{n-2})\in {\rm Sym}^{n-2}(\cP_2)$,
    \begin{eqnarray*}
    \langle\underline p,\iota_n(\underline q)\rangle_n&=&\sum_{\sigma\in S_n/(1,2)}\langle a,(p_{\sigma(1)}\cdot p_{\sigma(2)})\rangle_2\prod_{i=1}^{n-2}\langle q_i,p_{\sigma(i+2)}\rangle
    =\sum_{\sigma\in S_n/(1,2)}\langle p_{\sigma(1)}, p_{\sigma(2)}\rangle\prod_{i=1}^{n-2}\langle q_i,p_{\sigma(i+2)}\rangle\\
    &=&\sum_{i<j}\langle p_{i}, p_{j}\rangle\langle(p_1\cdots\hat p_i\cdots \hat p_j\cdots p_n),(q_1\cdots q_{n-2})\rangle_{n-2}=\langle\Delta_{2n}(\underline p),\underline q\rangle_{n-2},
    \end{eqnarray*}
    and, therefore, the commutativity of the diagram follows. 
\end{proof}
\begin{remark}\label{dualVkmodelpairing}
    The above proof implies that there exists a $B^\times$-invariant perfect pairing 
    $\langle\;,\;\rangle:V_k\times V_k\rightarrow F$,
    that fits under $\kappa$ with the perfect pairing $\langle\;,\;\rangle_{\cP(k)_L}$ provided by \eqref{dualVP}. Indeed, $\langle\;,\;\rangle$ is induced by 
    \[
    \langle\;,\;\rangle=\frac{1}{(k/2)!}\langle\;,\;\rangle_{\frac{k}{2}}:{\rm Sym}^{\frac{k}{2}}(B_0)\times {\rm Sym}^{\frac{k}{2}}(B_0)\longrightarrow F;\qquad \langle\;(b_1,\cdots b_{\frac{k}{2}}),(a_1,\cdots,a_{\frac{k}{2}})\;\rangle=\frac{1}{(k/2)!}\sum_{\sigma\in S_{k/2}}\prod_{i=1}^{k/2}{\rm Tr}(b_i\cdot \bar a_{\sigma(i)}).
    \]
    Notice that $\langle\;,\;\rangle_{\cP(k)_L}$ is the unique perfect pairing such that $\langle x^k,y^k\rangle_{\cP(k)_L}=1$. Moreover, if $b_1,b_2\in B_0$ are such that
    $\imath(b_1)=\left(\begin{smallmatrix}
        0&-1\\0&0
    \end{smallmatrix}\right)$ and $\imath(b_2)=\left(\begin{smallmatrix}
        0&0\\1&0
    \end{smallmatrix}\right)$, then $\kappa(b_1\cdots b_1)=x^k$ and $\kappa(b_2\cdots b_2)=y^k$. On the other side, 
    \[
    \langle\;(b_1,\cdots b_1),(b_2,\cdots,b_2)\;\rangle=\left({\rm Tr}(b_1\cdot \bar b_2)\right)^{\frac{k}{2}}=\left({\rm Tr}(b_1\cdot \bar b_2)\right)^{\frac{k}{2}}=\left({\rm Tr}(\imath(b_1)\cdot \overline{\imath(b_2)})\right)^{\frac{k}{2}}=1.
    \]
    Thus, we deduce that $\langle\kappa(v),\kappa(w)\rangle_{\cP(k)_L}=\langle v,w\rangle$, for all $v,w\in V_k$.
\end{remark}
\begin{lemma}\label{rationalinvariants}
    There exists a unique $E^\times$-invariant element of $V_k$ up to constant. Moreover, the image of such element under the natural morphism
    \begin{equation*}
    V_k\longrightarrow\cP(k)_E\longrightarrow C(T(F),E);\qquad v\longmapsto f_{v}(t)=\kappa\left(v\right)\ipa{1,\ipa{\frac{t}{\bar t}}}\ipa{\frac{t}{\bar t}}^{-\frac{k}{2}}
\end{equation*}
is a constant function that takes values in $\alpha^{\frac{k}{4}}F$, where $\alpha$ is such that $E=F(\sqrt{\alpha})$.
\end{lemma}
\begin{proof}
The unicity of the $E^\times$-invariant element comes from the isomorphism of Lemma \ref{lemaVkVk}. Now write $B=F\oplus Fi \oplus Fj\oplus Fk$, where $i^2,j^2,k^2\in F^\times$, $E=F\oplus Fi$ and $k=ij=-ji$. Notice that $B_0= Fi \oplus Fj\oplus Fk$. To prove the existence part of the statement, we argue by cases:
\begin{itemize}
\item    If $k=2$ the element $v_2=i\in B_0$ is $E^\times$-invariant.

\item If $k=4$ the element $v_4=(j\cdot j)-\frac{1}{i^2}(k\cdot k)-2\frac{j^2}{i^2}(i\cdot i)$ is $E^\times$-invariant. Indeed, if $t\in E^\times$ satisfies $t\bar t^{-1}=a+bi$, we have 
\begin{eqnarray*}
    t v_4&=&(tjt^{-1}\cdot tjt^{-1})-\frac{1}{i^2}(t kt^{-1}\cdot tkt^{-1})-2\frac{j^2}{i^2}(tit^{-1}\cdot tit^{-1})=\left(\frac{t}{\bar t}j\cdot \frac{t}{\bar t}j\right)-\frac{1}{i^2}\left(\frac{t}{\bar t} k\cdot \frac{t}{\bar t}k\right)-2\frac{j^2}{i^2}(i\cdot i)\\
    &=&\left((aj+bk)\cdot (aj+bk)\right)-\frac{1}{i^2}\left((ak+i^2bj)\cdot (ak+i^2bj)\right)-2\frac{j^2}{i^2}(i\cdot i)\\
    &=&(a^2-i^2b^2)(j\cdot j)-\frac{(a^2-i^2b^2)}{i^2}(k\cdot k)-2\frac{j^2}{i^2}(i\cdot i)=v_4.
\end{eqnarray*}
Moreover, $v_4\in V_4$ because
\[
\Delta_4(v_4)=\langle j,j\rangle- \frac{1}{i^2}\langle k,k\rangle-2\frac{j^2}{i^2}\langle i,i\rangle=-2j^2+ \frac{2}{i^2}k^2+4\frac{j^2}{i^2}i^2=0.
\]

\item Finally, if $k\geq 4$ and $n\leq \lfloor\frac{k}{4}\rfloor$ we can construct a $E^\times$-invariant element $\tilde v_k^n\in {\rm Sym}^{\frac{k}{2}}(B_0)$
\[
\tilde v_k^n:=
    \left(v_4\stackrel{n)}{\cdots}v_4\cdot v_2\stackrel{k/2-2n)}{\cdot\cdots\cdot}v_2\right).
\]
Notice that, for all $n\leq \lfloor\frac{k}{4}\rfloor$, we have
$\Delta_k(\tilde v_k^n)\in \bigoplus_{m=0}^{\lfloor\frac{k-4}{4}\rfloor}F\; \tilde v_{k-2}^m$.
Thus, there must be a linear combination $v_k=\sum_{n=0}^{\lfloor\frac{k}{4}\rfloor}a_n\tilde v_k^n$ such that $\Delta_k(v_k)=0$. This is the desired $E^\times$-invariant element of $V_k$.
\end{itemize}
In order to prove the second claim, notice that
\begin{equation}\label{auxeqSym}
    f_{\left(b_1\cdots b_{\frac{k}{2}}\right)}(t)=\kappa\left(b_1\cdots b_{\frac{k}{2}}\right)\ipa{\ipa{\frac{\bar t}{t}}^{-\frac{1}{2}},\ipa{\frac{t}{\bar t}}^{\frac{1}{2}}}=\prod_{i=1}^{\frac{k}{2}}{\rm Tr}\left(\left(\begin{array}{cc}  1&t/\bar t\\-\bar t/t&-1\end{array}\right)\;\imath(b_i)\right).
\end{equation}
Since there exist $e\in E^\times$ such that
\[
\imath(i)=\left(\begin{array}{cc}  i&\\&-i\end{array}\right),\qquad \imath(j)=\left(\begin{array}{cc}  &J^2e\\\bar e&\end{array}\right),\qquad \imath(i)=\left(\begin{array}{cc}  &J^2ei\\-\bar ei&\end{array}\right),
\]
we deduce that
\[
f_{v_2}(t)=2i\in i F;\qquad f_{v_4}(t)=\ipa{\frac{t}{\bar t}\bar e+J^2\frac{\bar t}{t}e}^2-\frac{1}{i^2}\ipa{-i\frac{t}{\bar t}\bar e+iJ^2\frac{\bar t}{t}e}^2-2\frac{j^2}{i^2}4i^2=-4j^2\in F.
\]
Notice that if we define the morphism ${\rm Sym}^{\frac{k}{2}}(B_0)\rightarrow C(T(F),E)$, $\tilde v\mapsto f_{\tilde v}$, by means of the formula \eqref{auxeqSym} then 
\[
f_{\tilde v_k^n}\in  i^{\frac{k}{2}-2n}F=i^{\frac{k}{2}}F, 
\]
and the last claim follows.
\end{proof}

\subsection{Models over number fields}\label{remdefmodVk}

Given a weight $\underline{k}=(k_\nu)_{\nu:F\hookrightarrow\bar\Q}$ as in previous sections, we consider 
\begin{equation}\label{defQk}
G_{\underline k}:=\{\tau\in \Gal(\bar\Q/\Q) \colon k_\nu=k_{\tau\nu}\text{ for all }\nu\},\qquad L_{\underline k}:=\bar\Q^{G_{\underline k}}.
\end{equation}
Thus, $\bigotimes_{\nu}V_{k_\nu}\otimes_{F,\nu} \bar\Q$ has a natural action of $G_{\underline k}$ given by
\[
\tau\left(\sum_i a_i\bigotimes_\nu b^i_\nu \right)=\sum_i \tau(a_i)\bigotimes_{\tau\nu} b^i_\nu,\qquad \tau\in G_{\underline k},\quad a_i\in\bar\Q,\quad b^i_\nu\in V_{k_\nu}.
\]
Notice that this Galois action commutes with the action of $G(F)$. Indeed, if $v=\sum_i a_i\bigotimes_\nu b^i_\nu\in \bigotimes_{\nu}V_{k_\nu}\otimes_{F,\nu} \bar\Q$,
\[
g\ast\left(\tau\left(v \right)\right)=g\ast\left(\sum_i \tau(a_i)\bigotimes_{\tau\nu} b^i_\nu\right)=\sum_i \tau(a_i)\bigotimes_{\tau\nu} (\tau\nu)(g)\ast b^i_\nu=\tau\left(\sum_i a_i\bigotimes_{\nu} \nu(g)\ast b^i_\nu\right)=\tau\left(g\ast\left(v \right)\right).
\]
  We consider the non-trivial $L_{\underline k}$-vector space $V(\underline{k})_{L_{\underline k}}:=\left(\bigotimes_{\nu}V_{k_\nu}\otimes_{F,\nu} \bar\Q\right)^{G_{\underline k}}$. By Lemma \ref{lemaVkVk}, this space defines a $L_{\underline k}$-rational model of the irreducible $G(F)$-representation $V(\underline{k})$.
Recall that, given the embedding $E\hookrightarrow B$, the morphism from Lemma \ref{rationalinvariants} induces a map
\[
V(\underline{k})_{L_{\underline k}}\hookrightarrow\bigotimes_{\nu}V_{k_\nu}\otimes_{F,\nu} \bar\Q\longrightarrow C(T(F_\infty),\C).
\]

\begin{lemma}\label{rationalinvariants2}
There exists a $T(F)$-invariant vector $v_{\underline k}\in V(\underline{k})_{L_{\underline k}}$ which is mapped, under the above morphism, to the constant function $\sqrt{\alpha^{\frac{\underline k}{2}}} \in C(T(F_\infty),\bar\Q)$, where $\alpha\in F$ is such that $E=F(\sqrt{\alpha})$.
\end{lemma}
\begin{proof}
By Lemma \ref{rationalinvariants}, the constant function $\sqrt{\alpha^{\frac{\underline k}{2}}}$ admits a preimage of the form $v_{\underline k} = \bigotimes_\nu v_{k_\nu} \in \bigotimes_\nu V_{k_\nu}$, where the tensor product is taken over $F$ and each vector $v_{k_\nu}$ depends only on the corresponding weight $k_\nu$. It is clear that $v_{k_\nu}$ if $G_{\underline k}$-invariant, hence, the result follows.
\end{proof}

\section{Local archimedean representations}\label{LocArchRep}

In this section we will study the infinite dimensional local archimedean representations generated by a cohomological automorphic form for $\GL_2$ of even weight and trivial central character. These representations correspond to $(\mfg,K)$-modules for the real Lie groups $\PGL_2(\R)$ or $\PGL_2(\C)$.

\subsection{The cohomological $(\mfg,K)$-modules (of discrete series) for $\PGL_2(\R)$}\label{GKR}

We write $K_\R$ and $K_{\R,+}$ for the maximal compact subgroup of $\PGL_2(\R)$ and its connected component provided by the image of ${\rm O}(2)=\SO(2)\rtimes\left\langle\hat H\right\rangle$ and $\SO(2)$, respectively, where
\[
\SO(2):=\left\{\kappa(\theta):=\left(\begin{array}{cc}\cos\theta&\sin\theta\\-\sin\theta&\cos\theta\end{array}\right),\;\theta\in S^1\right\}\subset \SL_2(\R);\qquad \hat H=\left(\begin{array}{rr}
       1  & 0 \\
        0 & -1
    \end{array}\right).
\]
Recall that any $g\in\GL_2(\R)_+$ admits a decomposition
\begin{equation}\label{eqstar}
g=u\left(\begin{array}{rr}y^{\frac{1}{2}}&xy^{-\frac{1}{2}}\\&y^{-\frac{1}{2}}\end{array}\right)\kappa(\theta),\qquad y\in\R_+^\times,\;u\in\R_+,\;x\in\R,\;\theta\in S^1.
\end{equation}
Notice that the Lie algebra of the real Lie group $\PGL_2(\R)$ is 
\[
\mfg_\R\simeq {\rm Lie}(\SL_2(\R))\simeq\{g\in\M_2(\R),\;\operatorname{Tr} g=0\}=\R \hat H\oplus\R\tilde W\oplus\R W,
\qquad
\tilde W:=\left(\begin{array}{cc}
       0  & 1 \\
        1 & 0
    \end{array}\right);\quad W:=\left(\begin{array}{rr}
       0  & 1 \\
        -1 & 0
    \end{array}\right).
\]
Moreover, $\cK_\R:={\rm Lie}(K_\R)=\R W$.

For any character $\chi:\R^\times\rightarrow\C^\times$, let us consider the induced representation of $\PGL_2(\R)_+$
\[
\cB(\chi):=\left\{f:\GL_2(\R)_+\rightarrow\C:\;f\left(\left(\begin{array}{cc}t_1&x\\&t_2\end{array}\right)g\right)=\chi(t_1/t_2)\cdot f(g)\right\}.
\]
By \eqref{eqstar}, we can identify
$\cB(\chi)\simeq\{f:S^1\rightarrow\C:\;f(\theta+\pi)= f(\theta)\}$.
Notice that the only characters $S^1\rightarrow\C^\times$ appearing in $\cB(\chi)$ under the above identification are those of the form $e^{2in\theta}$ with $n\in\Z$. If we write $\cB(\chi,n)$ for the subspace $\C e^{2in\theta}$ inside $\cB(\chi)$, then we can consider
\[
\tilde \cB(\chi):=\bigoplus_{n\in\Z}\cB(\chi,n)\subseteq \cB(\chi).
\] 
It is clear that $\tilde\cB(\chi)$ is a $(\mfg_\R,K_{\R,+})$-module.
If $\chi(t)=\chi_k(t):=t^{\frac{k}{2}}$, for an even integer $k\in 2\Z$, then we have a morphism of $\GL_2(\R)_+$-representations: (see \cite[proposition 4.2]{preprintsanti2})
\[
\rho:\cB(\chi_k)\longrightarrow V(k-2); \qquad \rho(f)(P):=\frac{1}{2\pi}\int_{0}^{2\pi}f(\kappa(\theta))P(-\sin\theta,\cos\theta)d\theta.
\]
Moreover, each $\cB(\chi_k,n)$ is generated  by the function 
\begin{equation}\label{fn}
f_n\left(u\left(\begin{array}{cc}t&s\\&t^{-1}\end{array}\right)\kappa(\theta)\right)=t^k\cdot e^{2in\theta}\in\cB(\chi_k).
\end{equation}
To extend $\tilde \cB(\chi_k)$ to a $(\mfg_\R,K_\R)$-module one has to define an action of $\hat H=\big(\begin{smallmatrix}1&\\&-1\end{smallmatrix}\big)$. It turns out we only have two possibilities: $\hat H(f_n)=\pm(-1)^{\frac{k-2}{2}} f_{-n}$  (see \cite[\S 1.3]{ESsanti}). Depending on this choice of sign, we obtain two different $(\mfg_\R,K_\R)$-modules  $\cB(\chi_k)^\pm$. Moreover, $D(k) :=\tilde \cB(\chi_k)\cap \ker(\rho) $ is the unique sub-$(\mfg_\R,K_\R)$-module for both $\tilde\cB(\chi_k)^\pm$. Thus, we have the exact sequences of $(\mfg_\R,K_\R)$-modules:
\begin{align}
0\longrightarrow  &D(k)\longrightarrow \tilde \cB(\chi_k)^+\stackrel{\rho_+}{\longrightarrow} V(k-2)\longrightarrow 0 \label{eq:exs1}\\
0\longrightarrow &D(k)\longrightarrow \tilde \cB(\chi_k)^-\stackrel{\rho_-}{\longrightarrow} V(k-2)(-1)\longrightarrow 0,\label{eq:exs2}
\end{align}
where $V(k-2)(-1)$ is the representation $V(k-2)$ twisted by the character $g\mapsto{\rm sign}\det(g)$. We will use the notation $V(k-2)(\pm)$ to denote either $V(k-2)$ or $V(k-2)(-1)$.

If we write $P_m(X,Y):=(Y+iX)^{m}(Y-iX)^{k-2-m}\in\cP(k-2)$, for $0\leq m\leq k-2$, we compute
\begin{equation}\label{norm1}
\rho(f_n)(P_m)=\frac{1}{2\pi}\int_{0}^{2\pi}e^{(k-2-2m+2n)i\theta} d\theta=\left\{\begin{array}{ll}1 &\text{ if } n=m-\frac{k-2}{2},\\0& \text{ if }n\neq m-\frac{k-2}{2}.\end{array}\right.
\end{equation}
This implies that the kernel of $\rho_\pm$ is generated by $f_n$ with $|n|\geq\frac{k}{2}$, and we deduce that
\[
D(k)=\sum_{|n|\geq \frac{k}{2}}\cB(\chi_k,n),
\]
and therefore the exact sequences \eqref{eq:exs1} and \eqref{eq:exs2} can be written as 
\begin{equation}\label{exseq1}
0\longrightarrow D(k)=\sum_{|n|\geq \frac{k}{2}}\cB(\chi_k,n)\stackrel{\iota_\pm}{\longrightarrow} \tilde \cB(\chi_k)^\pm=\sum_{n\in \Z}\cB(\chi_k,n) \stackrel{\rho_\pm}{\longrightarrow} V(k-2)(\pm)\longrightarrow 0,
\end{equation}
where $\iota_\pm(f_n)=(\pm 1)^{\frac{1-{\rm sign}(n)}{2}}f_n$.
Of course such exact sequences do not split in the category of $(\mfg_\R,K_\R)$-modules, but they split when regarded as ${\rm O}(2)$-modules. Hence, there exist unique $K_\R$-equivariant sections of $\rho_\pm$, namely, $K_\R$-equivariant morphisms  
\[
s_\pm:V(k-2)(\pm)\longrightarrow \tilde\cB(\chi_k)^\pm,\qquad \rho_\pm\circ s_\pm={\rm id}.
\]

\subsection{Explicit cohomology classes for $\mfg_\R$}\label{ExpliCCR}
   In this section we will explore the $(\mfg_\R,K_\R)$-cohomology of the previously described modules $D(k)$. Notice that  in $\mfg_\R$ we have the relations
    \begin{eqnarray}\label{relGK}
        \kappa(\theta)^{-1}\hat H\kappa(\theta)&=&\cos(2\theta)\hat H+\sin(2\theta)\tilde W,\\
        \kappa(\theta)^{-1}\tilde W\kappa(\theta)&=&-\sin(2\theta)\hat H+\cos(2\theta)\tilde W.
    \end{eqnarray}
    Thus, if we write $\kappa_1=\kappa(\pi/4)\in \SO(2)$, then we obtain that
    $\tilde W=\kappa_1^{-1}\hat H\kappa_1$.
    This implies that any $(\mfg_\R,K_\R)$-module is completely determined by the action of $K_\R$ and $\hat H$.

\subsubsection{1-cocycles associated with $D(k)$} 
The space of 1-cocycles with values in a given $(\mfg_\R,K_\R)$-module $M$ is 
\[
Z^1((\mfg_\R,K_\R),M)=\left\{\varphi_1\in \Hom_{K_\R}\left(\mfg_\R/\cK_\R,M\right)\mid  d\varphi_1(X,Y)=0 \text{ for all } X,Y\in \mfg_\R/\cK_\R\right\},
\]
where $d\varphi_1(X,Y):=X\varphi_1(Y)-Y\varphi_1(X)-\varphi_1([X,Y])$. Notice that $\mfg_\R/\cK_\R=\R \hat H\oplus\R\tilde W$.
Moreover, by \eqref{relGK}, an homomorphism $\varphi_1\in \Hom_{K_\R}\left(\mfg_\R/\cK_\R,M\right)$ is characterized by the image of $\hat H$. Hence, to describe any cocycle $\varphi_1\in Z^1((\mfg_\R,K_\R),M)$ it is enough to provide $\varphi(\hat H)\in M$.

Take the $(\mfg_\R,K_\R)$-module $\Hom(V(k-2)(\pm),D({k}))$, and consider the unique $K_\R$-equivariant section $s_\pm:V(\underline k-2)(\pm)\rightarrow\tilde \cB(\chi_{\underline k})^\pm$. We define the morphisms $c_1^\pm \in \Hom_{K_\R}\left(\mfg_\R/\cK_\R,\Hom(V(k-2)(\pm),D({k}))\right)$ by 
\[
c^\pm_1(X)(\mu):=\left(X(s_\pm\mu)-s_\pm(X\mu)\right).
\]
The corresponding classes in cohomology are precisely the classes associated to the exact sequences \eqref{exseq1}. Hence, they define non-trivial elements
\begin{equation}\label{firstdefc1}
    c_1^\pm\in H^1((\mfg_\R,K_\R),\Hom(V(k-2)(\pm),D({k})))={\rm Ext}^1(V(k-2)(\pm),D(k)).
\end{equation}
Let us consider 
\begin{equation}\label{firstdefds}
\delta s_\pm:=c_1^\pm(D)\in\Hom_{K_\R}\left(\mfg_\R/\cK_\R,\Hom(V(k-2)(\pm),D({k}))\right),\qquad   \text{ where } D:=\left(\begin{array}{cc}
    1&0\\0&0
\end{array}\right).
\end{equation}
In \cite[Proposition 4.7]{preprintsanti2} the morphisms $\delta s_\pm$ are computed explicitly. Since $2D= \hat H$ in $\mfg_\R$,
we can use such a result to obtain the image $c_1^\pm(\hat H)$ that characterizes $c_1^\pm$:
\begin{equation}\label{descc1+}
c_1^\pm(\hat H)(\mu_m)=2\delta s_\pm(\mu_m)=(k-1)\left((-i)^{\frac{k-2}{2}+m}f_{\frac{k}{2}}\pm i^{\frac{k-2}{2}+m}f_{-\frac{k}{2}}\right),
\end{equation}
where the $\mu_m$ are defined in \eqref{defmumglobal} and form a basis of $V(k-2)$.

\subsubsection{2-cocycles associated with $D(k)$}

In the previous section we have constructed cohomology classes
\[
c_1^\pm\in H^1((\mfg_\R,K_\R),M^\pm), \qquad \text{where }  M^\pm=\Hom(V(k-2)(\pm),D(k)).
\]
Notice that we have a natural $(\mfg_\R,K_\R)$-equivariant bilinear pairing
\begin{align}\label{eq:pairing}
M^+\times M^-\longrightarrow D({ k})\otimes D({ k})(-1);\qquad (\varphi_1,\varphi_2)\mapsto\varphi_1\varphi_2(\Upsilon), 
\end{align}
where $$\Upsilon=\left|\begin{array}{cc}
    x_1 & y_1 \\
     x_2 &  y_2
\end{array}\right|^{k-2}\in\cP(k-2)^{\otimes 2}\simeq V(k-2)^{\otimes 2}$$ and $\bullet(-1)$ means twisting by the character $({\rm sgn}\det):K_\R\rightarrow\{\pm 1\}$. 

Similarly as before, we define the space of $2$-cocycles
\begin{equation}\label{def2cocycles}
 Z^2((\mfg_\R,K_\R),D(k)^{\otimes 2}(-1)):=\left\{\varphi_2\in \Hom_{K_\R}\left(\bigwedge^2\mfg_\R/\cK_\R,D(k)^{\otimes 2}(-1)\right)\colon  d\varphi_2(X,Y,Z)=0\right\},
\end{equation}
where $$d\varphi_2(X,Y,Z):=X\varphi_2(Y,Z)-Y\varphi_2(X,Z)+Z\varphi_2(X,Y)-\varphi_2([X,Y],Z)+\varphi_2([X,Z],Y)-\varphi_2([Y,Z],X).$$
In the following result we give an explicit $2$-cocycle representing the cohomology class of the cup-product $(c_1^+\cup c_1^-)\in H^2((\mfg_\R,K_\R),D(k)^{\otimes 2}(-1))$, where the cup product is taken with respect to the pairing \eqref{eq:pairing}

\begin{proposition}\label{propcupprod1}
    The cup-product $(c_1^+\cup c_1^-)\in H^2((\mfg_\R,K_\R),D(k)^{\otimes 2}(-1))$ with respect to the pairing \eqref{eq:pairing} is provided by the 2-cocycle $c_2$ whose image is characterized by
    \[
    c_2(\hat H,\tilde W)=\frac{-8i}{{\rm vol}(K_{\R,0},dk)}\int_{K_{\R,0}}k\ast \delta s_+(\Upsilon)d k,
    \]
    for any choice of a Haar measure $dk$ of $K_{\R,+}$.
\end{proposition}
\begin{proof}
Notice that $K_\R=O(2)/\pm 1$ and $K_{\R,+}=\SO(2)/\pm 1$. 
On the one hand, the element $\Upsilon$ corresponds  under the isomorphism $V(k-2)^{\otimes 2}\simeq \cP(k-2)^{\otimes 2}$ to
\[
\Upsilon=\sum_{m}\binom{k-2}{\frac{k-2}{2}+m}(-1)^{\frac{k-2}{2}-m} \mu_m\otimes\mu_{-m}.
\]
By \eqref{descc1+}, we obtain that $I:=4(k-1)^{-2}\int_{0}^\pi\kappa(\theta)\ast \delta s_+(\Upsilon)d \theta$ is given by
\begin{eqnarray*}
I&=&\int_{0}^\pi\kappa(\theta)\ast \left(\sum_{m}\binom{k-2}{\frac{k-2}{2}-m}\left(i^{\frac{k-2}{2}+m}f_{\frac{k}{2}}+ (-i)^{\frac{k-2}{2}+m}f_{-\frac{k}{2}}\right)\otimes \left((-i)^{\frac{k-2}{2}-m}f_{\frac{k}{2}}+ i^{\frac{k-2}{2}-m}f_{-\frac{k}{2}}\right)\right)d \theta\\
&=&\sum_{m}\binom{k-2}{\frac{k-2}{2}-m}\int_{0}^\pi\left(i^{\frac{k-2}{2}+m}e^{ik\theta}f_{\frac{k}{2}}+ (-i)^{\frac{k-2}{2}+m}e^{-ik\theta}f_{-\frac{k}{2}}\right)\otimes \left((-i)^{\frac{k-2}{2}-m}e^{ik\theta}f_{\frac{k}{2}}+ i^{\frac{k-2}{2}-m}e^{-ik\theta}f_{-\frac{k}{2}}\right)d \theta\\
&=&(-1)^{\frac{k-2}{2}}\sum_{m}\binom{k-2}{\frac{k-2}{2}-m}\left(f_{\frac{k}{2}}\otimes f_{-\frac{k}{2}}+ f_{-\frac{k}{2}}\otimes f_{\frac{k}{2}}\right){\rm vol}(K_{\R,0},d\theta)=(2i)^{k-2}\left(f_{\frac{k}{2}}\otimes f_{-\frac{k}{2}}+ f_{-\frac{k}{2}}\otimes f_{\frac{k}{2}}\right){\rm vol}(K_{\R,0},d\theta).
\end{eqnarray*}

On the other hand,  the cup product $c_1^+\cup c_1^-$ is represented by the 2-cocycle
    \[
    c_2(\hat H,\tilde W)=c_1^+(\hat H)c_1^-(\tilde W)(\Upsilon)-c_1^+(\tilde W)c_1^-(\hat H)(\Upsilon).
    \]
   Moreover, for some coefficients $C(s)$, we have
    \begin{equation}\label{auxeqforprop}
    \kappa_1(x^{\frac{k-2}{2}-m}y^{\frac{k-2}{2}+m})=\left(\frac{x-y}{\sqrt{2}}\right)^{\frac{k-2}{2}-m}\left(\frac{x+y}{\sqrt{2}}\right)^{\frac{k-2}{2}+m}=\sum_s C(s)\cdot x^{\frac{k-2}{2}-s}y^{\frac{k-2}{2}+s}.
    \end{equation}
    Since by \eqref{relGK} we have that $c_1^\pm(\tilde W)=\kappa_1^{-1}c_1^\pm(\hat H)$, we deduce
\begin{eqnarray*}
    c_1^\pm(\tilde W)(\mu_m)&=&\kappa_1^{-1}\left(c_1^\pm(\hat H)(\kappa_1\mu_m)\right)=(k-1)\sum_s C(s)\kappa_1^{-1}\left((-i)^{\frac{k-2}{2}+s}f_{\frac{k}{2}}\pm i^{\frac{k-2}{2}+s}f_{-\frac{k}{2}}\right)\\
    &=&(k-1)\sum_s C(s)\left((-i)^{\frac{k-2}{2}+s}e^{-i\pi k/4}f_{\frac{k}{2}}\pm i^{\frac{k-2}{2}+s}e^{i\pi k/4}f_{-\frac{k}{2}}\right)\\
                            &=&(k-1)\left((-i)^{\frac{k}{2}}\left(\frac{1+i}{\sqrt{2}}\right)^{\frac{k-2}{2}-m}\left(\frac{1-i}{\sqrt{2}}\right)^{\frac{k-2}{2}+m}f_{\frac{k}{2}}\pm i^{\frac{k}{2}}\left(\frac{1-i}{\sqrt{2}}\right)^{\frac{k-2}{2}-m}\left(\frac{1+i}{\sqrt{2}}\right)^{\frac{k-2}{2}+m}f_{-\frac{k}{2}}\right)\\
                            &=&(k-1)\left((-i)^{m+\frac{k}{2}}f_{\frac{k}{2}}\pm i^{m+\frac{k}{2}}f_{-\frac{k}{2}}\right),
\end{eqnarray*}
where the third equality follows from \eqref{auxeqforprop}. We compute,
\begin{eqnarray*}
    c_1^+(\hat H)c_1^-(\tilde W)(\Upsilon)&=&c_1^+(\hat H)c_1^-(\tilde W)\left(\sum_{m}\binom{k-2}{\frac{k-2}{2}+m}(-1)^{\frac{k-2}{2}+m} \mu_m\otimes\mu_{-m}\right)\\
    &=&(k-1)^2\sum_{m}\binom{k-2}{\frac{k-2}{2}+m}\left(i^{\frac{k-2}{2}+m}f_{\frac{k}{2}}+ (-i)^{\frac{k-2}{2}+m}f_{-\frac{k}{2}}\right)\otimes\left((-i)^{-m+\frac{k}{2}}f_{\frac{k}{2}}- i^{-m+\frac{k}{2}}f_{-\frac{k}{2}}\right)\\
    &=&(k-1)^2\left(i^{-1}\sum_{m}\binom{k-2}{\frac{k-2}{2}+m}(-1)^{m}f_{\frac{k}{2}}\otimes f_{\frac{k}{2}}-i(-1)^{\frac{k-2}{2}}\sum_{m}\binom{k-2}{\frac{k-2}{2}+m}f_{\frac{k}{2}}\otimes f_{-\frac{k}{2}}-\right.\\
    &&\left.-i(-1)^{\frac{k-2}{2}}\sum_{m}\binom{k-2}{\frac{k-2}{2}+m}f_{-\frac{k}{2}}\otimes f_{\frac{k}{2}}-i\sum_{m}\binom{k-2}{\frac{k-2}{2}+m}(-1)^{m}f_{-\frac{k}{2}}\otimes f_{-\frac{k}{2}}\right)\\
    &=&-i(2i)^{k-2}(k-1)^2\left(f_{\frac{k}{2}}\otimes f_{-\frac{k}{2}}+f_{-\frac{k}{2}}\otimes f_{\frac{k}{2}}\right),
\end{eqnarray*}
and
\begin{eqnarray*}
    c_1^+(\tilde W)c_1^-(\hat H)(\Upsilon)&=&c_1^+(\tilde W)c_1^-(\hat H)\left(\sum_{m}\binom{k-2}{\frac{k-2}{2}+m}(-1)^{\frac{k-2}{2}-m} \mu_m\otimes\mu_{-m}\right)\\
    &=&(k-1)^2\sum_{m}\binom{k-2}{\frac{k-2}{2}+m}\left((-i)^{m+\frac{k}{2}}f_{\frac{k}{2}}+ i^{m+\frac{k}{2}}f_{-\frac{k}{2}}\right)\otimes\left(i^{\frac{k-2}{2}-m}f_{\frac{k}{2}}- (-i)^{\frac{k-2}{2}-m}f_{-\frac{k}{2}}\right)\\
    &=&(k-1)^2\left(i^{-1}\sum_{m}\binom{k-2}{\frac{k-2}{2}+m}(-1)^{m}f_{\frac{k}{2}}\otimes f_{\frac{k}{2}}+i(-1)^{\frac{k-2}{2}}\sum_{m}\binom{k-2}{\frac{k-2}{2}+m}f_{\frac{k}{2}}\otimes f_{-\frac{k}{2}}+\right.\\
    &&\left.+i(-1)^{\frac{k-2}{2}}\sum_{m}\binom{k-2}{\frac{k-2}{2}+m}f_{-\frac{k}{2}}\otimes f_{\frac{k}{2}}-i\sum_{m}\binom{k-2}{\frac{k-2}{2}+m}(-1)^{m}f_{-\frac{k}{2}}\otimes f_{-\frac{k}{2}}\right)\\
    &=&i(2i)^{k-2}(k-1)^2\left(f_{\frac{k}{2}}\otimes f_{-\frac{k}{2}}+f_{-\frac{k}{2}}\otimes f_{\frac{k}{2}}\right)
\end{eqnarray*}
Thus, we conclude
\[
    c_2(\hat H,\tilde W)=c_1^+(\hat H)c_1^-(\tilde W)(\Upsilon)-c_1^+(\tilde W)c_1^-(\hat H)(\Upsilon)=-(2i)^{k-1}(k-1)^2\left(f_{\frac{k}{2}}\otimes f_{-\frac{k}{2}}+f_{-\frac{k}{2}}\otimes f_{\frac{k}{2}}\right),
\]
and the result follows.
\end{proof}

\begin{remark}\label{remarkonhaarmeasurecocycle1}
Recall that any $X\in\mfg_\R$ induces an invariant derivation on $C^\infty(\PGL_2(\R),\C)$, hence, we can write $dX$ for the 1-form dual to such a derivation. Given any morphism $\varphi\in\Hom_{\mfg_\R,K_{\R,+}}(D(k)^{\otimes 2},C^\infty(\PGL_2(\R),\C))$, the 2-cocycle $\varphi(c_2)$ provides the differential 2-form
    \begin{eqnarray*}
    \varphi(c_2)(\hat H,\tilde W)\cdot d\hat H\wedge d\tilde W&=&\frac{-8i}{\pi}\int_{K_{\R,+}}\varphi\left(k\ast  \delta s_+(\Upsilon)\right)d W\wedge d\hat H\wedge d\tilde W =\frac{2i}{\pi}\int_0^\pi\kappa(\theta)\ast \varphi\left( \delta s_+(\Upsilon)\right)d^\times g,\\
    \end{eqnarray*}
    because ${\rm vol}(K_{\R,+},dW)={\rm vol}(K_{\R,+},d\theta)=\pi$, $d^\times g=y^{-2}dxdyd\theta$ and the action of $W$, $\hat H$ and $\tilde W$ on $C^\infty(\PGL_2(\R),\C)$ is given by (\cite[proposition 2.2.5]{Bump})
    \[
    \left(\begin{array}{c} W \\ \hat H\\\tilde W\end{array}\right)=\left(\begin{array}{ccc} 0&0&1 \\ -2y\sin 2\theta&2y\cos 2\theta&\sin 2\theta\\2y\cos 2\theta&2y\sin 2\theta&-\cos2\theta\end{array}\right)\left(\begin{array}{c} \frac{\partial}{\partial x} \\ \frac{\partial}{\partial y}\\\frac{\partial}{\partial\theta}\end{array}\right).
    \]
\end{remark}

\subsection{The cohomological $(\mfg,K)$-modules for $\PGL_2(\C)$}\label{GKC}
We write $K_\C$ for the maximal compact subgroup of $\PGL_2(\C)$ given by the image of 
\[
\SU(2):=\left\{\left(\begin{array}{cc}\alpha&\beta\\-\bar \beta&\bar \alpha\end{array}\right)\colon |\alpha|^2+|\beta|^2=1\right\}\subset \SL_2(\C).
\]
By \cite{Hall}, all irreducible representations of $\SU(2)$ are finite-dimensional $\C$-representations of the form $\cP(n):=\cP(n)_\C$, where $\cP(n)$ is viewed as a $\SU(2)$-representation by restricting the action of $\GL_2(\C)$ described in Section \ref{findimreps}. We recall also that we have an isomorphism $\cP(n)\simeq V(n)$, where $V(n):= V(n)_\C$.  Thus, the irreducible representations of $K_\C\subset\PGL_2(\C)$ are the representations $V(2n)$, where $n\in\N$. 

Similarly as in \eqref{eqstar},
any $g\in\GL_2(\C)$ admits a decomposition
\begin{equation}\label{eqstar2}
g=u\left(\begin{array}{cc}r&x\\&r^{-1}\end{array}\right)\kappa(\alpha,\beta),\qquad \text{with } u\in\C^\times\;,x\in\C\;,r\in\R^\times,    
\end{equation}
and where $\kappa(\alpha,\beta):=\left(\begin{array}{cc}\alpha&\beta\\-\bar \beta&\bar \alpha\end{array}\right)$. The Lie algebra of the real Lie group $\PGL_2(\C)$ is given by 
\[
\mfg_\C:={\rm Lie}(\PGL_2(\C))\simeq{\rm Lie}(\SL_2(\C))\simeq M_0:=\{g\in\M_2(\C)\mid \operatorname{Tr} g=0\}=\R \hat H\oplus\R\tilde W\oplus\R H\oplus\R \hat H_i\oplus\R W\oplus\R \tilde W_i,
\]
where 
\begin{eqnarray*}
&&\hat H:=\left(\begin{array}{cc}
       1  & 0 \\
        0 & -1
    \end{array}\right);\qquad \tilde W:=\left(\begin{array}{cc}
       0  & 1 \\
        1 & 0
    \end{array}\right);\qquad H:=\left(\begin{array}{cc}
       0  & -i \\
        i & 0
    \end{array}\right);\\
    &&\hat H_i:=\left(\begin{array}{cc}
       i  & 0 \\
        0 & -i
    \end{array}\right);\qquad W:=\left(\begin{array}{cc}
       0  & 1 \\
        -1 & 0
    \end{array}\right);\qquad \tilde W_i:=\left(\begin{array}{cc}
       0  & i \\
        i & 0
    \end{array}\right).
\end{eqnarray*}
Moreover, $\cK_\C:={\rm Lie}(K_\C)=\{X\in \M_2(\C)\mid  {\rm Tr}X=0, \, X^\ast=-X\}=\R \hat H_i\oplus\R W\oplus\R \tilde W_i$, where $X^\ast$ denotes complex conjugation of the transpose matrix.
\begin{remark}\label{relLiealg}
    We have the relations 
    \[
    [W_i,\hat H]=2H, \qquad [\hat H,W]=2\tilde W,\qquad [\hat H,\tilde W]=2W,\qquad [H,\hat H]=2D,\qquad [\tilde W,H]=2\hat H_i.
    \]
    Hence, for any Lie algebra representation of $\mfg_\C$, it is enough to control the action of $\cK_\C$ and $\hat H$. Moreover, we have the relations:
    \begin{eqnarray*}
        \kappa(\alpha,\beta)^{-1}\hat H\kappa(\alpha,\beta)&=&(|\alpha|^2-|\beta|^2)\hat H+2{\rm Re}(\bar \alpha\beta)\tilde W-2{\rm Im}(\bar\alpha\beta)H;\\
        \kappa(\alpha,\beta)^{-1}\tilde W\kappa(\alpha,\beta)&=&-2{\rm Re}(\alpha\beta)\hat H+{\rm Re}(\bar \alpha^2-\beta^2)\tilde W-{\rm Im}(\bar\alpha^2-\beta^2)H;\\
        \kappa(\alpha,\beta)^{-1}H\kappa(\alpha,\beta)&=&2{\rm Im}(\alpha\beta)\hat H+{\rm Im}(\bar \alpha^2+\beta^2)\tilde W+{\rm Re}(\bar\alpha^2+\beta^2)H.
    \end{eqnarray*}
    Thus, if we write $\kappa_1=\kappa(1/\sqrt{2},1/\sqrt{2}),\kappa_2=\kappa(1/\sqrt{2},-i/\sqrt{2})\in \SU(2)$, then we obtain that
    \[
    \tilde W=\kappa_1^{-1}\hat H\kappa_1=-\kappa_1\hat H\kappa_1^{-1};\qquad H=\kappa_2^{-1}\hat H\kappa_2=-\kappa_2\hat H\kappa_2^{-1};\qquad \kappa_1^{-1} H\kappa_1=H;\qquad \kappa_2^{-1}\tilde W\kappa_2=\tilde W.
    \]
    Hence, any $(\mfg_\C,K_\C)$-module is completely determined by the acion of $K_\C$ and $\hat H$.
\end{remark}

For any character $\chi:\C^\times\rightarrow\C^\times$,
let us consider the induced $\PGL_2(\C)$-representation
\[
\cB(\chi):=\left\{f:\GL_2(\C)\rightarrow\C:\;f\left(\left(\begin{array}{cc}t_1&x\\&t_2\end{array}\right)g\right)=\chi(t_1/t_2)\cdot f(g)\right\}.
\]
By \eqref{eqstar2} we have an isomorphism
$$\cB(\chi)\simeq\{f:\SU(2)\rightarrow\C:\;f(e^{i\theta}\alpha,e^{i\theta}\beta)=\chi(e^{2i\theta})\cdot f(\alpha,\beta)\};\qquad f(\alpha,\beta):=f(\kappa(\alpha,\beta)).$$
Thus, the $\SU(2)$-representation $\cB(\chi)$ is induced by the restriction of the character $\chi^2$ at $S^1$. If $\chi(e^{i\theta}) = e^{i\lambda\theta}$, by Frobenius reciprocity we have that
\begin{equation}\label{homsV}
\Hom_{\SU(2)}(V(2n),\cB(\chi))= \begin{cases}\C &\text{ if }|\lambda|\leq n\\
0&\mbox{otherwise.}
\end{cases}
\end{equation}
\begin{definition}\label{defvarphi}
Suppose that $\chi(e^{i\theta})=e^{i\lambda\theta}$. For $n\geq|\lambda|$ define $\varphi_n\in\Hom_{\SU(2)}(V(2n),\cB(\chi))$ to be the morphism given by
\begin{equation}\label{morfvarphn}
\varphi_n(\mu)(\alpha,\beta):=\mu\left(\left|\begin{array}{cc}\alpha&\beta\\x&y\end{array}\right|^{n+\lambda}\left|\begin{array}{cc}-\bar\beta&\bar\alpha\\x&y\end{array}\right|^{n-\lambda}\right), \qquad\mbox{for all }\mu\in V(2n).
\end{equation}
\end{definition}

Write $\cB(\chi,n)$ for the image of $V(2n)$ through $\varphi_n$. Hence the subspace
$\tilde \cB(\chi):=\bigoplus_{n\geq |\lambda|}\cB(\chi,n)\subseteq \cB(\chi)$ 
is a natural $(\mfg_\C,K_\C)$-module. Denote by $\Sigma$ the set of $\R$-isomorphisms $\sigma\colon \C\rightarrow\C$ and, for any $\underline{k}=(k_\sigma)_{\sigma\in\Sigma}\in (2\N_{>1})^2$, we consider $\chi_{\underline{k}}(t):=\prod_{\sigma\in\Sigma}\sigma(t)^{\frac{k_\sigma}{2}}$. By \cite[Proposition 4.18]{preprintsanti2}, we have a morphism of $\GL_2(\C)$-representations
\[
\rho:\cB(\chi_{\underline{k}})\longrightarrow V(\underline{k}-2):=\bigotimes_{\sigma\in\Sigma}V(k_\sigma-2); \qquad
 \rho(f)\left(\bigotimes_{\sigma\in\Sigma}P_\sigma\right)= \int_{S^3}f(\alpha,\beta)\left(\prod_{\sigma\in\Sigma} P_\sigma(-\sigma(\bar\beta),\sigma(\bar\alpha))\right)d(\alpha,\beta),
\]
where $\GL_2(\C)$ acts on each $V(k_\sigma-2)$ by means of $\sigma$.
It turns out that the subspace $D(\underline{k}):=\tilde \cB(\chi_{\underline{k}})\cap \ker(\rho)$ is the unique non-trivial sub-$(\mfg_\C,K_\C)$-module of $\tilde\cB(\chi_{\underline{k}})$. We have obtained an exact sequence of $(\mfg_\C,K_\C)$-modules:
\begin{equation}\label{exseq2}
0\longrightarrow D(\underline{k})=\bigoplus_{n>\frac{k_{\rm id}-2+k_c-2}{2}}\cB(\chi_{\underline{k}},n)\longrightarrow \tilde \cB(\chi_{\underline{k}})=\bigoplus_{n\geq \left|\frac{k_{\rm id}-k_c}{2}\right|}\cB(\chi_{\underline{k}},n)\longrightarrow V(\underline{k}-2)\longrightarrow 0,
\end{equation}
since $V(\underline{k}-2)\simeq\bigoplus_{\frac{k_{\rm id}-2+k_c-2}{2}\geq n\geq \left|\frac{k_{\rm id}-k_c}{2}\right|} V(2n)$, where ${\rm id},c\in\Sigma$ are the identity and complex conjugation, respectively.  
Similarly as in the real setting, $\rho$ admits a $K_\C$-equivariant section, namely, 
a $\SU(2)$-equivariant morphism
\[
s:V(\underline{k}-2)\longrightarrow \tilde\cB(\chi_{\underline{k}});\qquad \rho\circ s={\rm id}.
\]
In \cite[Lemma 4.21]{preprintsanti2} one can find an explicit description of $s$.

\subsubsection{Other models for $D(\underline k)$}\label{othermod}

By \cite[Theorem 6.2]{JL} the $(\mfg_\C,K_\C)$-module $D(\underline k)$ admits a model as an induced representation. Indeed, if we consider the character
$\hat\chi_{\underline k}:\C^\times\rightarrow\C^\times$, where $\hat\chi_{\underline k}(t):=t^{\frac{k_{\rm id}}{2}}(\bar t)^{\frac{2-k_{c}}{2}}$, then we have the isomorphism $D(\underline k)\simeq \tilde \cB(\hat\chi_{\underline k})$. Notice that 
\[
\hat\chi_{\underline k}(e^{i\theta})=e^{i\left(\frac{k_{\rm id}+k_{c}-2}{2}\right)\theta};\qquad \tilde \cB(\hat\chi_{\underline{k}})=\bigoplus_{n\geq \frac{k_{\rm id}+k_{c}-2}{2}}\cB(\hat\chi_{\underline{k}},n)=\bigoplus_{n> \frac{k_{\rm id}-2+k_{c}-2}{2}}\cB(\hat\chi_{\underline{k}},n)
\]
Thus, the decomposition of $\tilde \cB(\hat\chi_{\underline{k}})$ as a sum of $\SU(2)$-representations fits with that of $D(\underline k)$. As we have seen in Remark \ref{relLiealg}, to verify that both representations coincide, we need to check whether the action of $\hat H$ coincides. The following result describes the action of $\hat H$ for any induced representation.
\begin{proposition}\label{acthatHforvarphi}
    If $\chi(re^{i\theta})=r^{N_{\chi}}e^{i\lambda_\chi \theta}$ then the action of $\hat H$ on $\tilde \cB(\chi)$ is given by
\[
\hat H \varphi_n(\mu)=\lambda_{\chi}(N_{\chi}-1)\varphi_n(\mu_0)-(N_{\chi}+n)\varphi_{n+1}(\mu_1)+(n+\lambda_{\chi})(n-\lambda_{\chi})(n-N_{\chi}+1)\varphi_{n-1}(\mu_{-1}),
\]
where $\mu_0\in V(2n)$, $\mu_1\in V(2n+2)$ and $\mu_{-1}\in V(2n-2)$ are
\[
\mu_0(P):=\frac{1}{n(n+1)}\mu\left(
y\frac{\partial P}{\partial y}-x\frac{\partial P}{\partial x}\right),\quad\mu_1(P):=\frac{2}{(n+1)(2n+1)}\mu\left(\frac{\partial P}{\partial x\partial y}\right),\quad \mu_{-1}(P):=\frac{2}{n(2n+1)}\mu\left(xyP\right).
\] 
\end{proposition}
\begin{proof}
        If we consider $f:\PGL_2(\C)\rightarrow\C$ as a function with variables $s,r,\alpha,\beta$ by means of \eqref{eqstar2}, then
    \begin{eqnarray*}
    \hat H f(s,r,\alpha,\beta)&=&\frac{d}{dt}f(g\exp(t\hat H))\mid_{t=0}=\frac{d}{dt}f\left(\left(\begin{array}{cc}r&s\\&r^{-1}\end{array}\right)\left(\begin{array}{cc}\alpha&\beta\\-\bar \beta&\bar \alpha\end{array}\right)\left(\begin{array}{cc}
       e^{t}  & 0 \\
        0 & e^{-t}
    \end{array}\right)\right)\mid_{t=0}\\
    &=&\frac{d}{dt}f\left(\left(\begin{array}{cc}r&s\\&r^{-1}\end{array}\right)\left(\begin{array}{cc}e^t\alpha&e^{-t}\beta\\-e^t\bar \beta&e^{-t}\bar \alpha\end{array}\right)\right)\mid_{t=0}=\frac{d}{dt}f\left(\left(\begin{array}{cc}rR^{-1}&sR+rA\\&r^{-1}R\end{array}\right)\left(\begin{array}{cc}\frac{e^{-t}\alpha}{R}&\frac{e^{t}\beta}{R}\\\frac{-e^t\bar \beta}{R}&\frac{e^{-t}\bar \alpha}{R}\end{array}\right)\right)\mid_{t=0}\\
    &=&\frac{d}{dt}f(sR+rA,rR^{-1},e^{-t}\alpha R^{-1},e^{t}\beta R^{-1})\mid_{t=0}
    \end{eqnarray*}
    where $R^2=R(t)^2=e^{2t}|\beta|^2+e^{-2t}|\alpha|^2$ and $A=A(t)=\alpha\beta R^{-1}(e^{-2t}-e^{2t})$.
    Since $R(0)=1$, $R'(0)=|\beta|^2-|\alpha|^2$, $A(0)=0$ and $A'(0)=-4\alpha\beta$, we conclude
    \begin{equation}\label{acthatH}
        \hat H f=(s(|\beta|^2-|\alpha|^2)-4r\alpha\beta)\frac{\partial f}{\partial s}+(\bar s(|\beta|^2-|\alpha|^2)-4r\bar\alpha\bar\beta)\frac{\partial f}{\partial \bar s}-r(|\beta|^2-|\alpha|^2)\frac{\partial f}{\partial r}-2\alpha|\beta|^2\frac{\partial f}{\partial \alpha}-2\bar \alpha|\beta|^2\frac{\partial f}{\partial \bar\alpha}+2\beta|\alpha|^2\frac{\partial f}{\partial \beta}+2\bar\beta|\alpha|^2\frac{\partial f}{\partial \bar\beta}.
    \end{equation}

If we write $\chi(re^{i\theta})=r^Ne^{i\lambda\theta}$,
then we have by definition
\begin{equation}\label{generalvarphin}
\varphi_n(\mu)\left(\left(\begin{array}{cc}r&s\\&r^{-1}\end{array}\right)\kappa(\alpha,\beta)\right):=r^{2N}\cdot\mu\left(
P_{n+\lambda,n-\lambda}\right);\qquad P_{a,b}:=\left|\begin{array}{cc}\alpha&\beta\\x&y\end{array}\right|^{a}\left|\begin{array}{cc}-\bar\beta&\bar\alpha\\x&y\end{array}\right|^{b}.
\end{equation}
Hence, we obtain by \eqref{acthatH}
\begin{eqnarray*}
\hat H\varphi_n(\mu)(\alpha,\beta)&=&\mu\left(2N(|\alpha|^2-|\beta|^2)P_{n+\lambda,n-\lambda}-2(n+\lambda)\alpha|\beta|^2yP_{n-1+\lambda,n-\lambda}+2(n-\lambda)\bar\alpha|\beta|^2xP_{n+\lambda,n-1-\lambda}-\right.\\
&&\left.-2(n+\lambda)\beta|\alpha|^2xP_{n-1+\lambda,n-\lambda}-2(n-\lambda)\bar\beta|\alpha|^2yP_{n+\lambda,n-1-\lambda}\right)\\
&=&\mu\left(2N(|\alpha|^2-|\beta|^2)P_{n+\lambda,n-\lambda}+2(n+\lambda)(-\alpha|\beta|^2y-\beta|\alpha|^2x)P_{n-1+\lambda,n-\lambda}+\right.\\
&&\left.+2(n-\lambda)(\bar\alpha|\beta|^2x-\bar\beta|\alpha|^2y)P_{n+\lambda,n-1-\lambda}\right)\\
&=&\mu\left(2N(|\alpha|^2-|\beta|^2)P_{n+\lambda,n-\lambda}+(n+\lambda)(-\alpha y-\beta x+(\beta x-\alpha y)(|\beta|^2-|\alpha|^2))P_{n-1+\lambda,n-\lambda}+\right.\\
&&\left.+(n-\lambda)(\bar\alpha x-\bar\beta y+(\bar\alpha x+\bar\beta y)(|\beta|^2-|\alpha|^2))P_{n+\lambda,n-1-\lambda}\right)\\
&=&\mu\left((2N+2n)(|\alpha|^2-|\beta|^2)P_{n+\lambda,n-\lambda}+(n+\lambda)(-\alpha y-\beta x)P_{n-1+\lambda,n-\lambda}+\right.\\
&&\left.+(n-\lambda)(\bar\alpha x-\bar\beta y)P_{n+\lambda,n-1-\lambda}\right)\\
&=&\mu\left((2N+2n)(|\alpha|^2-|\beta|^2)P_{n+\lambda,n-\lambda}+(n+\lambda)2xyP_{n-1+\lambda,n-1-\lambda}+2\lambda(\bar\beta y-\bar\alpha x) P_{n+\lambda,n-1-\lambda}\right),
\end{eqnarray*}
where the last equality follows from $(-\beta x-\alpha y)P_{a-1,b}=2xyP_{a-1,b-1}+(\bar\beta y-\bar\alpha x) P_{a,b-1}$ deduced from the relation
\begin{equation}\label{keyequationPab}
    -\beta P_{a,b+1}=yP_{a,b}-\bar\alpha P_{a+1,b};\qquad -\alpha P_{a,b+1}=xP_{a,b}+\bar\beta P_{a+1,b}.
\end{equation}
We compute similarly
\begin{eqnarray*}
y\frac{\partial P_{a,b}}{\partial y}-x\frac{\partial P_{a,b}}{\partial x}&=&y(a\alpha P_{a-1,b}-b\bar\beta P_{a,b-1})-x(-a\beta P_{a-1,b}-b\bar\alpha P_{a,b-1})=a(y\alpha+x\beta)P_{a-1,b}+b(x\bar\alpha -y\bar\beta )P_{a,b-1}\\
&=&-2axyP_{a-1,b-1}-(a+b)(\bar\beta y-\bar\alpha x) P_{a,b-1}.
\end{eqnarray*}
Moreover, using the relations \eqref{keyequationPab}
one obtains
\begin{eqnarray*}
\frac{\partial^2 P_{a,b}}{\partial y\partial x}&=&-a\beta ((a-1)\alpha P_{a-2,b}-b\bar\beta P_{a-1,b-1})-b\bar\alpha (a\alpha P_{a-1,b-1}-(b-1)\bar\beta P_{a,b-2})\\
&=&-a (a-1)\beta\alpha P_{a-2,b}+ab(|\beta|^2 -|\alpha|^2) P_{a-1,b-1}+b(b-1)\bar\alpha\bar\beta P_{a,b-2}\\
&=&a (a-1)(y\alpha P_{a-2,b-1}-|\alpha|^2P_{a-1,b-1})+ab(|\beta|^2 -|\alpha|^2) P_{a-1,b-1}+b(b-1)(-|\alpha|^2 P_{a-1,b-1}-x\bar\alpha P_{a-1,b-2})\\
&=&\frac{(a+b)(a+b-1)}{2}(|\beta|^2 -|\alpha|^2) P_{a-1,b-1}-a(a-1)yxP_{a-2,b-2}-\frac{a(a-1)-b(b-1)}{2}(y\bar\beta -x\bar\alpha) P_{a-1,b-2}
\end{eqnarray*}
Thus,
\begin{eqnarray*}
\hat H\varphi_n(\mu)(\alpha,\beta)&=&\mu\left(\frac{-2(N+n)}{(n+1)(2n+1)}\frac{\partial^2 P_{n+1+\lambda,n+1-\lambda}}{\partial y\partial x}+\frac{\lambda(N-1)}{n(n+1)}\left(y\frac{\partial P_{n+\lambda,n-\lambda}}{\partial y}-x\frac{\partial P_{n+\lambda,n-\lambda}}{\partial x}\right)+\right.\\
&&\left.+\frac{2(n+\lambda)(n-\lambda)(n-N+1)}{n(2n+1)}xyP_{n-1+\lambda,n-1-\lambda}\right),
\end{eqnarray*}
and the result follows.
\end{proof}

Since we have
\[
N_{\chi_{\underline k}}=\frac{k_{\rm id}}{2}+\frac{k_c}{2},\qquad \lambda_{\chi_{\underline k}}=\frac{k_{\rm id}}{2}-\frac{k_c}{2},\qquad N_{\hat\chi_{\underline k}}=\frac{k_{\rm id}}{2}+\frac{2-k_c}{2}=\lambda_{\chi_{\underline k}}+1,\qquad \lambda_{\hat\chi_{\underline k}}=\frac{k_{\rm id}}{2}+\frac{k_c-2}{2}=N_{\chi_{\underline k}}-1,
\]
we obtain the isomorphism between $D(\underline k)$ and $\tilde \cB(\hat\chi_{\underline k})$ given by
\begin{equation}\label{isoBD}
\psi:\tilde \cB(\hat\chi_{\underline k})=\bigoplus_{n> \frac{k_{\rm id}-2+k_{c}-2}{2}}\cB(\hat\chi_{\underline{k}},n)\longrightarrow \bigoplus_{n> \frac{k_{\rm id}-2+k_{c}-2}{2}}\cB(\chi_{\underline{k}},n)= D(\underline k);\qquad \psi(\varphi_n(\mu))=\binom{n+\lambda_{\hat\chi_{\underline{k}}}}{N_{\hat\chi_{\underline{k}}}+n-1}\varphi_n(\mu).
\end{equation}
Indeed, one can check using the above proposition that such a morphism respects the action of $\hat H$.

\subsection{Explicit cohomology classes of $\mfg_\C$}\label{ExplCCC}

In this section we will describe the cocycles and cohomology classes involving the $(\mfg_\C,K_\C)$-module $D(\underline{k})$.  

\subsubsection{1-cocycles associated with $D(\underline{k})$} Note that in this complex setting $\mfg_\C/\cK_\C=\R \hat H\oplus\R\tilde W\oplus\R H$.
Moreover, by Remark \ref{relLiealg}, an homomorphism $\varphi_1\in \Hom_{K_\C}\left(\mfg_\C/\cK_\C,M\right)$ is characterized by the image of $\hat H$. 
    Take the $(\mfg_\C,K_\C)$-module $\Hom(V(\underline k-2),D({\underline k}))$. 
Let us recall the unique $K_\C$-equivariant section $s:V(\underline k-2)\rightarrow\tilde \cB(\chi_{\underline k})$ of the exact sequence \eqref{exseq2}. Hence, the 1-cocycle associated to the aforementioned exact sequence is the class of
\begin{equation}\label{defc1}
c_1\in \Hom_{K_\C}\left(\mfg_\C/\cK_\C,\Hom(V(\underline k-2),D({\underline k}))\right);\qquad c_1(X)(\mu)=\left(X(s\mu)-s(X\mu)\right).
\end{equation}
Let us consider similarly as in the real setting
\begin{equation}\label{secdefds}
\delta s:=c_1(D)\in\Hom_{K_\C}\left(\mfg_\C/\cK_\C,\Hom(V(\underline{k}-2),D(\underline{k}))\right);\qquad  D:=\left(\begin{array}{cc}
    1&0\\0&0
\end{array}\right)
\end{equation}
In \cite[Proposition 4.24]{preprintsanti2} the morphisms $\delta s$ are computed explicitly. Since $2D= \hat H$ in $\mfg_\C$,
we can use such a result to obtain $c_1(\hat H)$ that characterizes $c_1$: Notice that as a $\SU(2)$-representation $V(\underline{k}-2)\simeq \bigoplus_{|\lambda_{\chi_{\underline k}}|\leq n\leq N_{\chi_{\underline k}}-2}V(2n)$. Hence, we can consider the $K_\C$-equivariant morphism for $|\lambda_{\chi_{\underline k}}|\leq n\leq N_{\chi_{\underline k}}-2$
\begin{eqnarray*}
t_n:V(\underline{k}-2)=V(k_{\rm id}-2)\otimes V(k_c-2)\longrightarrow \cP(2n)\simeq V(2n);\qquad t_n(\mu_{\rm id}\otimes\mu_c)=\mu_{\rm id}\mu_c(\Delta_n);\\
\Delta_n(X_{\rm id},Y_{\rm id},X_{c},Y_{c},x,y):=\left|\begin{array}{cc}X_{\rm id}&Y_{\rm id}\\x&y\end{array}\right|^{r_1}\left|\begin{array}{cc}-Y_{c}&X_{c}\\x&y\end{array}\right|^{r_2}\left|\begin{array}{cc}X_{\rm id}&Y_{\rm id}\\-Y_c&X_c\end{array}\right|^{r_3},\qquad\qquad
\end{eqnarray*}
where $r_1:=n+\lambda_{\chi_{\underline k}}$, $r_2:=n-\lambda_{\chi_{\underline k}}$ and $r_3:=N_{\chi_{\underline k}}-2-n$. Then, for any $\mu\in V(\underline{k}-2)$, we have that 
\begin{equation}\label{descc1C}
c_1(\hat H)(\mu)=2\delta s(\mu)=-4\binom{2N_{\chi_{\underline k}}-4}{k_{\rm id}-2}\varphi_{N_{\chi_{\underline k}}-1}(t_{N_{\chi_{\underline k}}-2}(\mu)^*)
\end{equation}
where $t_{N_{\chi_{\underline k}}-2}(\underline{\mu})^*\in V(2N_{\chi_{\underline k}}-2)$ is given by
$t_{N_{\chi_{\underline k}}-2}(\underline{\mu})^*(P):=t_{N_{\chi_{\underline k}}-2}(\underline{\mu})\left(\frac{\partial^2 P}{\partial x\partial y}\right)$.
\begin{remark}\label{Altdescc1}
    It is easier to understand the class $c_1\in H^1((\mfg_\C,K_\C),\Hom(V(\underline k-2),D({\underline k})))$ once $D(\underline k)$ is described as an induced representation (see \S \ref{othermod}). Indeed, we know that $D(\underline k)\simeq \tilde \cB(\hat\chi_{\underline k})$, hence,
    \begin{equation}\label{ShapiroGK}
    H^m((\mfg_\C,K_\C),\Hom(V(\underline k-2),D({\underline k})))\simeq H^m((\cB,K_B),\Hom(V(\underline k-2),\hat\chi_{\underline k})),
    \end{equation}
    where $B\subset\PGL_2(\C)$ is the usual Borel subgroup, $\cB={\rm Lie}(B)$, $K_B=K\cap B$, and $\cK_B={\rm Lie}(K_B)$. Note that 
    \[
    \Hom(V(\underline k-2),\hat\chi_{\underline k})=\cP(\underline k-2)(\hat\chi_{\underline k})=\bigoplus_{\underline n}\C x^{\underline n}y^{\underline k-2-\underline n};\qquad \underline n=(n_{\rm id},n_c);\quad x=(x_{\rm id},x_c);\quad y=(y_{\rm id},y_c),
    \]
    and each subspace $\C x^{\underline n}y^{\underline k-2-\underline n}$ is an eigenspace for the action of the matrices $\kappa_\alpha:=\kappa(\alpha,0)\in K_B$, for $|\alpha|^2=1$, with eigenvalues $\alpha^{2n_{\rm id}-2n_c+2k_c-2}$. Moreover, $\cB/\cK_B\simeq \mfg_\C/\cK_\C$ and it is generated by 
\begin{eqnarray*}
&&\hat H:=\left(\begin{array}{cc}
       1  & 0 \\
        0 & -1
    \end{array}\right), \qquad N_1:=\frac{\tilde W+W}{2}=\left(\begin{array}{cc}
       0  & 1 \\
        0 & 0
    \end{array}\right), \qquad N_2:=\frac{D-H}{2}=\left(\begin{array}{cc}
       0  & i \\
        0 & 0
    \end{array}\right).
\end{eqnarray*}
Since $\kappa_\alpha\hat H\kappa_\alpha^{-1}=\hat H$, for any $\varphi\in Z^1((\mfg_B,K_B),\cP(\underline k-2)(\hat\chi_{\underline k}))$ we have 
$\varphi(\hat H)=0$ because $2n_{\rm id}-2n_c+2k_c-2\geq 2$ and no eigenspaces with eigenvalue 1 appear in $\cP(\underline k-2)(\hat\chi_{\underline k})$. Similarly,  
\begin{equation}\label{actionKBNs}
\left(\begin{array}{c}
    \kappa_\alpha N_1\kappa_\alpha^{-1} \\
     \kappa_\alpha N_2\kappa_\alpha^{-1}
\end{array}\right)=\left(\begin{array}{cc}
    {\rm Re}(\alpha^2) & {\rm Im}(\alpha^2) \\
     -{\rm Im}(\alpha^2)&{\rm Re}(\alpha^2) 
\end{array}\right)\left(\begin{array}{c}
     N_1  \\
     N_2 
\end{array}\right)
\mbox{ and }
\left(\begin{array}{c}
    \kappa_\alpha\varphi( N_1) \\
     \kappa_\alpha\varphi(  N_2)
\end{array}\right)=\left(\begin{array}{cc}
    {\rm Re}(\alpha^2) & {\rm Im}(\alpha^2) \\
     -{\rm Im}(\alpha^2)&{\rm Re}(\alpha^2) 
\end{array}\right)\left(\begin{array}{c}
     \varphi( N_1)  \\
     \varphi( N_2) 
\end{array}\right).
\end{equation}
This implies that $\varphi(N_1)-i\varphi(N_2)$ is an eigenvector with eigenvalue $\alpha^2$ and  $\varphi(N_2)-i\varphi(N_1)$ is an eigenvector with eigenvalue $\alpha^{-2}$. Since there are no eigenvectors with eigenvalue $\alpha^{-2}$, we conclude $\varphi(N_2)=i\varphi(N_1)$ and $\varphi(N_1)\in \C y_{\rm id}^{k_{\rm id}-2}x_c^{k_c-2}$. We have obtained that $Z^1((\cB,K_B),\cP(\underline k-2)(\hat\chi_{\underline k}) )$ is one dimensional. In fact, it is easy to check that $\Hom_{K_B}\left(\cB/\cK_B,\cP(\underline k-2)(\hat\chi_{\underline k}) \right)=Z^1((\cB,K_B),\cP(\underline k-2)(\hat\chi_{\underline k}) )\simeq \C$.

Note that the isomorphism \eqref{ShapiroGK} is provided by 
\begin{eqnarray}\label{Phieq}
\Phi:\Hom_{K_\C}\left(\bigwedge^m\mfg_\C/\cK_\C,\Hom(V(\underline k-2),D({\underline k})) \right)&\stackrel{\simeq}{\longrightarrow} &\Hom_{K_B}\left(\bigwedge^m\cB/\cK_B,\cP(\underline k-2)(\hat\chi_{\underline k}) \right),\\
\mu(\Phi(\varphi)(X))&=&\varphi(X)(\mu)(\kappa(1,0)),
\end{eqnarray}
for all $X\in\bigwedge^m\cB/\cK_B$.
Hence, to obtain $(\Phi c_1)(N_1)$ we use the isomorphism \eqref{isoBD} and the  explicit description of $c_1(\hat H)$ of \eqref{descc1C}: (recall that $\tilde W=\kappa_1^{-1}\hat H\kappa_1$)
\begin{eqnarray*}
    (y_{\rm id}^{k_{\rm id}-2}x_c^{k_c-2})^\vee (\Phi c_1(N_1))&=&c_1(N_1)((y_{\rm id}^{k_{\rm id}-2}x_c^{k_c-2})^\vee)(\kappa(1,0))=\frac{1}{2}c_1(\kappa_1^{-1}\hat H\kappa_1)((y_{\rm id}^{k_{\rm id}-2}x_c^{k_c-2})^\vee)(\kappa(1,0))\\
    &=&\frac{1}{2}c_1(\hat H)(\kappa_1(y_{\rm id}^{k_{\rm id}-2}x_c^{k_c-2})^\vee)(\kappa_1^{-1})\\
    &=&-2\frac{(k_{\rm id}-1)(k_c-1)}{(2N_{\chi_{\underline k}}-2)(2N_{\chi_{\underline k}}-3)}\varphi_{N_{\chi_{\underline k}}-1}(t_{N_{\chi_{\underline k}}-2}(\kappa_1(y_{\rm id}^{k_{\rm id}-2}x_c^{k_c-2})^\vee)^*)(\kappa_1^{-1})\\
    &=&-2\frac{(k_{\rm id}-1)(k_c-1)}{(2N_{\chi_{\underline k}}-2)(2N_{\chi_{\underline k}}-3)}\kappa_1t_{N_{\chi_{\underline k}}-2}((y_{\rm id}^{k_{\rm id}-2}x_c^{k_c-2})^\vee)^*\left(\left(\frac{y+x}{\sqrt{2}}\right)^{2N_{\chi_{\underline k}}-2}\right)\\
    &=&-(k_{\rm id}-1)(k_c-1)t_{N_{\chi_{\underline k}}-2}((y_{\rm id}^{k_{\rm id}-2}x_c^{k_c-2})^\vee)\left(y^{2N_{\chi_{\underline k}}-4}\right).
\end{eqnarray*}
Recall that, by definition,
\[
t_{N_{\chi_{\underline k}}-2}:V(\underline{k}-2)\longrightarrow \cP(2({N_{\chi_{\underline k}}-2}))\simeq V(2({N_{\chi_{\underline k}}-2})),\qquad t_{N_{\chi_{\underline k}}-2}(\mu)=\mu\left(\left|\begin{array}{cc}x_{\rm id}&y_{\rm id}\\x&y\end{array}\right|^{k_{\rm id}-2}\left|\begin{array}{cc}-y_{c}&x_{c}\\x&y\end{array}\right|^{k_c-2}\right).
\]
Hence, $t_{N_{\chi_{\underline k}}-2}((y_{\rm id}^{k_{\rm id}-2}x_c^{k_c-2})^\vee)=x^{2N_{\chi_{\underline k}}-4}\in\cP(2N_{\chi_{\underline k}}-4)$ corresponds to 
$t_{N_{\chi_{\underline k}}-2}((y_{\rm id}^{k_{\rm id}-2}x_c^{k_c-2})^\vee)=(y^{2N_{\chi_{\underline k}}-4})^\vee\in V(2N_{\chi_{\underline k}}-4)$.
We conclude that
\[
\Phi c_1(N_1)=(k_{\rm id}-1)(1-k_c)y_{\rm id}^{k_{\rm id}-2}x_c^{k_c-2}=-i\Phi c_1(N_2);\qquad \Phi c_1(\hat H)=0.
\]
\end{remark}

\subsubsection{2-cocycles associated with $D(\underline{k})$} Given a $(\mfg_\C,K_\C)$-module $M$, we want to characterize the space of 2-cocycles $Z^2((\mfg_\C,K_\C),M)$ with values in $M$ defined as in \eqref{def2cocycles}.
Similarly as before, any $\varphi_2\in Z^2((\mfg_\C,K_\C),M)$ is completely determined by the images $\varphi_2(\hat H,\tilde W)$, $\varphi_2(\hat H,H)$ and $\varphi_2(H,\tilde W)$. Moreover, using $K_\C$-equivariance and Remark \ref{relLiealg} we find that
\begin{eqnarray*}
   \varphi_2(\tilde W,H)&=&\varphi_2(\kappa_1^{-1}\hat H\kappa_1,H)=\kappa_1^{-1}\varphi_2(\hat H,\kappa_1H\kappa_1^{-1})=\kappa_1^{-1}\varphi_2(\hat H,H),\\
   \varphi_2(\tilde W,H)&=&\varphi_2(\tilde W,\kappa_2^{-1}\hat H\kappa_2,H)=\kappa_2^{-1}\varphi_2(\kappa_2\tilde W\kappa_2^{-1},\hat H)=\kappa_2^{-1}\varphi_2(\tilde W,\hat H).
\end{eqnarray*}
Thus, $\varphi_2$ is completely determined by the value $\varphi_2(\tilde W,H)$. 
\begin{proposition}\label{propdefc2}
    The element
$c_2\in \Hom_{K_\C}\left(\bigwedge^2\mfg_\C/\cK_\C,\Hom(V(\underline k-2),D({\underline k}))\right)$ given by
\[
c_2(\tilde W,H)(\mu):=\left(\hat H(s\mu)-s(\hat H\mu)\right)
\]
defines a 2-cocycle  whose class in $H^2\left((\mfg_\C,K_\C),\Hom(V(\underline k-2),D({\underline k}))\right)$ is non-trivial.
\end{proposition}
\begin{proof}
    We have seen in Remark \ref{Altdescc1} that it is more convenient to work with the description $D(\underline k)\simeq \tilde \cB(\hat\chi_{\underline k})$ as an induced representation because $H^2((\mfg_\C,K_\C),\Hom(V(\underline k-2),D({\underline k})))=H^2((\cB,K_B),\cP(\underline k-2)(\hat\chi_{\underline k}))$, and the corresponding cocycles $\varphi_2\in \Hom_{K_B}\left(\bigwedge^2\cB/\cK_B,\cP(\underline k-2)(\hat\chi_{\underline k}) \right)$ are easier to describe.
    By equation \eqref{actionKBNs} we have that, for any $\kappa_\alpha=\kappa(\alpha,0)\in K_B$,
    \[
    \kappa_\alpha\varphi_2(N_1,N_2)=\varphi_2(\kappa_\alpha N_1\kappa_\alpha^{-1},\kappa_\alpha N_1\kappa_\alpha^{-1})=({\rm Re}(\alpha^2)^2+{\rm Im}(\alpha^2)^2)\varphi_2(N_1,N_2)=\varphi_2(N_1,N_2),
    \]
    and since no eigenspaces with eigenvalue 1 appear in $\cP(\underline k-2)(\hat\chi_{\underline k})$ we conclude $\varphi(N_1,N_2)=0$. Similarly,
    \[\left(\begin{array}{c}
    \kappa_\alpha\varphi_2(\hat H, N_1) \\
     \kappa_\alpha\varphi_2(\hat H,  N_2)
\end{array}\right)=\left(\begin{array}{cc}
    {\rm Re}(\alpha^2) & {\rm Im}(\alpha^2) \\
     -{\rm Im}(\alpha^2)&{\rm Re}(\alpha^2) 
\end{array}\right)\left(\begin{array}{c}
     \varphi_2(\hat H, N_1)  \\
     \varphi_2(\hat H, N_2) 
\end{array}\right).
\]
This implies that $\varphi_2(\hat H,N_1)-i\varphi_2(\hat H,N_2)$ is an eigenvector with eigenvalue $\alpha^2$ and  $\varphi_2(\hat H,N_2)-i\varphi_2(\hat H,N_1)$ is an eigenvector with eigenvalue $\alpha^{-2}$. Since there are no eigenvectors with eigenvalue $\alpha^{-2}$ in $\cP(\underline k-2)(\hat\chi_{\underline k})$, we conclude $\varphi_2(\hat H,N_2)=i\varphi_2(\hat H,N_1)$ and $\varphi_2(\hat H,N_1)\in \C y_{\rm id}^{k_{\rm id}-2}x_c^{k_c-2}$. Since $[\hat H,N_1]=2N_1$, $[\hat H,N_2]=2N_2$, $[N_1,N_2]=0$, and $N_i$ acts trivially on  $\cP(\underline k-2)(\hat\chi_{\underline k})$, we obtain
\[
d\varphi_2(\hat H,N_1,N_2)=\hat H\varphi_2(N_1,N_2)-N_1\varphi_2(\hat H,N_2)+N_2\varphi_2(\hat H,N_1)-4\varphi_2(N_1,N_2)=0,
\]
and, therefore,
\[
\Hom_{K_B}\left(\bigwedge^2\cB/\cK_B,\cP(\underline k-2)(\hat\chi_{\underline k}) \right)=Z^2((\cB,K_B),\cP(\underline k-2)(\hat\chi_{\underline k}) )\simeq\C.
\]
Moreover, there are no coboundaries because $\Hom_{K_B}(\cB/\cK_B,\cP(\underline k-2)(\hat\chi_{\underline k}))=Z^1((\cB,K_B),\cP(\underline k-2)(\hat\chi_{\underline k}) )$.

It remains to check that, if $\Phi$ is the identification of \eqref{Phieq}, $\Phi(c_2)\in\Hom_{K_B}\left(\bigwedge^2\cB/\cK_B,\cP(\underline k-2)(\hat\chi_{\underline k}) \right)$ defines a non-trivial homomorphism. Indeed, since $c_2(\tilde W, H)=c_1(\hat H)$, for all $\mu\in V(\underline k-2)$
\begin{eqnarray*}
\mu(\Phi(c_2)(\hat H,N_1))&=&\frac{1}{2}c_2(\hat H,\tilde W)(\mu)(\kappa(1,0))=\frac{1}{2}\kappa_2 c_2( H,\tilde W)(\mu)(\kappa(1,0))=\frac{-1}{2}\kappa_2 c_1(\hat H)(\mu)(\kappa(1,0))\\
&=&\frac{1}{2}c_1(H)(\mu)(\kappa(1,0))=-\mu(\Phi(c_1)(N_2))=\mu\left(i(k_{\rm id}-1)(k_c-1)y_{\rm id}^{k_{{\rm id}}-2}x_c^{k_c-2}\right).
\end{eqnarray*}
Hence, $\Phi(c_2)(\hat H,N_1)=i(k_{\rm id}-1)(k_c-1)y_{\rm id}^{k_{{\rm id}}-2}x_c^{k_c-2}$ and the result follows.
\end{proof}

\subsubsection{Prasanna--Venkatesh action}
Let us consider the Cartan involution $\theta(X)=-X^\ast$, and the fundamental Cartan subalgebra 
$\left\langle \hat H,\hat H_i\right\rangle\subset\mfg_\C$.
It is clear that $\mathfrak{a}=\langle\hat H\rangle\subset \mfg_\C$ is the $-1$ eigenspace for $\theta$. We write $\mathfrak{a}^\ast$ for its dual. Imitating \cite{PV} in case of trivial coefficients, we aim to define an action of $\mathfrak{a}^\ast$ on $H^\ast((\mfg_\C,K_\C),\Hom(V(\underline k-2),D({\underline k})))$: Indeed, we know that 
\[
H^i((\mfg_\C,K_\C),\Hom(V(\underline k-2),D({\underline k})))\simeq H^i((\cB,K_B),\cP(\underline k-2)(\hat\chi_{\underline k})).
\]
Moreover, we have a decomposition of $\cB/\cK_B=\mathfrak{a}\oplus \mathfrak{u}$, where $\mathfrak{u}=\langle N_1,N_2\rangle$ comes from the unipotent radical of $\cB$. Such a decomposition provides an embedding $\mathfrak{a}^\ast\subset (\cB/\cK_B)^\ast$. Thus, we can consider the contraction
\[
\lrcorner:\mathfrak{a}^\ast\times\bigwedge^m(\cB/\cK_B)\longrightarrow\bigwedge^{m-1}(\cB/\cK_B);\qquad X\lrcorner(x_1\wedge\cdots\wedge x_m)=\sum_{j=1}^m(-1)^{j-1}\langle X,x_i\rangle x_1\wedge\cdots x_{j-1}\wedge x_{j+1}\wedge\cdots x_m.
\]
Thus, the action of $X\in\mathfrak{a}^\ast$ on the $(\mfg,K)$-cohomology 
\[
X:H^1((\mfg_\C,K_\C),\Hom(V(\underline k-2),D({\underline k})))\longrightarrow H^2((\mfg_\C,K_\C),\Hom(V(\underline k-2),D({\underline k})))
\]
is induced by the action on cocycles
\[
Xf\in \Hom_{K_B}\left(\bigwedge^2\cB/\cK_B,\cP(\underline k-2)(\hat\chi_{\underline k})\right);\qquad Xf(W)=f(X\lrcorner W).
\]
\begin{lemma}\label{LemPraVen}
    Let $\hat H^\ast\in\mathfrak{a}^\ast$ be the functional such that $\hat H^\ast(\hat H)=1$. Then we have 
    \[
    (\hat H^\ast c_1)=ic_2\in H^2((\mfg_\C,K_\C),\Hom(V(\underline k-2),D({\underline k}))).
    \]
\end{lemma}
\begin{proof}
In Remark \ref{Altdescc1} and in the proof of Proposition \ref{propdefc2}, we have explicit descriptions of $\Phi c_1$ and $\Phi c_2$, where $\Phi$ is the morphism \eqref{Phieq}. Thus, we can compare $\Phi(\hat H^\ast c_1)$ and $\Phi c_2$: By definition,
\begin{eqnarray*}
    \Phi(\hat H^\ast c_1)(N_1, N_2)&=&\Phi(c_1)(\hat H^\ast\lrcorner (N_1\wedge N_2))=\langle \hat H^\ast,N_1\rangle \Phi(c_1)(N_2)-\langle \hat H^\ast,N_2\rangle \Phi(c_1)(N_1)=0;\\
    \Phi(\hat H^\ast c_1)(N_1, \hat H)&=&\Phi(c_1)(\hat H^\ast\lrcorner (N_1\wedge \hat H))=\langle \hat H^\ast,N_1\rangle \Phi(c_1)(\hat H)-\langle \hat H^\ast,\hat H\rangle \Phi(c_1)(N_1)=-\Phi(c_1)(N_1);\\
     \Phi(\hat H^\ast c_1)(N_2, \hat H)&=&\Phi(c_1)(\hat H^\ast\lrcorner (N_2\wedge \hat H))=\langle \hat H^\ast,N_2\rangle \Phi(c_1)(\hat H)-\langle \hat H^\ast,\hat H\rangle \Phi(c_1)(N_2)=-i\Phi(c_1)(N_1).
    \end{eqnarray*}
    On the other side,
    \[
    \Phi c_2(N_1, N_2)=0;\qquad
    \Phi c_2(N_1, \hat H)=\Phi c_1(N_2)=i\Phi c_1(N_1);\qquad
     \Phi c_2(N_2, \hat H)=i\Phi c_2(N_1, \hat H)=-\Phi c_1(N_1).
    \]
    Hence, the result follows.
\end{proof}

\subsubsection{3-cocycles associated with $D(\underline{k})$}

In the previous sections we have constructed classes
\[
c_1\in H^1((\mfg_\C,K_\C),M),\qquad c_2\in H^2((\mfg_\C,K_\C),M),\qquad M:=\Hom(V(\underline k-2),D({\underline k})).
\]
Note that we have a natural $(\mfg_\C,K_\C)$-equivariant bilinear pairing
\[
M\times M\longrightarrow D({\underline k})\otimes D({\underline k});\qquad (\varphi_1,\varphi_2)\mapsto\varphi_1\varphi_2(\underline{\Upsilon});\quad \underline{\Upsilon}=\left|\begin{array}{cc}
    \underline x_1 & \underline y_1 \\
    \underline x_2 & \underline y_2
\end{array}\right|^{\underline k-2}\in\cP(\underline{k}-2)^{\otimes 2}\simeq V(\underline{k}-2)^{\otimes 2}.
\]
It is $(\mfg_\C,K_\C)$-equivariant because $\underline{\Upsilon}$ is $\GL_2(\C)$-invariant. 
If we recall the morphism $\delta s$ of \eqref{secdefds}, then the following result characterizes the cup-product $c_1\cup c_2$ with respect to the above pairing:
\begin{proposition}\label{propcupprod2}
    The cup-product $(c_1\cup c_2)\in H^3((\mfg_\C,K_\C),D({\underline k})^{\otimes 2})$ with respect to $(\;,\;)$ is provided by the 3-cocycle 
    \[
    c_3(\hat H,\tilde W,H)=\frac{12}{{\rm vol}(K_\C)}\int_{K_\C}k\ast \delta s(\underline{\Upsilon})d k\qquad c_3\in Z^3((\mfg_\C,K_\C),D({\underline k})^{\otimes 2})=\Hom_{K_\C}\left(\bigwedge^3\mfg_\C/\cK_\C,D({\underline k})^{\otimes 2}\right),
    \]
    for any choice of a Haar measure $dk$ of $K_\C$.
\end{proposition}
\begin{proof}
    The cup product $c_1\cup c_2$ is represented by the 3-cocycle
    \[
    c_3(\hat H,\tilde W,H)=c_1(\hat H)c_2(\tilde W,H)(\underline{\Upsilon})-c_1(\tilde W)c_2(\hat H,H)(\underline{\Upsilon})+c_1(H)c_2(\hat H,\tilde W)(\underline{\Upsilon}).
    \]
    By definition $c_1(\hat H)=c_2(\tilde W,H)=2\delta s$. Moreover, by Remark \ref{relLiealg}, we have that 
    \begin{eqnarray*}
        c_2(\hat H,H)=\kappa_1 c_2(\tilde W,H);&\qquad& c_2(\tilde W,\hat H)=\kappa_2 c_2(\tilde W,H)\\
        c_1(\tilde W)=-\kappa_1c_1(\hat H);&\qquad& c_1(H)=-\kappa_2c_1(\hat H).
    \end{eqnarray*}
Thus, by the $\PGL_2(\C)$ equivariance of $\underline\Upsilon$,
\begin{eqnarray*}
    c_1(\tilde W)c_2(\hat H,H)(\underline{\Upsilon})&=&-\kappa_1\left(c_1(\hat H)c_2(\tilde W,H)(\kappa_1^{-1}\underline{\Upsilon})\right)=-\kappa_1\left(c_1(\hat H)c_2(\tilde W,H)(\underline{\Upsilon})\right)=-4\kappa_1(\delta s(\underline{\Upsilon}));\\
    c_1(H)c_2(\hat H,\tilde W)(\underline{\Upsilon})&=&\kappa_2\left(c_1(\hat H)c_2(\tilde W,H)(\kappa_2^{-1}\underline{\Upsilon})\right)=\kappa_2\left(c_1(\hat H)c_2(\tilde W,H)(\underline{\Upsilon})\right)=4\kappa_2(\delta s(\underline{\Upsilon})).
\end{eqnarray*}
 On the other hand, we know by the properties of $\hat H$ exposed in Remark \ref{relLiealg} that, for any $\mu\in V(\underline k-2)$,
\begin{eqnarray*}
\kappa(\alpha,\beta)c_1(\hat H)(\mu)&=&c_1\left(\kappa(\alpha,\beta)\hat H\kappa(\alpha,\beta)^{-1}\right)(\kappa(\alpha,\beta) \mu)=c_1\left((|\alpha|^2-|\beta|^2)\hat H-2{\rm Re}(\alpha\beta)\tilde W+2{\rm Im}(\alpha\beta)H\right)(\kappa(\alpha,\beta) \mu)\\
&=&(|\alpha|^2-|\beta|^2)c_1(\hat H)(\kappa(\alpha,\beta) \mu)-2{\rm Re}(\alpha\beta)c_1(\tilde W)(\kappa(\alpha,\beta) \mu)+2{\rm Im}(\alpha\beta)c_1(H)(\kappa(\alpha,\beta) \mu)\\
&=&(|\alpha|^2-|\beta|^2)c_1(\hat H)(\kappa(\alpha,\beta) \mu)+2{\rm Re}(\alpha\beta)\kappa_1c_1(\hat H)(\kappa_1^{-1}\kappa(\alpha,\beta) \mu)-2{\rm Im}(\alpha\beta)\kappa_2c_1(\hat H)(\kappa_2^{-1}\kappa(\alpha,\beta) \mu).
\end{eqnarray*}
By \cite[lemma 4.15]{preprintsanti2}, for any $n_1,n_2,m_1,m_2\in\N$,
\[
\frac{1}{{\rm vol}(K_\C)}\int_{K_\C}\alpha^{n_1}\beta^{m_1}\bar \alpha^{n_2}\bar \beta^{m_2}dk=\begin{cases}(n_1+m_1+1)^{-1}\binom{n_1+m_1}{n_1}^{-1}&\text{ if }n_1=n_2,\;m_1=m_2,\\
0&\mbox{otherwise.} \end{cases}
\]
Thus, the functions $(|\alpha|^2-|\beta|^2)$, $2{\rm Re}(\alpha\beta)$, $2{\rm Im}(\alpha\beta)$ are orthogonal with respect to the pairing given by $dk$. Hence, we compute using the $K_\C$-invariance of $\underline \Upsilon$ plus the fact that $c_1(\hat H)=2\delta s$,
\begin{eqnarray*}
\frac{1}{{\rm vol}(K_\C)}\int_{K_\C}\kappa(\alpha,\beta)\delta s(\underline{\Upsilon})dk&=&\frac{1}{{\rm vol}(K_\C)}\int_{K_\C}\left((|\alpha|^2-|\beta|^2)^2\delta s(\underline{\Upsilon})+4{\rm Re}(\alpha\beta)^2\kappa_1\delta s(\underline{\Upsilon})+4{\rm Im}(\alpha\beta)^2\kappa_2\delta s(\underline{\Upsilon})\right)d(\alpha,\beta)\\
&=&\frac{1}{3}\delta s(\underline{\Upsilon})+\frac{1}{3}\kappa_1\delta s(\underline{\Upsilon})+\frac{1}{3}\kappa_2\delta s(\underline{\Upsilon})=\frac{1}{12}c_3(\hat H,\tilde W,H),
\end{eqnarray*}
and the result follows.
\end{proof}

\begin{remark}\label{remarkonhaarmeasurecocycle2}
    Similarly as in Remark \ref{remarkonhaarmeasurecocycle1}, given any morphism $\varphi\in\Hom_{\mfg_\C,K_\C}(D(\underline k)^{\otimes 2},C^\infty(\PGL_2(\C),\C))$, the 3-cocycle $\varphi(c_3)$ provides the differential 3-form 
    \begin{eqnarray*}
    \varphi(c_3)(\hat H,\tilde W,H)\cdot d\hat H\wedge d\tilde W \wedge dH&=&\frac{12}{{\rm vol}(K_\C)}\int_{K_\C}\varphi\left(k\ast \delta s(\underline\Upsilon)\right)d k\wedge d\hat H\wedge d\tilde W \wedge dH\\
    &=&\frac{48}{{\rm vol}(K_\C)}\int_{K_\C}\varphi\left(k\ast \delta s(\underline\Upsilon)\right)d\hat H\wedge dN_1 \wedge dN_2\wedge d k. 
    \end{eqnarray*}
    By \eqref{Haararch2} we have  $d^\times g=16d_Lp dk$, where $d_Lp=r^{-3} drds_1ds_2$ is the left Haar measure of the Parabolic subgroup $P=\left\{\big(\begin{smallmatrix}
        r&s_1+is_2\\&r^{-1}
    \end{smallmatrix}\big)\colon  r,s_1,s_2\in\R\right\}$, and $dk=\sin2\theta dadbd\theta$ is a Haar measure of $K_\C$ such that ${\rm vol}(K_\C)=2\pi^2$. We easily compute, for any $\phi(r,s_1,s_2)\in C^\infty(P,\C)$,
    \begin{eqnarray*}
        \hat H\phi(r,s_1,s_2)&=&\frac{d}{dt}\phi(\big(\begin{smallmatrix}
        r&s_1+is_2\\&r^{-1}
    \end{smallmatrix}\big)\exp(t\hat H))\mid_{t=0}=\frac{d}{dt}\phi(\big(\begin{smallmatrix}
        r&s_1+is_2\\&r^{-1}
    \end{smallmatrix}\big)\big(\begin{smallmatrix}
        e^t&\\&e^{-t}
    \end{smallmatrix}\big))\mid_{t=0}=\frac{d}{dt}\phi(re^t,s_1e^{-t},s_2e^{-t})\mid_{t=0}\\
    N_1\phi(r,s_1,s_2)&=&\frac{d}{dt}\phi(\big(\begin{smallmatrix}
        r&s_1+is_2\\&r^{-1}
    \end{smallmatrix}\big)\exp(tN_1))\mid_{t=0}=\frac{d}{dt}\phi(\big(\begin{smallmatrix}
        r&s_1+is_2\\&r^{-1}
    \end{smallmatrix}\big)\big(\begin{smallmatrix}
        1&t\\&1
    \end{smallmatrix}\big))\mid_{t=0}=\frac{d}{dt}\phi(r,s_1+rt,s_2)\mid_{t=0}\\
    N_2\phi(r,s_1,s_2)&=&\frac{d}{dt}\phi(\big(\begin{smallmatrix}
        r&s_1+is_2\\&r^{-1}
    \end{smallmatrix}\big)\exp(tN_2))\mid_{t=0}=\frac{d}{dt}\phi(\big(\begin{smallmatrix}
        r&s_1+is_2\\&r^{-1}
    \end{smallmatrix}\big)\big(\begin{smallmatrix}
        1&it\\&1
    \end{smallmatrix}\big))\mid_{t=0}=\frac{d}{dt}\phi(r,s_1,s_2+rt)\mid_{t=0}.
    \end{eqnarray*}
    Thus, 
    \[
    \left(\begin{array}{c}
         \hat H \\
         N_1\\
         N_2
    \end{array}\right)=\left(\begin{array}{ccc}
         r&-s_1&-s_2 \\
         &r&\\
         &&r
    \end{array}\right)\left(\begin{array}{c}
         \frac{\partial}{\partial r} \\
         \frac{\partial}{\partial s_1}\\
         \frac{\partial}{\partial s_2}
    \end{array}\right);\qquad d\hat H\wedge dN_1 \wedge dN_2=d_Lp=r^{-3} drds_1ds_2
    \]
    and 
    \[
    \varphi(c_3)(\hat H,\tilde W,H)\cdot d\hat H\wedge d\tilde W \wedge dH=\frac{3}{2\pi^2}\int_{0}^\pi\int_0^{2\pi}\int_0^{\frac{\pi}{2}}\varphi\left(k\ast \delta s(\underline\Upsilon)\right)d^\times g.
    \]
\end{remark}

\section{Fundamental classes}\label{fundClasses}

Let $H$ be an algebraic group over $F$ that is either $G$ or a maximal torus $T$ in $G$. For any $\sigma\in\Sigma_F$, write $K_{\sigma}^H$ for a maximal compact subgroup of $H(F_\sigma)$, $K_{\sigma,+}^H$ the connected component of 1, and $H(F_\sigma)_0=H(F_\sigma)/K_{\sigma}^H$.  
Write also $H(F_\sigma)_+$ for the connected component of $1$ in $H(F_\sigma)$. We can visualize  in Table \ref{table 1} and Table \ref{table 2} the cases we are interested in.

  \begin{table}[h!]
    \centering
\begin{tabular}{ c | c c c c}
$T(F_\sigma)$ & $T(F_\sigma)_+$&$T(F_\sigma)_0$&$K_\sigma^T$&$K_{\sigma,+}^T$ \\ 
 \hline
 $\R^\times$ & $\R_+$& $\R_+$& $\pm 1$ & 1\\  
  $\C^\times$ & $\C^\times$ & $\R_+$ & $S^1$& $S^1$ \\
  $\C^\times/\R^\times$ & $\C^\times/\R^\times$ & $1$ & $\C^\times/\R^\times$ & $\C^\times/\R^\times$ \\
\end{tabular}
\caption{Case $H=T$}
\label{table 1}

\end{table}
\begin{table}[h!]
  \centering
\begin{tabular}{ c | c c c c}
$G(F_\sigma)$ & $G(F_\sigma)_+$&$G(F_\sigma)_0$&$K_\sigma^G$&$K_{\sigma,+}^G$ \\ 
 \hline
  $\PGL_2(\R)$ & $\PGL_2(\R)_+$ & $\mathfrak{H}_2$ & ${\rm O}(2)$ & $\SO(2)$\\
  $\PGL_2(\C)$ & $\PGL_2(\C)$ & $\mathfrak{H}_3$ & $U(2)$ & $U(2)$\\
  $\bH^\times/\R^\times$ & $\bH^\times/\R^\times$ & $1$ & $\bH^\times/\R^\times$ & $\bH^\times/\R^\times$
\end{tabular}
\caption{Case $H=G$}
\label{table 2}
\end{table}

Here $\mathfrak{H}_2$ is the Poincar\'e upper half plane, $\mathfrak{H}_3$ is the hyperbolic 3-space, and $\bH$ are the Hamiltonians. Write $H(F_\infty)_0=\prod_{\sigma\in\Sigma_F}H(F_\sigma)_0$ and $H(F)_+=H(F_\infty)_+\cap H(F)$, where $H(F_\infty)_+=\prod_{\sigma\in\Sigma_F}H(F_\sigma)_+\subseteq H(F_\infty)$. Notice that $H(F_\infty)_0\simeq \R^u$, for some $u\in\N$, and 
$H(F_\infty)_+= H(F_\infty)_0\times K^H_{\infty,+}$, where $K^H_{\infty,+}=\prod_{\sigma\in\Sigma_F}K^H_{\sigma,+}$.

Fix $U\subset H(\A_F^\infty)$ an open compact subgroup, and fix representatives 
\[
\tilde g_i\in H(\A_F^\infty);\qquad \{[\tilde g_i]=g_i\}_i={\rm Pic}_H(U):=H(F)_+\backslash H(\A_F^\infty)\slash U. 
\]
We write $\Gamma_{g_i}=\tilde g_iU \tilde g_i^{-1}\cap H(F)_+$ and let $\cG_{g_i}=\Gamma_{g_i}\cap K^H_{\infty,+}$. Since $\Gamma_{g_i}$ is discrete and $K^H_{\infty,+}$ is compact, $\cG_{g_i}$ is finite.

\subsection{Case $H\neq \G_m$}

Write $M=H(F_\infty)_0\simeq \R^u$. The de Rham complex $\Omega^\bullet_M$ is a resolution for $\R$. This implies that we have an edge morphism of the induced spectral sequence
\[
e:H^0(\Gamma_{g_i},\Omega_M^u)\longrightarrow H^u(\Gamma_{g_i},\R).
\]
Since $\Gamma_{g_i}\backslash M$ is compact or admits a Borel-Serre compactification, we can identify $c\in H_{u}(\Gamma_{g_i}\backslash M,\Z)$ with a group cohomology element $c\in H_{u}(\Gamma_{g_i},\Z)$ by means of the relation
\[
\int_c \omega=e(\omega)\cap c,\qquad \omega\in H^0(\Gamma_{g_i},\Omega_M^u)=\Omega_{\Gamma_{g_i}\backslash M}^u.
\]
In particular, we can think of the fundamental class as an element $\xi_{g_i}\in H_u(\Gamma_{g_i},\Z)$ satisfying
\begin{equation}\label{eqeint}
e(\omega)\cap \xi_{g_i}=\int_{\Gamma_{g_i}\backslash M} \omega,\qquad \omega\in H^0(\Gamma_{g_i},\Omega_M^u).
\end{equation}
\begin{remark}\label{remonSL}
 Notice that we have a natural $\Gamma_{g_i}$-equivariant embedding 
\[
\iota_{g_i}:C^0(\tilde g_iU,\Q)\hookrightarrow C_c^0(H(\A_F^\infty),\Q),\qquad \iota_{g_i} \phi(h_f) =\phi(h_f)\cdot1_{\tilde g_iU}(h_f),
\]
where $C_c^0$ denotes the space of locally constant and compactly supported functions.
Such an embedding provides an isomorphism of $H(F)_+$-modules
\[
\bigoplus_{g_i\in \Pic_H(U)}\Ind_{\Gamma_{g_i}}^{H(F)_+}(C^0(\tilde g_iU,\Q))\longrightarrow C_c^0(H(\A_F^\infty),\Q).
\]
\end{remark}
We define the fundamental class
\[
\eta_H=\left(\frac{1}{\#\cG_{g_i}}(\xi_{g_i}\cap 1_{\tilde g_i U})\right)_{g_i\in {\rm Pic}_H(U)}\in \bigoplus_{g_i \in {\rm Pic}_H(U)} H_u(\Gamma_{g_i},C_c^0(\tilde g_iU,\Q))=H_u(H(F)_+,C_c^0(H(\A_F^\infty),\Q)),
\]
where the last equality follows from Shapiro's lemma and Remark \ref{remonSL}.
\begin{remark}\label{remarkonHC}
    The above defined fundamental classes differ from those defined in \cite{preprintsanti2} when $H=T$.
    Indeed, in \cite{preprintsanti2} the class $\eta_T$ is defined without dividing by $\#\cG_{g_i}$ and, therefore, this makes it so that it is an element of $H_u(T(F)_+,C_c^0(T(\A_F^\infty),\Z))$. But this is not such a big difference because $T$ is abelian and $\cG_{g_i}$ is independent of $g_i$, hence, both definitions differ by the factor $h=\#\cG_{g_i}$. In the general situation that we present here, we have
    \[
    \eta_H\in \frac{1}{N}H_u(H(F)_+,C_c^0(H(\A_F^\infty),\Z));\qquad N={\rm lcm}(\#\cG_{g_i})_{g_i\in {\rm Pic}_H(U)}.
    \]
\end{remark}

\subsection{Case $H=\G_m$}\label{fundclassGm} 
In case $H=\G_m$ we have a natural morphism 
\begin{equation}\label{absvalueGm}
    N:H(F_\infty)_0=\R_+^{\Sigma_F}\rightarrow\R_+;\qquad (x_\sigma)_{\sigma\in \Sigma_F}\longmapsto\prod_{\sigma\in\Sigma_F}x_\sigma^{[F_\sigma:\R]},
\end{equation}
coming from the absolute value $F_\infty^\times\rightarrow\R_+$, and write $M_0=\ker(N)$. It is well known that the image of $\Gamma_{g_i}=\Gamma=U\cap H(F)_+$ in $H(F_\infty)_0$ lies in $M_0$ and $\Gamma\backslash M_0$ is compact. Thus, similarly as in \eqref{eqeint}, we can define a fundamental class $\xi\in H_{u-1}(\Gamma,\Z)$ with $u=\#\Sigma_F$, so that,
\begin{equation}\label{eqeint2}
e(\omega)\cap \xi=\int_{\Gamma\backslash M_0} \omega,\qquad \omega\in H^0(\Gamma,\Omega_{M_0}^{u-1}).
\end{equation}
Analogously as above, we define
\[
\eta_H=\eta_{\G_m}=\left(\frac{1}{\#\Gamma_{\rm tors}}\xi\cap 1_{\tilde x_i U}\right)_{x_i\in {\rm Pic}_H(U)}\in \bigoplus_{x_i \in {\rm Pic}_H(U)} H_{u-1}(\Gamma,C_c^0(\tilde x_iU,\Q))=H_{u-1}(F_+^\times,C_c^0((\A_F^\infty)^\times,\Q)).
\]

\subsection{Independence of the choice of $U$}\label{IndepU} It seems that the definition of $\eta_H$ depends on the choice of the open compact subgroup $U\subset H(\A_F^\infty)$, but in this section we will show that this is not the case. 
\begin{proposition}\label{prop:indep of U}
    The class $\eta_H$ does not depend on the choice of the open compact subgroup $U\subset H(\A_F^\infty)$.
\end{proposition}
In the proof of Proposition \ref{prop:indep of U} we will assume that $H\neq \G_m$, but a similar argument applies for the case $H=\G_m$. 
First we realize that it is enough to check that we obtain the same fundamental class if we consider a finite index normal subgroup $V\subset U$. Indeed, if we write $\eta_{H,U}$ and $\eta_{H,U'}$ for the fundamental classes obtained by means of $U$ and $U'\subset H(\A_F^\infty)$, then we can always find finite index compact subgroups $V,V'\subset U\cap U'$ such that $V\unlhd U$ and $V'\unlhd U'$. In particular, we have $V,V'\unlhd U\cap U'$. Thus, the claim for normal subgroups implies 
\[
\eta_{H,U}=\eta_{H,V}=\eta_{H,U\cap U'}=\eta_{H,V'}=\eta_{H,U'}.
\]

Let $V\unlhd U$ be a normal subgroup of finite index. We aim to show that $\eta_{H,U}=\eta_{H,V}$. Observe that we have a surjective map
\[
p:\Pic_H(V)\rightarrow\Pic_H(U).
\]
For any $g\in \Pic_H(U)$, write $W_g=p^{-1}(g)$.
\begin{lemma}\label{lemma:above}
    For any $g\in \Pic_H(U)$ we have an isomorphism
    \[
    \bigoplus_{w\in W_g}{\rm Ind}_{\Gamma_{w}}^{\Gamma_g}C^0(\tilde{w}V,\Q)\simeq C^0(\tilde g U,\Q),
    \]
    where $\tilde{g}\in H(\A_F^\infty)$ is a lift of $g$, $\tilde{w}\in H(\A_F^\infty)$ is a lift of $w\in \Pic_H(V)$ such that $\tilde wU=\tilde gU$, $\Gamma_g=\tilde gU\tilde g^{-1}\cap H(F)_+$ and $\Gamma_w=\tilde wV\tilde w^{-1}\cap H(F)_+$. The preimage of $1_{\tilde gU}$ is $(f^0_w)_{w\in W_g}$ so that $f^0_w(\gamma)=1_{\tilde wV}$, for all $\gamma\in\Gamma_g$.
\end{lemma}
\begin{proof}
    It is easy to check that $\bigcup_{w\in W_g}\Gamma_g\tilde wV\subseteq \tilde gU$. In fact, $\bigcup_{w\in W_g}\Gamma_g\tilde wV= \tilde gU$ because any $\tilde gu\in\tilde gU$ satisfies $\gamma\tilde g u=\tilde w v$, for some $\gamma\in H(F)_+$, $w\in W_g$ and $v\in V$. Moreover, by construction $\tilde g^{-1}\tilde w\in U$, hence, $\gamma=\tilde g\tilde g^{-1}\tilde w vu^{-1}\tilde g^{-1}\in \Gamma_g$ and $\tilde g u=\gamma^{-1}\tilde wv\in \Gamma_g\tilde wV$. Thus,     
    we can define a morphism
    \[
    \iota:\bigoplus_{w\in W_g}{\rm Ind}_{\Gamma_{w}}^{\Gamma_g}C^0(\tilde{w}V,\Q)\longrightarrow C^0(\tilde g U,\Q);\qquad \iota((f_w)_{w\in W_g})(\gamma\tilde wv)=f_w(\gamma^{-1})(\tilde w v).
    \]
    Such a morphism is well defined because, if $\gamma_1\tilde wv_1=\gamma_2\tilde wv_2\in \Gamma_g\tilde wV$ then $\gamma_1^{-1}\gamma_2=\tilde w v_1v_2^{-1}\tilde w^{-1}\in \Gamma_w$, and
        \[
        \iota((f_w)_{w\in W_g})(\gamma_1\tilde wv_1)=f_w(\gamma_1^{-1}\gamma_2\gamma_2^{-1})(\tilde w v_1)=f_w(\gamma_2^{-1})(\gamma_2^{-1}\gamma_1\tilde w v_1)=f_w(\gamma_2^{-1})(\tilde w v_2)=\iota((f_w)_{w\in W_g})(\gamma_2\tilde wv_2).
        \]
        Moreover, $\iota((f^0_w)_{w\in W_g})=1_{\tilde gU}$.
    Finally, from the aforementioned equality $\bigcup_{w\in W_g}\Gamma_g\tilde wV= \tilde gU$, we deduce that it is an isomorphism.
\end{proof}
We write $\xi_{w}\in H_u(\Gamma_{w},\Z)$ for the fundamental class associated with $w\in W_{g}\subseteq\Pic_H(V)$, and $N_{w}$ for the number of fundamentals domains in $\Gamma_{w}\backslash M$ under the action of $\Gamma_{g}$. By definition, for any $\omega\in H^0(\Gamma_g,\Omega^u_M)$
\[
e(\omega)\cap{\rm cores}_{\Gamma_{w}}^{\Gamma_g}\xi_w=e({\rm res}_{\Gamma_{w}}^{\Gamma_g}\omega)\cap\xi_w=\int_{\Gamma_w\backslash M}\omega=N_w\int_{\Gamma_g\backslash M}\omega=N_w(e(w)\cap\xi_g).
\]
Thus, ${\rm cores}_{\Gamma_{w}}^{\Gamma_{g}}\xi_{w}=N_w\xi_{g}$. Moreover, the class of $[\gamma]\in M=H(F_\infty)_+/K^H_{\infty,+}$, for $\gamma\in \Gamma_g$, must be in a unique fundamental domain.  Thus, since $\Gamma_w\unlhd\Gamma_g$ because $V\unlhd U$,
\[
N_w=\#\{[\gamma]\in \Gamma_{w}\backslash H(F_\infty)_+/K^H_{\infty,+}\colon \gamma\in\Gamma_g\}=\#\left(\Gamma_w\backslash\Gamma_g/\cG_g\right)=\frac{[\Gamma_g:\Gamma_w]}{[\cG_g:\cG_w]}.
\]
By Lemma \ref{lemma:above}, this implies
\begin{eqnarray*}
    \eta_{H,U}&=&\left(\frac{1}{\#\cG_{g_i}}(\xi_{g_i}\cap 1_{\tilde g_i U})\right)_{g_i\in {\rm Pic}_H(U)}=\left(\frac{1}{\#\cG_{g_i}N_w}({\rm cores}_{\Gamma_{w_j}}^{\Gamma_{g_j}}\xi_{w_j}\cap 1_{\tilde g_i U})\right)_{g_i\in {\rm Pic}_H(U)}\\
    &=&\left(\frac{1}{N_w\#\cG_{g_i}}(\xi_{w_j}\cap {\rm res}_{\Gamma_{w_j}}^{\Gamma_{g_j}}1_{\tilde g_i U})\right)_{g_i\in {\rm Pic}_H(U)}=\left(\frac{[\Gamma_{g_i}:\Gamma_{w_j}]}{\#\cG_{g_i}N_w}(\xi_{w_j}\cap 1_{\tilde w_j V})\right)_{w_j\in {\rm Pic}_H(V)}=\eta_{H,V}.
\end{eqnarray*}
This concludes the proof of Proposition \ref{prop:indep of U}.

\section{Modular forms, periods, and global formulas}\label{sec: global formulas}

Let $\mfg_\infty=\prod_{\sigma\mid\infty}\mfg_\sigma$ be the real Lie algebra of $G(F_\infty)$, and $K_\infty=\prod_{\sigma\mid\infty}K_\sigma$ will denote the usual maximal compact subgroup. We also write $\cK_\infty\subseteq \mfg_\infty$ for the Lie algebra of $K_\infty$ and $K_{\infty,+}\subseteq K_\infty$ for the connected component of the identity.

\subsection{Modular forms}

Write $U_0(N)\subset\PGL_2(\A_F^\infty)$ for the usual open compact subgroup of matrices with integer entries that are upper triangular modulo  $N$.
The usual space of modular forms for $\PGL_2$ of weight $\underline k\in (2\N)^d$ and level $U_0(N)$ can be described as
\begin{equation}\label{defS2}
\cM_{\underline k}(U_0(N)):=\Hom_{(\mfg_\infty,K_\infty)}\ipa{D(\underline{k}),\cA(U_0(N))^{\PGL_2(F)}},    
\end{equation}
where $\cA(U_0(N))^{\PGL_2(F)}$ is the space of $U_0(N)$-invariant $K_\infty$-finite automorphic forms for $\PGL_2/F$ and $D(\underline k)$ is the $(\mfg_\infty,K_\infty)$-module $D(\underline k)=\bigotimes_{\sigma\mid\infty}D_\sigma(k_\sigma)$, where $D_\sigma(k_\sigma)$ are the $(\mfg_\sigma,K_\sigma)=(\mfgltwo(\R),O(2))$-modules described in \S \ref{GKR} and \S \ref{GKC}.

To provide similar definitions for the group $G$, first we have to fix isomorphisms $G(F_\sigma)\simeq \PGL_2(F_\sigma)$ at places $\sigma\in\Sigma_B$.  For this, from now on, we will assume that the fixed embedding $T\subseteq G$ associated with a maximal torus $E\hookrightarrow B$ satisfies the following hypothesis:
\begin{assumption}\label{assuSigmaSigma}
The set $\Sigma_B$ coincides with the set of archimedean places $\sigma$ where $E$ splits.
\end{assumption}
We will write 
\[
\Sigma_B^\R:=\{\sigma\in\Sigma_F\colon G(F_\sigma)\simeq\PGL_2(\R)\},\quad\Sigma_B^\C:=\{\sigma\in\Sigma_F \colon G(F_\sigma)\simeq\PGL_2(\C)\},\quad r_{1,B}:=\#\Sigma_B^\R, \quad r_2:=\#\Sigma_B^\C.
\]
Write also $r_B:=r_{1,B}+r_2$, $r_1:=\#\{\sigma\in\Sigma_F \colon F_\sigma\simeq\R\}$, $r_1^B:=\#(\Sigma_F\setminus\Sigma_B)$, and notice that $\Sigma_B=\Sigma_B^\R\cup\Sigma_B^\C$, $r_1=r_{1,B}+r_1^B$,
\[
(\Sigma_F\setminus\Sigma_B)=\{\sigma\in\Sigma_F\colon T(F_\sigma)\simeq\C^\times/\R^\times\},\qquad \Sigma_B^\R=\{\sigma\in\Sigma_F\colon T(F_\sigma)\simeq\R^\times\},\quad\mbox{and}\quad\Sigma_B^\C=\{\sigma\in\Sigma_F\colon F_\sigma\simeq\C\}.
\]
Under the above assumption, we obtain the desired isomorphisms: for each place $\sigma \in \Sigma_B$, the standard isomorphism $E \otimes_F F_\sigma \simeq F_\sigma^2$, together with the fixed embedding $B \hookrightarrow \M_2(E)$ from \eqref{embEinB}, induces an isomorphism $G(F_\sigma)\simeq\PGL_2(F_\sigma)$ which identifies $T(F_\sigma)$ with the diagonal torus.
\begin{remark}
    Unless otherwise stated, the chosen torus in case $G=\PGL_2$ will be the diagonal torus $E=F^2\hookrightarrow\M_2(F)$ and the corresponding embedding $B=\M_2(F)\hookrightarrow \M_2(E)$ is the diagonal one. This choice is consistent because the induced isomorphism $G(F_\sigma)\simeq\PGL_2(F_\sigma)$ is the identity.
\end{remark}
Once the isomorphisms $G(F_\sigma)\simeq\PGL_2(F_\sigma)$ are fixed at places $\sigma\in\Sigma_B$, we can define a maximal compact subgroup $K_\infty\subseteq G(F_\infty)$ whose components at $\Sigma_B$ can be identified with those described in \S \ref{GKR} and \S \ref{GKC}, moreover, we can define the $(\mfg_\infty,K_\infty)$-module
\[
D(\underline k)_B:=\bigotimes_{\sigma\in\Sigma_B} D(k_\sigma)\otimes\bigotimes_{\sigma\in\Sigma_F\setminus\Sigma_B}V(k_\sigma-2);\qquad k_\sigma=(k_\nu)_{\nu\mid\sigma},
\]
for any even weight $\underline k=(k_{\nu})_{\nu:F\hookrightarrow\C}\in (2\N_{\geq 2})^{[F:\Q]}$.
Given a open compact subgroup $U\subset G(\A_F^\infty)$, we  write $\cA(U)\subseteq C^\infty(G(\A_F)/U,\C)$ for the subset of $K_\infty$-finite vectors.
Then the \emph{space of modular forms for $G$ of lever $U$ and weight $\underline k$} is (in analogy with \eqref{defS2})
\[
\cM_{\underline k}(U):=\Hom_{(\mfg_\infty,K_\infty)}\left(D(\underline k)_B,\cA(U)^{G(F)}\right).
\]

\subsection{Petersson products}

In \cite[\S 3.2]{preprintsanti2}, a natural bilinear inner product $\langle\;,\;\rangle:\cM_{\underline k}(U)\times \cM_{\underline k}(U)\rightarrow\C$ is introduced. To describe this Petersson product, notice that the local morphisms \eqref{firstdefds} and \eqref{secdefds} induce a natural morphism
\begin{equation}\label{deltasdef}
    \underline{\delta s}_\lambda:V(\underline{k}-2)(\lambda)\longrightarrow D(\underline{k})_B,
\end{equation}
depending on a character $\lambda:G(F)/G(F)_+\rightarrow\pm1$ (and the above and the fixed embedding $E\hookrightarrow B$). Then the aforementioned pairing $\langle\;,\;\rangle$ is given by
\begin{equation*}\label{eqpeterssonprod}
    \langle\Phi_1,\Phi_2\rangle:=\int_{G(F)\backslash G(\A_F)}\Phi_1\Phi_2\left(\underline{\delta s}_\lambda(\Upsilon)\right)(g,g)d^\times g,
\end{equation*}
where $d^\times g$ is the usual Tamagawa measure with volume ${\rm vol}(G(\A_F)/G(F))=2$ and
\begin{equation}\label{defUpsilon}
    \bigotimes_{\sigma\mid\infty}\Upsilon_\sigma=\Upsilon=\left|\begin{array}{cc}
    x_1 & y_1 \\
    x_2 & y_2
\end{array}\right|^{\underline k-2}\in \cP(\underline{k}-2)\otimes \cP(\underline{k}-2)\simeq V(\underline{k}-2)\otimes V(\underline{k}-2).
\end{equation}
By \cite[remark 4.13]{preprintsanti2} the above definition of $\langle\;,\;\rangle$ is independent of $\lambda$.
As explained in \cite[Remark 3.2]{preprintsanti2}, if $F$ is totally real and $G=\PGL_2$ then we have a natural identification between $\cM_{\underline{k}}(U)$ and the space of Hilbert modular forms of weight $\underline{k}$. Under this identification, $\langle\Phi,\bar\Phi\rangle=2^{\underline k}2^{-[F:\Q]}{\rm vol}(U)\pi^{[F:\Q]}(\Phi,\Phi)_{U}$, where $(\;,\;)_{U}$ is the usual Petersson inner product.

\subsection{Normalized forms}\label{normforms}

Let $\pi$ an automorphic representation for $G$ of weight $\underline k$ and level $N$, and let $\Pi$ be its Jacquet-Langlands lift to $\PGL_2$. We will denote by $\Pi^\infty$ the representation $\Pi\mid_{G(\A_F^\infty)}$. We will define the normalized generator $\Psi\in \cM_{\underline k}(U_0(N))$ as follows:

Let $\Psi\in \cM_{\underline k}(U_0(N))$ be the form generating $\Pi^\infty$, normalized so that 
\begin{equation}\label{eqdefPsi}
\Lambda(s,\Pi)=|d_F|^{s-1/2}\int_{\A_F^\times/F^\times}\Psi(\underline {\delta s}_1(\mu_{\underline 0}))\left(\begin{array}{cc}a&\\&1\end{array}\right)|a|^{s-1/2}d^\times a,
\end{equation}
where $|\cdot|:\I_F\rightarrow\R_+$ is the standard adelic absolute value, $d_F\subset\cO_F$ is the different of $F$, and $\Lambda(s,\Pi)$ is the (completed) global L-function associated with $\Pi$. 
As pointed out in \cite[\S 3.2]{preprintsanti2}, in case where $F$ is totally real $\Psi$ corresponds to the normalized Hilbert newform under the natural identification between $\cM_{\underline k}(U_0(N))$ and the space of Hilbert modular forms.

Given an Eichler order $\cO_N\subset B$ of discriminant $N$, we write $U_N=\hat\cO_N^\times\subseteq G(\A_F^\infty)$. Notice that the space $(\pi^\infty)^{U_N}$ is one dimensional, and any non-zero element generates $\pi^\infty$, since all Eichler orders are conjugate. For any such a choice of the Eichler order, we fix $\Phi_0\in \cM_{\underline k}(U_N)$ to be the generator of $\pi^\infty$ normalized so that 
\begin{equation}\label{condnormalization}
\frac{\langle \Phi_0,\Phi_0\rangle}{\langle \Psi,\Psi\rangle}\cdot\frac{{\rm vol}(U_0(N))}{{\rm vol}(U_N)}=1.
\end{equation}
Observe that this characterizes $\Phi_0$ up to sign.

\subsection{$(\mfg_\infty,K_\infty)$-cohomology and differential forms}

Let us consider $\mathbb{H}:=G(F_\infty)_+/K_{\infty,+}$ the symmetric space associated with $G$. For any $V$ finite dimensional irreducible $G(F_\infty)$-representation over $\C$, we can consider the local system
\[
\tilde V:=G(F)_+\backslash (\mathbb{H}\times V\times G(\A_F^\infty)/U)\longrightarrow S_U:=G(F)_+\backslash (\mathbb{H}\times G(\A_F^\infty)/U)
\]
Then the space $\Omega^n(\tilde V)$ of (twisted) $n$-forms with values in $V$ admits a one-to-one correspondence
\[
\Omega:\Hom_{K_{\infty,+}}\big(\bigwedge^n(\mfg_\infty/\cK_\infty),C^\infty(G(\A_F)_+/U,V)^{G(F)_+}\big)=:
C^n((\mfg_\infty,K_{\infty,+}),C^\infty(G(\A_F)_+/U,V)^{G(F)_+})\stackrel{\simeq}{\longrightarrow}\Omega^n(\tilde V),
\]
where $G(\A_F)_+:=G(F_\infty)_+\times G(\A_F^\infty)$.
Indeed, given $\varphi\in \Hom_{K_{\infty,+}}\left(\bigwedge^n(\mfg_\infty/\cK_\infty),C^\infty(G(\A_F)_+/U,V)^{G(F)_+}\right)$ we consider the $n$-differential form (with coefficients in $V$)
\[
\Omega(\varphi)=\sum_{X_1^i\wedge\cdots \wedge X_n^i}\varphi(X_1^i,\cdots,X_n^i)\cdot dX_1^i\wedge\cdots\wedge dX_n^i,
\]
for any choice of a basis $\{X_1^i\wedge\cdots \wedge X_n^i\}_i$ of $\bigwedge^n(\mfg_\infty/\cK_\infty)$, where we write $dX$ for the 1-form dual to the left invariant derivation provided by $X$.

\begin{remark}\label{isoCvsCtensV}
It is easy to check that there is a morphism of $G(F)$-modules
    \[
    \iota:C^\infty(G(\A_F)/U,\C)\otimes V\rightarrow C^\infty(G(\A_F)/U,V);\qquad \iota(\phi\otimes v)(g)=\phi(g)\cdot (g_\infty v);\qquad g=(g_\infty,g_f)\in G(\A_F),
    \]
    where on the left hand side $V$ is considered with the trivial $G(F)$-action. 
\end{remark}    
The above remark induces an isomorphism
    \[
   C^\infty(G(\A_F)_+/U,V)^{G(F)_+}\simeq C^\infty(G(\A_F)_+/U,\C)^{G(F)_+}\otimes V\simeq C^\infty(G(\A_F)/U,\C)^{G(F)}\otimes V,
    \]
    since $C^\infty(G(\A_F)/U,\C)\simeq {\rm Ind}_{G(F)_+}^{G(F)}C^\infty(G(\A_F)_+/U,\C)$ because $G(F)/G(F)_+\simeq G(F_\infty)/G(F_\infty)_+$.
Hence, we obtain 
\[
\Omega^n(\tilde V)=C^n((\mfg_\infty,K_{\infty,+}),C^\infty(G(\A_F),\C)^{G(F)}\otimes V)\\
=C^n((\mfg_\infty,K_{\infty,+}),\cA(U)^{G(F)}\otimes V),
\]
where the last equality follows from the fact that any $\varphi\in \Hom_{K_{\infty,+}}\left(\bigwedge^n(\mfg_\infty/\cK_\infty),C^\infty(G(\A_F)/U,\C)\right)$ must have values in $\cA(U)$. Thus, we can identify $H^n((\mfg_\infty,K_{\infty,+}),\cA(U)^{G(F)}\otimes V)$ with the de Rham cohomology of $S_U$ with coefficients in $V$, obtaining an isomorphism
    \begin{equation*}
    H^n((\mfg_\infty,K_{\infty_+}),\cA(U)^{G(F)}\otimes V)\simeq H^n(S_U,\tilde V).
    \end{equation*}
    
Fo any $G(F)$-representation $M$ over some field $L$, we define the $G(F)$-representation  
\[
\cA^{\infty}(M)^U:=\{\phi: G(\A_F^{\infty})/U\longrightarrow M\};\qquad (\gamma\phi)(g)=\gamma\phi(\gamma^{-1}g),
\]
for $\gamma\in G(F)$, $g\in G(\A_F^\infty)$ and $\phi\in \cA^{\infty}(M)^U$.
Since $V$ is finite dimensional, it is easy to see that (see Remark \ref{remonSL}) 
\[
\cA^{\infty}(V)^U\simeq \bigoplus_{g_i\in\Pic_G(U)}{\rm coInd}_{\Gamma_{g_i}}^{G(F)_+}V,\qquad \Pic_G(U)=G(F)_+\backslash G(\A_F^\infty)/U, \quad \Gamma_{g_i}=G(F)_+\cap g_i Ug_i^{-1}.
\]
Hence, the usual identification of Betti and group cohomology induces an isomorphism
\begin{equation}\label{isoGKBetti}
    \kappa: H^n((\mfg_\infty,K_{\infty,+}),\cA(U)^{G(F)}\otimes V)\stackrel{\simeq}{\longrightarrow} H^n(S_U,\tilde V)\stackrel{\simeq}{\longrightarrow}H^n(G(F)_+,\cA^\infty(V)^U).
\end{equation}
\begin{remark}
Notice that, for any $K_\infty$-module $M$
\[
C^n((\mfg_\infty,K_{\infty}),M)=\Hom_{K_{\infty}}\big(\bigwedge^n(\mfg_\infty/\cK_\infty),M\big)=\Hom_{K_{\infty,+}}\big(\bigwedge^n(\mfg_\infty/\cK_\infty),M\big)^{K_\infty/K_{\infty,+}}=
C^n((\mfg_\infty,K_{\infty,+}),M)^{K_\infty/K_{\infty,+}}.
\]
Hence, we can identify 
\[
H^n((\mfg_\infty,K_{\infty}),\cA(U)^{G(F)}\otimes V)=H^n((\mfg_\infty,K_{\infty,+}),\cA(U)^{G(F)}\otimes V)^{K_\infty/K_{\infty,+}}\subseteq H^n((\mfg_\infty,K_{\infty,+}),\cA(U)^{G(F)}\otimes V).
\]
Similarly, we have 
\[
H^n(G(F),\cA^\infty(V)^U)=H^n(G(F)_+,\cA^\infty(V)^U)^{G(F)/G(F)_+}\subseteq H^n(G(F)_+,\cA^\infty(V)^U),
\]
and it is clear that $K_\infty/K_{\infty,+}=G(F_\infty)/G(F_\infty)_+=G(F)/G(F)_+$.
\end{remark}
\begin{lemma}\label{lemmaisoGKBetti}
    The restriction of $\kappa$ induces an isomorphism 
    \[
    \kappa: H^n((\mfg_\infty,K_{\infty}),\cA(U)^{G(F)}\otimes V)\stackrel{\simeq}{\longrightarrow} H^n(G(F),\cA^\infty(V)^U).
    \]
\end{lemma}
\begin{proof}
    
    If we write $\tilde \Gamma_g=G(F)\cap \tilde gU\tilde g^{-1}$ for a fixed representative $\tilde g$ of $g\in \widetilde{\Pic}_G(U)=G(F)\backslash G(\A_F^\infty)/U$, 
    \begin{eqnarray*}
    H^n(G(F),\cA^\infty(V)^U)&=&\sum_{g\in \widetilde{\Pic}_G(U)} H^n(\tilde\Gamma_g,V);\\
    H^n((\mfg_\infty,K_{\infty}),\cA(U)^{G(F)}\otimes V)&=&\sum_{g\in \widetilde{\Pic}_G(U)}H^n((\mfg_\infty,K_{\infty}),C^\infty(G(F_\infty),V)^{\tilde\Gamma_g}).
    \end{eqnarray*}
    Thus, one has to check that $H^n(\tilde\Gamma_g,V)\simeq H^n((\mfg_\infty,K_{\infty}),C^\infty(G(F_\infty),V)^{\tilde\Gamma_g})$. Notice that the functor from finite dimensional $G(F_\infty)$-representations to $(\mfg_\infty,K_{\infty})$-modules 
    \[
    W\longmapsto C^\infty(G(F_\infty),W)^{\tilde\Gamma_g}\simeq C^\infty(\tilde \Gamma_g\backslash G(F_\infty),\C)\otimes W,
    \]
     is exact (see Remark \ref{isoCvsCtensV}), hence, $H^n((\mfg_\infty,K_{\infty}),C^\infty(G(F_\infty),\bullet)^{\tilde\Gamma_g})$ is the derived functor of 
     \begin{eqnarray*}
         W\longmapsto H^0((\mfg_\infty,K_{\infty}),C^\infty(G(F_\infty),W)^{\tilde\Gamma_g})&=&\left\{\varphi\in \Hom_{K_\infty}(\C,C^\infty(G(F_\infty),W)^{\tilde\Gamma_g})\colon  X\varphi(1)=0 \mbox{ for all }X\in \mfg_\infty\right\}\\
         &=&\left\{f\in C^\infty(\mathbb{H},W)^{\tilde\Gamma_g}\colon  Xf=0 \mbox{ for all }X\in \mfg_\infty\right\}=C^0(\mathbb{H},W)^{\tilde\Gamma_g}=W^{\tilde\Gamma_g},
     \end{eqnarray*}
     and the result follows.
\end{proof}
\subsection{Cohomology of arithmetic groups and the Eichler-Shimura morphism}\label{cohoAG}

For any place $\sigma\mid\infty$ and any weight $\underline k=(k_{\nu})_{\nu:F\hookrightarrow\C}\in (2\N_{\geq 2})^{[F:\Q]}$, we write $V(k_\sigma-2)=\bigotimes_{\nu\mid\sigma}V(k_{\nu}-2)$, and
\[
M_\sigma:=\Hom(V(k_\sigma-2),D(k_\sigma))\text{ for } \sigma\in\Sigma_B;\qquad M_\sigma:=\Hom(V(k_\sigma-2),V(k_\sigma-2))\text{ for }  \sigma\not\in\Sigma_B.
\]
Thanks to the work done in \S \ref{ExpliCCR} and \S \ref{ExplCCC}, we have cohomology classes:
\[
\left\{\begin{array}{ll}
     c_{1,\sigma}^+\in H^1((\mfg_\sigma,K_\sigma),M_\sigma(+)),\quad c_{1,\sigma}^-\in H^1((\mfg_\sigma,K_\sigma),M_\sigma(-))&\text{ for } \sigma\in\Sigma_B^\R  \\
     c_{1,\sigma}\in H^1((\mfg_\sigma,K_\sigma),M_\sigma),\quad c_{2,\sigma}\in H^2((\mfg_\sigma,K_\sigma),M_\sigma)&\text{ for } \sigma\in\Sigma_B^\C\\
     c_{0,\sigma}={\rm id}\in H^0((\mfg_\sigma,K_\sigma),M_\sigma)& \text{ for } \sigma\not\in\Sigma_B. 
\end{array}\right.
\]

For any $\varepsilon=(\varepsilon_\sigma)_\sigma\in \{\pm1\}^{\Sigma_B}$, write $\varepsilon_\R$ for the character
\[
\varepsilon_\R:G(F_\infty)/G(F_\infty)_+\longrightarrow \{\pm 1\};\qquad \varepsilon_\R((g_\sigma)_\sigma)=\prod_{F_\sigma=\R} \left(\frac{\det(g_\sigma)}{|\det(g_\sigma)|}\right)^{\frac{1-\varepsilon_\sigma}{2}}.
\]
Notice that, given such an $\varepsilon$, we can consider the cross-product
\begin{equation}\label{defcvarepsilon}
c_\varepsilon=\prod_{\sigma\in\Sigma_B^\R}c_{1,\sigma}^{\varepsilon_\sigma}\times\prod_{\sigma\in\Sigma_B^\C}c_{\frac{3-\varepsilon_\sigma}{2},\sigma}\times\prod_{\sigma\not\in\Sigma_B}c_{0,\sigma}\in H^{n_\varepsilon}((\mfg_\infty,K_\infty),M_\infty(\varepsilon_\R)),
\end{equation}
where
\[
M_\infty=\bigotimes_{\sigma\mid \infty} M_\sigma=\Hom(V(\underline k-2),D(\underline k))\, \text{ and }\,   
n_\varepsilon=r_{1,B}+\sum_{\sigma\in\Sigma_B^\C}\frac{3-\varepsilon_\sigma}{2}.
\]
Hence the degree of the cohomology $n_\varepsilon$ belongs to $\{r_B,\cdots,r_B+r_2\}$.

\begin{definition}\label{DefES}
We construct the \emph{Eichler-Shimura morphism associated with $\varepsilon\in \{\pm1\}^{\Sigma_B}$} as the map
\begin{eqnarray*}
{\rm ES}_\varepsilon:\cM_{\underline k}(U)=H^0((\mfg_\infty,K_\infty),\Hom(D(\underline k),\cA(U)^{G(F)}))&\stackrel{\cup c_\varepsilon}{\longrightarrow}& H^{n_\varepsilon}((\mfg_\infty,K_\infty),\Hom(V(\underline k-2)(\varepsilon_\R),\cA(U)^{G(F)}))\simeq\\
&\stackrel{\kappa}{\simeq}& H^{n_\varepsilon}(G(F),\cA^\infty(V(\underline k-2)(\varepsilon_\R))^U),
\end{eqnarray*}
where the last isomorphism $\kappa$ is the one provided by Lemma \ref{lemmaisoGKBetti}.
\end{definition}
\begin{remark}
    There are other approaches to defining ${\rm ES}_\varepsilon$ (see \cite{ESsanti} or \cite{preprintsanti2}), but we believe the one above is the most elegant, generalizable, and easy to work with.
\end{remark}

\subsection{Independence of choices}
Recall that the isomorphism $G(F_\sigma)\simeq \PGL_2(F_\sigma)$, required to endow $G(F_\sigma)$-structure to $D(k_\sigma)$  if $\sigma\in\Sigma_B$, depends on the fixed embedding $\varphi_E:B\hookrightarrow\M_2(E)$ of \eqref{embEinB}. Similarly, as explained in \S \ref{PolTorus}, it also provides the $G(F_\sigma)$-structure to $V(k_\sigma-2)$, and so $\varphi_E$ determines the $G(F_\infty)$-structure of $D(\underline k)_B$.

Assume that we have two (possibly different) torus $\imath_1:E_1\hookrightarrow B$ and $\imath_2:E_2\hookrightarrow B$, both satisfying Assumption \ref{assuSigmaSigma}, and we fix isomorphisms $\varphi_1:B\hookrightarrow\M_2(E_1)$ and $\varphi_2:B\hookrightarrow\M_2(E_2)$ as in \eqref{embEinB}. If we denote the identifications provided by $\varphi_i$ by $I_i:G(F_{\Sigma_B})\rightarrow\PGL_2(F_{\Sigma_B})$, we have that $I_1^{-1}\circ I_2$ is given by conjugation. Hence, there exists $\gamma_{\Sigma_B}\in G(F_{\Sigma_B})$ such that
\[
(I_1^{-1}\circ I_2)(g)=\gamma_{\Sigma_B}^{-1}g\gamma_{\Sigma_B};\qquad\mbox{for all }\quad g\in G(F_{\Sigma_B}).
\]
Similarly, the $\varphi_i$ provide embeddings $e_i:G(F_{\Sigma_F\setminus\Sigma_B})\hookrightarrow \PGL_2(E_{\Sigma_F\setminus\Sigma_B})$, where $E_{\Sigma_F\setminus\Sigma_B}=E_{1,\Sigma_F\setminus\Sigma_B}=E_{2,\Sigma_F\setminus\Sigma_B}=\C^{r_1^n}$. Thus, there exists $\gamma_{\Sigma_F\setminus\Sigma_B}\in\PGL_2(E_{\Sigma_F\setminus\Sigma_B})$ such that 
\[
e_2(g)=\gamma_{\Sigma_F\setminus\Sigma_B}^{-1}e_1(g)\gamma_{\Sigma_F\setminus\Sigma_B};\qquad\mbox{for all }\quad g\in G(F_{\Sigma_F\setminus\Sigma_B}).
\]
Let us also choose any pair of Eichler orders $\cO_{N,1},\cO_{N,2}\subset B$ of level $N$ and write $U_{N,i}=\hat\cO_{N,i}^\times\subset G(\A_F^\infty)$ as above.
Since all local Eichler orders of level $N$ are conjugated, there exists $\gamma_f\in G(\A_F^\infty)$ such that $U_{N,1}=\gamma_f^{-1}U_{N,2}\gamma_f$.

Thus, if $\cM_{\underline k}(U_{N,i})$ is the space of modular forms constructed by means of the embedding $\varphi_i$ and the open compact subgroup $U_{N,i}\subset G(\A_F^\infty)$, we have an isomorphism
\begin{equation}\label{deftheta}
\theta:\cM_{\underline k}(U_{N,1})\longrightarrow \cM_{\underline k}(U_{N,2});\qquad \theta(\Phi)(f)(g_{\Sigma_B},g_{\Sigma_F\setminus\Sigma_B},g_f)=\Phi(\gamma_{\Sigma_F\setminus\Sigma_B}f)(g_{\Sigma_B}\gamma_{\Sigma_B},g_{\Sigma_F\setminus\Sigma_B},g_f\gamma_f),
\end{equation}
for all $f\in D(\underline k)_B$, $g_{\Sigma_B}\in G(F_{\Sigma_B})$, $g_{\Sigma_F\setminus\Sigma_B}\in G(F_{\Sigma_F\setminus\Sigma_B})$ and $g_f\in G(\A_F^\infty)$. By the $\PGL_2(F_\infty)$-invariance of $\Upsilon$ and the Haar measure $d^\times g$, it is easy to check that
\begin{equation}\label{eqindemb}
    \langle\Phi_1,\Phi_2\rangle=\langle\theta\Phi_1,\theta\Phi_2\rangle;\qquad \Phi_1,\Phi_2\in \cM_{\underline k}(U_{N,1}).
\end{equation}
Hence $\theta$ sends normalized forms to normalized forms (it is clear that ${\rm vol}(U_{N,1})={\rm vol}(U_{N,2})$). 

By means of such Eichler--Shimura morphisms we can realize the automorphic representations in the group cohomology spaces $H^n(G(F),\cA^{\infty}(V(\underline k-2)(\lambda))^{U})=H^{n}(G(F)_+,\cA^{\infty}(V(\underline k-2))^{U})^{\lambda}$, $n\in \{r_B,\cdots, r_B+r_2\}$, for any fixed character $\lambda:G(F)/G(F)_+=G(F_\infty)/G(F_\infty)_+\rightarrow\pm1$. The following lemma is straightforward.
\begin{lemma}\label{lemindemb}
    For any $\varepsilon\in\{\pm 1\}^{\Sigma_B}$, we have the following commutative diagram
    \[
    \xymatrix{
\cM_{\underline k}(U_{N,1})\ar[r]^{\theta}\ar[d]^{{\rm ES}_\varepsilon}& \cM_{\underline k}(U_{N,2})\ar[d]^{{\rm ES}_\varepsilon}\\
H^{n_\varepsilon}(G(F)_+,\cA^{\infty}(V(\underline k-2))^{U_{N,1}})^{\varepsilon_\R}\ar[r]^{\theta^\ast}&H^{n_\varepsilon}(G(F)_+,\cA^{\infty}(V(\underline k-2))^{U_{N,2}})^{\varepsilon_\R}
    }
    \]
    where the morphism $\theta^\ast$ is induced by 
    \[
    \theta^\ast:\cA^{\infty}(V(\underline k-2))^{U_{N,1}}\longrightarrow\cA^{\infty}(V(\underline k-2))^{U_{N,2}};\  (\theta^\ast\phi)(g_f)=\gamma_\infty^{-1}\phi(g_f\gamma_f);\ \gamma_\infty=(I_1\gamma_{\Sigma_B},\gamma_{\Sigma_F\setminus\Sigma_B})\in \PGL_2(F_\infty).
    \]
\end{lemma}
\begin{remark}\label{remindemb}
Let $L_{\underline k}$ be the number field defined in \eqref{defQk}, and
recall the $L_{\underline k}$-model $V(\underline k-2)_{L_{\underline k}}$ of $V(\underline k-2)$ introduced in \S \ref{remdefmodVk}.
Then the restriction of $\theta^\ast$ to $H^{n_\varepsilon}(G(F)_+,\cA^{\infty}(V(\underline k-2)_{L_{\underline k}})^{U_{N,1}})^{\lambda}$ coincides with the action of right translation by $\gamma_f$. Indeed, if we write $V(\underline{k}-2)_{(i)}$ for the $\C$-vector space $V(\underline{k}-2)$ endowed with the action of $G(F_\infty)$ provided by $\imath_i=(I_i,e_i):G(F_\infty)\hookrightarrow\PGL_2(F_{\Sigma_B}\times E_{\Sigma_F\setminus\Sigma_B})$, we have
\[
\xymatrix{
&V(\underline{k}-2)_{L_{\underline k}}\ar[dl]_{\kappa_1}\ar[dr]^{\kappa_2}&\\
V(\underline{k}-2)_{(1)}\ar[rr]^{\gamma_\infty^{-1}}&&V(\underline{k}-2)_{(2)}
}
\]
where $\kappa_i$ are the embeddings induced by the morphisms of Lemma \ref{lemaVkVk}
\end{remark}

\subsection{Periods}\label{sec:periods}

Let $\pi$ be an automorphic representation for $G$ of weight $\underline k$ and level $N$. By means of the morphism ${\rm ES}_\varepsilon$ we can realize $\pi^\infty$ in the cohomology spaces $H^{n_\varepsilon}(G(F)_+,\cA^{\infty}(V(\underline k-2))^{U})^{\varepsilon_\R}$, for any choice of $\varepsilon\in\{\pm 1\}^{\Sigma_B}$.
Let $L_\pi$ be the coefficient field of $\pi$, namely, the minimum extension of $L_{\underline k}$ that contains all the eigenvalues of all Hecke operators.  The $\pi^\infty$-isotypical component of the corresponding cohomology space is one-dimensional precisely when $n_\varepsilon=r_B$ (in the lowest degree) or $n_\varepsilon=r_B+r_2$ (in the highest degree). This implies that, in these lowest and highest degree situations, there exists $\Omega^\pi_\varepsilon\in\C^\times$ such that
\[
\frac{{\rm ES}_\varepsilon(\Phi_0)}{\Omega^\pi_\varepsilon}\in H^{n_\varepsilon}(G(F)_+,\cA^{\infty}(V(\underline k-2)_{L_\pi})^{U_N})^{\varepsilon_\R},
\]
for a normalized modular form $\Phi_0\in \cM_{\underline k}(U_N)$ of $\pi^\infty$.
With this definition, the period $\Omega^\pi_\varepsilon$ is well defined up to a factor in $L_\pi^\times$. By equation \eqref{eqindemb}, Lemma \ref{lemindemb} and Remark \ref{remindemb} the class $\Omega^\pi_\varepsilon\mod L_\pi^\times$ is independent of the embedding   $E\hookrightarrow B$ used to determine the $G(F_\infty)$-structure of $D(\underline k)_B$. Throughout this article, for any group $G'$ such that $\pi$ admits a Jacquet-Langlands lift $\pi'$ to $G'$, we will fix a choice of a period $\Omega_\varepsilon^{\pi'}$ for any character $\varepsilon$ of lowest or highest degree. We will denote by  $\Pi$ the Jacquet--Langlands lift of $\pi$ to $\PGL_2$, and by  $\Omega_\varepsilon^\Pi$ the associated period.

\subsection{Modular symbols}

In this section we will assume that $G=\PGL_2$. In this situation, the real manifold $S_U$ is non-compact and we can consider the Borel-Serre compactification $\bar S_U$. For any finite dimensional $G(F)$-representation $V$, one defines the cuspidal cohomology
\[
H^n_{\rm cusp}(S_U,\tilde V)=H^{n}_{\rm cusp}(G(F)_+,\cA^\infty(V)^U):=\bigoplus_{\pi}H^{n}((\mfg_\infty,K_{\infty,+}),\Hom(V,(V_\pi^{\rm fin})^U)),
\]
where the sum runs over all irreducible cuspidal automorphic representations $\pi$, and $(V_\pi^{\rm fin})^U\subset \cA(U)$ is the $(\mfg_\infty,K_\infty)$-module of $U$-invariant $K_\infty$-finite vectors of $\pi$.
By a theorem due to Borel we have that
\[
H^n_{\rm cusp}(S_U,\tilde V)\hookrightarrow H^n_!(S_U,\tilde V)\subset H^n(S_U,\tilde V)
\]
where $H^n_!(S_U,\tilde V)$ is the image of the canonical map $H^n_c(S_U,\tilde V)\rightarrow H^n(S_U,\tilde V)$, being $H^n_c(S_U,\tilde V)$ the cohomology with compact support.

Similarly as above, we can define
\[
\cS_{\underline k}(U):=\bigoplus_{\pi}\Hom_{(\mfg_\infty,K_\infty)}\left(D(\underline k),(V_\pi^{\rm fin})^U\right).
\]
As discussed above, for any $f\in \cS_{\underline k}(U)$ the cup-product $f\cup c_\varepsilon\in H^n_{\rm cusp}(S_U,\tilde V)$ lies in the image of the cohomology with compact support, for any $c_\varepsilon$ as in \eqref{defcvarepsilon}. Hence, we can integrate $f\cup c_\varepsilon$ through the geodesic joining two cusps. Notice that the set of cusps is in correspondence with
\[
B(F)_+\backslash G(\A_F^\infty)/U=\bigsqcup_{[g_i]\in \Pic_G(U)}\Gamma_{g_i}\backslash \PP^1(F);\qquad {\rm Pic}_G(U):=G(F)_+\backslash G(\A_F^\infty)\slash U,\quad 
\Gamma_{g_i}=\tilde g_iU \tilde g_i^{-1}\cap G(F)_+.
\]
Thus, if we write $\cA^\infty(\ast,\bullet)^U:=\cA^\infty(\Hom(\ast,\bullet))^U$, the map ${\rm MS}_\varepsilon(f)(a-b)=\int_b^af\cup c_\varepsilon$ defines a morphism 
\begin{eqnarray*}
{\rm MS}_\varepsilon:\cS(U,\underline k)\longrightarrow \left(\bigoplus_{[g_i]\in \Pic_G(U)}H^{n_\varepsilon-1}(\Gamma_{g_i},\Hom(\Delta_0,V(\underline k-2)))\right)^{\varepsilon_\R}&=&H^{n_\varepsilon-1}(G(F)_+,\cA^\infty(\Delta_0,V(\underline k-2))^U)^{\varepsilon_\R};\\ 
&=&H^{n_\varepsilon-1}(G(F),\cA^\infty(\Delta_0,V(\underline k-2)(\varepsilon_\R))^U),
\end{eqnarray*}
where $\Delta_0$ is the group of degree zero divisor of $\Delta=\Z[\PP^1(F)]$. It is clear that the degree short exact sequence
\[
0\longrightarrow V(\underline k-2)(\varepsilon_\R)\longrightarrow \Hom(\Delta,V(\underline k-2)(\varepsilon_\R))\longrightarrow \Hom(\Delta_0,V(\underline k-2)(\varepsilon_\R))\longrightarrow 0,
\]
provides a connection morphism
\[
\delta:H^{n_\varepsilon-1}(G(F),\cA^\infty(\Delta_0,V(\underline k-2)(\varepsilon_\R))^U)\longrightarrow H^{n_\varepsilon}(G(F),\cA^\infty(V(\underline k-2)(\varepsilon_\R))^U),
\]
such that $\delta\circ {\rm MS}_\varepsilon={\rm ES}_\varepsilon$.

\subsection{Modular symbols and the diagonal torus}\label{MSandDT}

Let us consider the embedding $T=\G_m\hookrightarrow \PGL_2=G$ provided by $t\mapsto\big(\begin{smallmatrix}
    t&\\&1
\end{smallmatrix}\big)$. Write $\I_F=\G_m(\A_F)$ and $\I_F^\infty=\G_m(\A_F^\infty)$ for the group of ideles and finite ideles of $F$. Notice that the divisor $(\infty-0)\in\Delta_0$ is fixed by $\G_m(F)=F^\times$. Thus, the evaluation at $(\infty-0)$ provides a morphism
\[
H^{n}(G(F)_+,\cA^{\infty}(\Delta_0,V(\underline k-2)_{L_{\underline k}})^{U})^{\lambda}\longrightarrow H^{n}(F_+^\times,\cA^{\infty}(V(\underline k-2)_{L_{\underline k}})^{U})^{\lambda};\qquad \varphi\longmapsto \varphi(\infty-0),
\]
for any $\lambda:G(F)/G(F)_+\rightarrow \{\pm 1\}$.  On the other side, in \S \ref{fundclassGm} we have constructed a fundamental class 
\[
    \eta_{\G_m}\in H_{r-1}(F_+^\times,C_c^0(\I_F^\infty,\Q)),
\]
where $r=r_{B}=\#\infty$ in this situation.
Notice that we have a $F_+^\times$-equivariant morphism:
\begin{eqnarray}\label{firstvarphi}
\varphi:\left(C^{0}(\I_F^\infty,{L_{\underline k}})\otimes V(\underline{k}-2)_{L_{\underline k}}\right)\times \cA^{\infty}(V(\underline{k}-2)_{L_{\underline k}})^U&\longrightarrow& C^{0}(\I_F^\infty,{L_{\underline k}}),\\
\varphi((f\otimes \mu)\otimes\Phi)(t)&=&f(t)\cdot\langle\Phi(t),\mu\rangle,
\end{eqnarray}
where the $G(F)$-invariant pairing $\langle\;,\;\rangle:V(\underline{k}-2)_{L_{\underline k}}\times V(\underline{k}-2)_{L_{\underline k}}\rightarrow {L_{\underline k}}$ arises from the pairing defined in Remark \ref{dualVkmodelpairing}, noting that its extension to $\bar\Q$  is clearly compatible with the $G_{\underline k}$-action.
Let $\rho:\I_F/F^\times\rightarrow\C$ be a locally polynomial character of degree less that $\frac{\underline k-2}{2}$, namely, a character such that 
\[
\rho\mid_{F_\infty^\times}(t)=\rho_0(t)t^{\underline m};\qquad \frac{2-\underline k}{2}\leq\underline m\leq \frac{\underline k-2}{2},
\]
for some locally constant character $\rho_0$. If we interpret the function $t\mapsto t^{\underline m}$ as an element $\mu_{\underline m}\in V(\underline{k}-2)$ by means of \eqref{PtoCoverC} and \eqref{dualVP}, then we can regard $\rho$ as an element
\[
\rho=\rho\mid_{\I_F^\infty}\otimes\mu_{\underline m} \in H^0(F^\times_+,C^{0}(\I_F^\infty,\C)\otimes V(\underline{k}-2)).
\]
With the notation of \S \ref{fundClasses}, let us consider the connected compact subgroup $K^T_{\infty,+}=\prod_{\sigma\in\Sigma_F}K_{\sigma,+}^{T}$. Notice that $F^\times_{\sigma,+}=\R_+\times K_{\sigma,+}^{T}$, and our choice of the Haar measure of  $F^\times_{\sigma,+}$ in \S \ref{haarmeasures} is given by $d^\times x_\sigma=d^\times r_\sigma d k_\sigma$, where $d^\times r_\sigma=\frac{dr_\sigma}{r_\sigma}$ is the usual Haar measure on $\R_+$ and $d k_\sigma$ is the Haar measure on $K_{\sigma,+}^{T}$
\[
d k_\sigma=1,\quad\mbox{if }F_\sigma=\R;\qquad d k_\sigma=2\pi^{-1}d\theta_\sigma,\quad\mbox{if }F_\sigma=\C\mbox{ and }x_\sigma=r_\sigma e^{i\theta_\sigma}.
\]
The Tamagawa measure $d^\times t$ on $(\I_F^\infty)$ provides a $F_+^\times$-invariant pairing 
\begin{equation}\label{pairing1CC}
    C^0(\I_F^\infty,\C)\times C_c^0(\I_F^\infty,\C)\longrightarrow\C;\qquad (f_1,f_2)\longmapsto {\rm vol}(K^T_{\infty,+})\cdot\#( F^\times/F^\times_+)\cdot\int_{\I_F^\infty}f_1(z,t)f_2(z,t)d^\times t,
\end{equation}
where ${\rm vol}(K^T_{\infty,+})$ is taken with respect to the measures $dk_\sigma$ above. Any $\varepsilon\in \{\pm 1\}^{\Sigma_B}$ as above is called of \emph{lowest degree} if $n_\varepsilon=r_B=r$. Given such an lowest degree $\varepsilon$, we can consider the cup product
\[
\varphi\left({\rm MS}_\varepsilon(\Phi)(\infty-0)\cup\rho\right)\cap\eta_{\G_m}\in \C,
\]
for any $\Phi\in \cS_{\underline k}(U)$.
\begin{remark}\label{algebraicitypairing1}
    We can easily compute that ${\rm vol}(K^T_{\infty,+})=4^{r_2}$. Moreover, we have ${\rm vol}(\hat\cO_F^\times)=|d_F|^{\frac{-1}{2}}$. This implies that \eqref{pairing1CC} is the extension of scalars of a pairing
    \[
    C^0(\I^\infty_F,\Q)\times C_c^0(\I^\infty_F,\Q)\longrightarrow |d_F|^{\frac{-1}{2}}\Q,
    \]
    since any open compact subgroup $\cO\subset \I_F^\infty$ must be a finite index subgroup of $\hat\cO_F^\times$. 
\end{remark}
\begin{proposition}\label{ThmintMS}
Let $\Phi\in \cS_{\underline k}(U)$ and let $\rho:\I_F/F^\times\rightarrow\C$ be a locally polynomial character such that 
\[
\rho\mid_{F_\infty^\times}(t)=\rho_0(t)t^{\underline m};\qquad \frac{2-\underline k}{2}\leq\underline m\leq \frac{\underline k-2}{2},
\]
for some locally constant character $\rho_0$. If $\rho_{0,\sigma}(-1)=\varepsilon_\sigma$ for all $\sigma\mid\infty$ (in particular $\varepsilon$ is of lowest degree), then 
\[
\varphi({\rm MS}_\varepsilon(\Phi)(\infty-0)\cup\rho)\cap \eta_{\G_m}=\int_{\I_F/F^\times}\rho(t)\cdot \Phi(\underline{\delta s}_{\varepsilon_\R}(\mu_{\underline{m}}))\left(\begin{array}{cc}
t     &  \\
     & 1
\end{array}\right)d^\times t.
\]
\end{proposition}
\begin{proof}
The above fixed embedding $\G_m\hookrightarrow G$ provides the subspace
\[
s_\cO:=F^\times\backslash (M\times \I_F^\infty/\cO)\subseteq S_U=G(F)\backslash (\mathbb{H}\times G(\A_F^\infty)/U);\qquad \cO=U\cap\I_F^\infty,
\]
where $M\simeq \R_+^{q}$ is the connected component of $F_\infty^\times$. Notice that $M\simeq M_0\times \R_+$, where $M_0$ is the kernel of the norm morphism $N:M\rightarrow\R_+$. We can identify such an $\R_+$ as the geodesic joining the cusp $0$ and $\infty$. Since $K^T_{\infty,+}=K_{\infty,+}\cap\G_m(F_\infty)_+$ and $\mathfrak{g}^T_\infty/\cK^T_\infty=\bigoplus_\sigma \R D_\sigma$, where $\cK^T_\infty={\rm Lie}(K^T_{\infty,+})$, $\mathfrak{g}^T_\infty={\rm Lie}(F_\infty^\times)$ and $D_\sigma=\big(\begin{smallmatrix}
    1&\\&
\end{smallmatrix}\big)\in {\rm Lie}(F_\sigma^\times)$,
this implies that the differential associated to ${\rm MS}_\varepsilon(\Phi)(\infty-0)_{\mid s_\cO}$ is
\[
\Omega\big({\rm MS}_\varepsilon(\Phi)(\infty-0)_{\mid s_\cO}\big)=\int_0^\infty \iota\big(\Phi\big( c_\varepsilon\big(\bigwedge_{\sigma\mid\infty} D_\sigma\big)\big)\big)_{\mid \I_F}\bigwedge_{\sigma\mid\infty} dD_\sigma,
\]
where we regard $\Phi\left( c_\varepsilon\left(\bigwedge_{\sigma\mid\infty} D_\sigma\right)\right)\in \Hom(V(\underline k-2),\cA(U)^{G(F)})$ as an element in $C^\infty(G(\A_F)/U,V(\underline k-2))^{G(F)_+}$ by means of the embedding (see Remark \ref{isoCvsCtensV})
\[
\iota:\Hom(V(\underline k-2),\cA(U)^{G(F)})\longrightarrow C^\infty(G(\A_F)/U,V(\underline k-2))^{G(F)_+};\qquad\langle\iota(\varphi)(g),\mu\rangle=\varphi(g_\infty^{-1}\mu)(g),
\]
for all $g=(g_\infty,g_f)\in G(\A_F)$ and $\mu\in V(\underline k-2)$.
First notice that $\bigwedge_{\sigma\mid\infty} dD_\sigma=\prod_\sigma d^\times r_\sigma$ is the Haar measure of $M$ described above. Moreover, by construction 
\[
\delta s_\pm=c_{1,\sigma}^\pm(D_\sigma)\in \Hom(V(k_\sigma-2),D(k_\sigma));\qquad c_\varepsilon\left(\bigwedge_{\sigma\mid\infty} D_\sigma\right)=\underline{\delta s}_{\varepsilon_\R}.
\]
This implies that, for a small enough open compact subgroup ${\rm U}\subset \I_F^\infty$ so that both $\Phi$ and $\rho$ are ${\rm U}$-invariant, if we write $\Gamma={\rm U}\cap F_+^\times$ and we consider the identification
\[
\bigotimes_{x_i\in \Pic_{\G_m}({\rm U})}H^{q-1}(\Gamma,\C)\simeq H^{q-1}(F_+^\times,C^0(\I_F^\infty,\C)^{\rm U}), 
\]
provided by $C^0(\I_F^\infty,\C)^{\rm U}\simeq\bigoplus_{x_i \in \Pic_{\G_m}({\rm U})}{\rm coInd}_{\Gamma}^{F^\times_+}\C$, then $\varphi\left({\rm MS}_\varepsilon(\Phi)(\infty-0)\cap \rho\right)$ corresponds to the differential
\[
\Omega\left(\varphi\left({\rm MS}_\varepsilon(\Phi)(\infty-0)\cap \rho\right)\right)=\left(\rho(x_i)\int_0^\infty r_\infty^{\underline m}\Phi({\delta s}_{\varepsilon_\R}(\mu_{\underline m}))((r_\sigma)_\sigma,x_i)\prod_\sigma d^\times r_\sigma\right)_{i}.
\] 
Hence, if we write $d^\times r_\infty:=\prod_\sigma d^\times r_\sigma$ and $d^\times x_\infty:=\prod_\sigma d^\times x_\sigma=\prod_\sigma d^\times r_\sigma dk_\sigma$, then we obtain by definition of $\eta_{\G_m}$ 
\begin{eqnarray*}
    \varphi({\rm MS}_\varepsilon(\Phi)(\infty-0)\cup\rho)\cap \eta_{\G_m}&=&\sum_{x_i\in\Pic_{\G_m}(\rm U)}\frac{{\rm vol}({\rm U}){\rm vol}(k_\infty^0)\#( F^\times/F^\times_+)}{\#(\Gamma\cap k_\infty^0)}\int_0^\infty \int_{\Gamma\backslash M_0}r_\infty^{\underline m}\rho(x_i)\Phi({\delta s}_{\varepsilon_\R}(\mu_{\underline m}))(r_\infty,x_i) d^\times r_\infty\\
    &=&\frac{{\rm vol}(k_\infty^0)\#( F^\times/F^\times_+)}{\#(k_\infty^0\cap\Gamma)}\int_{\bigsqcup{\tilde x_i}{\rm U}}\int_{\Gamma\backslash M}\rho(x)r_\infty^{\underline m}\Phi({\delta s}_{\varepsilon_\R}(\mu_{\underline m}))(r_\infty,x) d^\times r_\infty d^\times x\\
    &=&\sum_{z\in \G_m(F)/\G_m(F)_+}\int_{\bigsqcup{\tilde x_i}{\rm U}}\int_{\Gamma\backslash \G_m(F_\infty)_+}\varepsilon_\R(z)\rho(x)x_\infty^{\underline m}\Phi({\delta s}_{\varepsilon_\R}(\mu_{\underline m}))(zx_\infty,x) d^\times x_\infty d^\times x,
\end{eqnarray*}
where the last equality follows from the fact that $\underline{\delta s}_{\varepsilon_\R}=c_\varepsilon\left(\wedge_{\sigma\mid\infty} D_\sigma\right)$ commutes with the elements $\alpha\in K^T_\infty$, 
\begin{equation}\label{actTdeltas}
    \alpha \underline{\delta s}_{\varepsilon_\R}(\mu_{\underline m})=\underline{\delta s}_{\varepsilon_\R}(\alpha\mu_{\underline m})=
\varepsilon_\R(\alpha)\alpha^{-\underline m}\underline{\delta s}_{\varepsilon_\R}(\mu_{\underline m}).
\end{equation}
Hence, the result follows.
\end{proof}

Let $\Pi$ be a cuspidal automorphic representation for $\PGL_2$ of weight $\underline k$ and level $N$. The we can use the above proposition to relate periods and critical values of L-functions. Indeed, we recover in our setting a classical result due to Shimura. In particular, this implies that the notion of periods used in the present article coincide with the notion found elsewhere.
\begin{corollary}\label{ShimRel}
    Let $\rho:\I_F/F^\times\rightarrow\C$ be a locally constant character. If $\rho_{\sigma}(-1)=\varepsilon_\sigma$ for all $\sigma\in\Sigma_F$, we have that 
    \[
    \frac{L(1/2,\Pi, \rho)}{|d_F|^{\frac{1}{2}}\cdot(2\pi)^{\frac{\underline k}{2}}\cdot\mathfrak{g}(\rho)\cdot \Omega_\varepsilon^\Pi}\,\text{ belongs to the field }\, L_{\Pi}(\rho),
    \]
    where $L_{\Pi}(\rho):=L_\Pi\otimes\Q(\rho)$, $\Q(\rho)$ is the field of coefficients of $\rho$.
\end{corollary}
\begin{proof}
Let us consider the normalized form $\Psi\in \cM_{\underline k}(U_0(N))$ generating $\Pi^\infty$. Since $\Pi$ is cuspidal, $\Psi\in \cS_{\underline k}(U_0(N))$, hence, we can define the modular symbol ${\rm MS}_\varepsilon(\Psi)$.
Since
\[
\frac{{\rm ES}_\varepsilon(\Psi)}{\Omega_\varepsilon^\Pi}=\frac{\partial\circ{\rm MS}_\varepsilon(\Psi)}{\Omega_\varepsilon^\Pi}\in H^{r}(G(F)_+,\cA^{\infty}(V(\underline k-2)_{L_\Pi})^{U_0(N)})^{\varepsilon_\R},
\]
we deduce that ${\rm MS}_\varepsilon(\Psi)/\Omega_\varepsilon^\Pi$ has coefficients in $V(\underline{k}-2)_{L_\Pi}$. Given the locally constant character $\rho$, write $c_v=\delta_v^{-1}y_v^{-1}$ for any non-archimedean place $v$, where $y_v\in F_v^\times$ satisfies $d_{F_v}^{-1}=y_v\cdot\mathfrak{c}(\rho_v)$ as in \S\ref{Gausssumssec}. Let us consider 
\[
b=(b_v)_{v\nmid\infty}\in G(\A_F^\infty);\qquad b_v=\left(\begin{array}{cc}
    1&c_v^{-1}\\&1
\end{array}\right).
\]
It is clear that 
\[
\frac{b{\rm MS}_\varepsilon(\Psi)}{\Omega_\varepsilon^\Pi}=\frac{{\rm MS}_\varepsilon(b\Psi)}{\Omega_\varepsilon^\Pi}\in H^{r-1}(G(F)_+,\cA^{\infty}(\Delta_0,V(\underline k-2)_{L_\Pi})^{b^{-1}U_0(N)})^{\varepsilon_\R}.
\]
Since $\rho$ corresponds to $\rho\mid_{\I_F^\infty}\otimes \mu_{\underline{0}}\in C^{0}(\I_F,\Q(\rho))\otimes V(\underline{k}-2)_{L_\Pi}$, we obtain by Proposition \ref{ThmintMS} and Remark \ref{algebraicitypairing1}
\[
|d_F|^{\frac{-1}{2}}L_\Pi(\rho)\ni\frac{\varphi({\rm MS}_\varepsilon(b\Psi)(\infty-0)\cup\rho)\cap \eta_{\G_m}}{\Omega_\varepsilon^\Pi}=\frac{1}{\Omega_\varepsilon^\Pi}\int_{\I_F/F^\times}\rho(t)\cdot b\Psi(\underline{\delta s}_{\varepsilon_\R}(\mu_{\underline{0}}))\left(\begin{array}{cc}
t     &  \\
     & 1
\end{array}\right)d^\times t.
\]

Write $\lambda=\varepsilon_\R$ and let us consider the Whittaker model element
\[
W_\lambda:\PGL_2(\A_F)\longrightarrow \C;\qquad W_\lambda(g)=\int_{\A_F/F}\Psi(\underline{\delta s}_\lambda(\mu_{\underline{0}}))\left(\left(\begin{array}{cc}
1     &x  \\
     & 1
\end{array}\right)g\right)\psi(-x)dx
\]
By \cite[theorem 3.5.5]{Bump} we have a Fourier expansion
\[
\Psi(\underline{\delta s}_\lambda(\mu_{\underline{0}}))(g)=\sum_{a\in F^\times}W_\lambda\left(\left(\begin{array}{cc}
a     &  \\
     & 1
\end{array}\right)g\right).
\]
Since $W_\lambda=\prod_vW_{\lambda,v}$, this implies that
\[
\int_{\I_F/F^\times}\rho(t)\cdot b\Psi(\underline{\delta s}_\lambda(\mu_{\underline{0}}))\big(\begin{smallmatrix}
t     &  \\
     & 1    
\end{smallmatrix}
\big)d^\times t=\sum_{a\in F^\times}\int_{\I_F/F^\times}\rho(at)\cdot W_\lambda\left(\big(\begin{smallmatrix}at     &  \\     & 1    
\end{smallmatrix}\big)b\right)d^\times t=\prod_v\int_{F_v^\times}\rho_v(t_v)W_{\lambda,v}\left(\big(\begin{smallmatrix}t_v     &  \\     & 1    
\end{smallmatrix}\big)b_v\right)dt_v^\times.
\]
By the definition of $\Psi$ provided in \eqref{eqdefPsi},
\[
\int_{\A_F^\times/F^\times}\Psi(\underline {\delta s}_1(\mu_{\underline 0}))\left(\begin{array}{cc}a&\\&1\end{array}\right)d^\times a=\Lambda(1/2,\Pi)=\prod_{v}L(1/2,\Pi_v).
\]
Thus,  we can assume that $\int_{F_v^\times}W_{1,v}\big(\begin{smallmatrix}t_v     &  \\     & 1    
\end{smallmatrix}\big)dt_v^\times=L(1/2,\Pi_v)$ for all $v$. If $G(F_\sigma)\simeq\PGL_2(\C)$ then we have  
\[
\int_{F_\sigma^\times}\rho_\sigma(t_\sigma)W_{\lambda,\sigma}\big(\begin{smallmatrix}t_\sigma     &  \\     & 1    
\end{smallmatrix}\big)dt_\sigma^\times=L(1/2,\Pi_\sigma, \rho_\sigma),
\]
because $\rho_\sigma=1$.
If $G(F_\sigma)\simeq\PGL_2(\R)$ then
\[
\int_{F_\sigma^\times}\rho_\sigma(t_\sigma)W_{\lambda,\sigma}\big(\begin{smallmatrix}t_\sigma     &  \\     & 1    
\end{smallmatrix}\big)dt_\sigma^\times=\int_{\R_+}W_{\lambda,\sigma}\big(\begin{smallmatrix}t_\sigma     &  \\     & 1    
\end{smallmatrix}\big)dt_\sigma^\times+\int_{\R_+}\lambda_\sigma(-1)W_{\lambda,\sigma}\big(\begin{smallmatrix}-t_\sigma     &  \\     & 1    
\end{smallmatrix}\big)dt_\sigma^\times;\qquad\lambda_\sigma=\lambda\mid_{F_\sigma^\times}.
\]
Since $w_\sigma \underline{\delta s}_\lambda(\mu_{\underline{0}})=\lambda_\sigma(-1)\underline{\delta s}_\lambda(\mu_{\underline{0}})$, where $w_\sigma=\big(\begin{smallmatrix}-1     &  \\     & 1    
\end{smallmatrix}\big)\in G(F_\sigma)$ (see \eqref{actTdeltas}), we deduce that 
\[
\int_{F_\sigma^\times}\rho_\sigma(t_\sigma)W_{\lambda,\sigma}\big(\begin{smallmatrix}t_\sigma     &  \\     & 1    
\end{smallmatrix}\big)dt_\sigma^\times=2\int_{\R_+}W_{\lambda,\sigma}\big(\begin{smallmatrix}t_\sigma     &  \\     & 1    
\end{smallmatrix}\big)dt_\sigma^\times=L(1/2,\Pi_\sigma)=L(1/2,\Pi_\sigma, \rho_\sigma),
\]
where the second equality is clear if $\lambda_\sigma=1$ and, if otherwise $\lambda_\sigma\neq 1$, it follows from the fact that $(W_{1,\sigma}-W_{\lambda,\sigma})\big(\begin{smallmatrix}t_\sigma     &  \\     & 1    
\end{smallmatrix}\big)$ is supported in $\R^\times\setminus\R_+$ since it is the antiholomorphic test vector of the Kirillov model.

If $v$ is non-archimedean, $\phi^0_v(t_v):=W_v\big(\begin{smallmatrix}t_v     &  \\     & 1    
\end{smallmatrix}\big)$ is in the Kirillov model of $\Pi_v$.  Moreover, in \cite[P. 23]{Schmidt2002} we have an explicit description of $\phi_v^0$:
If $d_{F_v}=\delta_v\cO_{F_v}$ then we have 
\begin{enumerate}
    \item If $\Pi_v\simeq\pi(\xi,\xi^{-1})$ spherical, 
    \[
    \phi^0_v(y)=|y|^{1/2}\sum_{k+l=v(y\delta_v)\colon k,l\geq 0}\xi(\varpi_v)^{k-l}1_{\cO_{F_v}}(\delta_vy).
    \]

    \item If $\Pi_v\simeq\sigma(\xi,\xi^{-1})$ with $\xi$ unramified,
    \[
    \phi^0_v(y)=|y|^{1/2}\xi(\delta_vy)1_{\cO_{F_v}}(\delta_vy).
    \]

    \item $\phi^0_v(y)=|\delta_v|^{-1/2}1_{\cO^\times_{F_v}}(\delta_vy)$ otherwise.
\end{enumerate}
We will compute that, 
\[
    \int_{F_v^\times}\rho_v(t_v)W_{\lambda,v}\left(\big(\begin{smallmatrix}t_v     &  \\     & 1    
\end{smallmatrix}\big)b_v\right)dt_v^\times=\int_{F_v^\times}\rho_v(t_v)\left(\begin{array}{cc}
    1&c_v^{-1}\\&1
\end{array}\right)\phi_v^0(t_v)dt_v^\times=\rho_v(\delta_v)^{-1}\cdot\mathfrak{g}(\rho_v^{-1},y_v)\cdot L(1/2,\Pi_v, \rho_v).
    \]
Indeed, we recall that $d_{F_v}^{-1}=y_v\mathfrak{c}(\rho_v)$, where $\mathfrak{c}(\rho_v)$ is the conductor of $\rho_v$, hence $\mathfrak{c}(\rho_v)=c_v\cO_{F_v}$. By the action on the Kirillov model:
\begin{itemize}
\item If $\Pi_v\simeq\pi(\xi,\xi^{-1})$ spherical and we write $n=v(t_v)+v(\delta_v)$
    \begin{eqnarray*}
        \int_{F_v^\times}\rho_v(t_v)\big(\begin{smallmatrix}
    1&c^{-1}\\&1
\end{smallmatrix}\big)\phi_v^0(t_v)dt_v^\times&=&\sum_{n\geq 0}\int_{\varpi_v^{n}\delta_v^{-1}\cO_{F_v}^\times}\rho_v(t_v)\psi_v(c_v^{-1}t_v)|t_v|^{1/2}\sum_{k+l=v(t_v\delta_v)\colon k,l\geq 0}\xi(\varpi_v)^{k-l}d^\times t_v\\
        &=&\sum_{n\geq 0}\frac{\rho_v(\varpi_v)^{n}}{\rho_v(\delta_v)}\frac{|\varpi_v|^{n/2}}{|\delta_v|^{1/2}}\sum_{k= 0}^{n}\xi(\varpi_v)^{2k-n}\int_{\cO_{F_v^{\times}}}\rho_v(x_v)\psi_v(c_v^{-1}\delta_{v}^{-1}\varpi_v^{n}x_v)d^\times x_v\\
        &=&\sum_{n\geq 0}\frac{\rho_v(\varpi_v)^{n}\xi(\varpi_v)^{-n}}{\rho_v(\delta_v)}q_v^{-n/2}\frac{1-\xi(\varpi_v)^{2n+2}}{1-\xi(\varpi_v)^{2}}\frac{1}{{\rm vol}(\cO_{F_v}^\times)}\int_{\cO_{F_v^{\times}}}\rho_v(x_v)\psi_v(y_v\varpi_v^{n}x_v)d^\times x_v.        
    \end{eqnarray*}
    Write $I_n=\frac{1}{{\rm vol}(\cO_{F_v}^\times)}\int_{\cO_{F_v^{\times}}}\rho_v(x_v)\psi_v(y_v\varpi_v^{n}x_v)d^\times x_v$. By \cite[Lemma 2.2]{Spiess}, if $v(c_v)\neq 0$ then $I_n=0$ unless $n=0$ and $I_0=\mathfrak{g}(\rho_v^{-1},y_v)$, moreover, if $v(c_v)=0$ then $I_n=1=\mathfrak{g}(\rho_v^{-1},y_v)$, for all $n\geq 0$. We conclude that
    \[
    \int_{F_v^\times}\rho_v(t_v)\big(\begin{smallmatrix}
    1&c_v^{-1}\\&1
\end{smallmatrix}\big)\phi_v^0(t_v)dt_v^\times=\rho_v(\delta_v)^{-1}\cdot\mathfrak{g}(\rho_v^{-1},y_v)\cdot L(1/2,\Pi_v, \rho_v).
    \]

 \item If $\Pi_v\simeq\sigma(\xi,\xi^{-1})$ with $\xi$ unramified and we write again $n=v(t_v)+v(\delta_v)$,
    \begin{eqnarray*}
        \int_{F_v^\times}\rho_v(t_v)\big(\begin{smallmatrix}
    1&c^{-1}\\&1
\end{smallmatrix}\big)\phi_v^0(t_v)dt_v^\times&=&\sum_{n\geq 0}\int_{\varpi_v^{n}\delta_v^{-1}\cO_{F_v}^\times}\rho_v(t_v)\psi_v(c_v^{-1}t_v)|t_v|^{1/2}\xi(\delta_vt_v)dt_v^\times\\
        &=&\sum_{n\geq 0}\rho_v(\varpi_v)^{n}\rho_{v}(\delta_v)^{-1}\frac{|\varpi_v|^{n/2}}{|\delta_v|^{1/2}}\xi(\varpi_v)^n\int_{\cO_{F_v^{\times}}}\rho_v(x_v)\psi_v(c_v^{-1}\delta_{v}^{-1}\varpi_v^{n}x_v)d^\times x_v\\
        &=&\rho_v(\delta_v)^{-1}\sum_{n\geq 0}\rho_v(\varpi_v)^{n}q_v^{-n/2}\xi(\varpi_v)^n I_n=\rho_v(\delta_v)^{-1}\cdot \mathfrak{g}(\rho_v^{-1},y_v)\cdot L(1/2,\Pi_v,\rho_v).
    \end{eqnarray*}
    
    \item If $\phi^0_v(y)=|\delta_v|^{-1/2}1_{\cO^\times_{F_v}}(\delta_vy)$, we can make the change of variables $x_v=\delta_v t_v$ to obtain
    \begin{eqnarray*}
        \int_{F_v^\times}\rho_v(t_v)\big(\begin{smallmatrix}
    1&c_v^{-1}\\&1
\end{smallmatrix}\big)\phi_v^0(t_v)dt_v^\times&=&|\delta_v|^{-1/2}\int_{F_v^\times}\rho_v(t_v)\psi_v(c_v^{-1}t_v)1_{\cO^\times_{F_v}}(\delta_vt_v)dt_v^\times\\
&=&\frac{1}{{\rm vol}(\cO_{F_v}^\times)}\int_{\cO_{F_v}^\times}\rho_v(\delta_{v}^{-1}x_v)\psi_v(y_vx_v)dx_v^\times=\rho_v(\delta_v)^{-1}\mathfrak{g}(\rho_v^{-1},y_v)\\
&=&\rho_v(\delta_v)^{-1}\cdot\mathfrak{g}(\rho_v^{-1},y_v)\cdot L(1/2,\Pi_v,\rho_v).
    \end{eqnarray*}
\end{itemize}
In conclusion, we have obtained that 
\begin{eqnarray*}
|d_F|^{\frac{-1}{2}}L_\Pi(\rho)&\ni&(\Omega_\varepsilon^\Pi)^{-1}\prod_v\int_{F_v^\times}\rho_v(t_v)W_{\lambda,v}\left(\big(\begin{smallmatrix}t_v     &  \\     & 1    
\end{smallmatrix}\big)b_v\right)dt_v^\times=(\Omega_\varepsilon^\Pi)^{-1}\prod_v\rho_v(\delta_v)^{-1}\cdot\mathfrak{g}(\rho_v^{-1},y_v)\cdot L(1/2,\Pi_v,\rho_v)\\
&=&\mathfrak{g}(\rho^{-1})\rho(d_F)^{-1}(\Omega_\varepsilon^\Pi)^{-1}\Lambda(1/2,\Pi, \rho),
\end{eqnarray*}
and using Proposition \ref{propertiesGaussS} the result follows.

\end{proof}

\subsection{Gross formula in higher cohomology}

In this section, we return to the general setting where $G$ is associated with an arbitrary quaternion algebra $B$, recalling that we have fixed an embedding $E\hookrightarrow B$ of a quadratic extension $E/F$ satisfying Assumption \ref{assuSigmaSigma}. Let $\pi$ be an automorphic representation for $G$ of weight $\underline k$ and level $N$, and let $\Pi$ be it Jacquet-Langlands lift to $\PGL_2$.
In \S \ref{fundClasses} we construct a fundamental class associated with $T\subset G$, the algebraic group corresponding to $E^\times/F^\times$,
\[
\eta_T\in H_{r_B}(T(F)_+,C^0_c(T(\A_F^\infty),\Q)).
\]
Moreover, we have a natural $T(F)_+$-equivariant pairing analogous to that of \eqref{firstvarphi}
\begin{eqnarray*}
\varphi:\left(C^0(T(\A_F^\infty),{L_{\underline k}})\otimes V(\underline{k}-2)_{L_{\underline k}}\right)\times \cA^{\infty}(V(\underline{k}-2)_{L_{\underline k}})^U&\longrightarrow& C^0(T(\A_F^\infty),L_{\underline k}),\\
\varphi((f\otimes\mu)\otimes\Phi)(t)&=&f(t)\cdot \langle\Phi(t),\mu\rangle.
\end{eqnarray*}
Recall that $C^0(T(\A_F^\infty),{L_{\underline k}})\otimes V(\underline{k}-2)_{L_{\underline k}}$ can be seen as a subspace of the set of locally polynomial functions in $T(\A_F)$ by means of \eqref{PtoCoverC} and \eqref{dualVP}. Indeed, analogously as in \S\ref{MSandDT}, any locally polynomial character $\chi:T(\A_F)/T(F)\rightarrow\C$ of degree less that $\frac{\underline k-2}{2}$, 
\[
\chi\mid_{T(F_\infty)}(t)=\chi_0(t)t^{\underline m};\qquad \frac{2-\underline k}{2}\leq\underline m\leq \frac{\underline k-2}{2},
\]
can be interpreted as the element
\[
\chi=\chi\mid_{T(\A_F^\infty)}\otimes\mu_{\underline m} \in H^0(T(F)_+,C^{0}(T(\A_F^\infty),\C)\otimes V(\underline{k}-2)).
\]
Since the Tamagawa measure $d^\times t$ on $T(\A_F^\infty)$ provides a $T(F)$-invariant pairing 
\begin{equation}\label{pairing2CC}
    C^0(T(\A^\infty_F),\C)\times C_c^0(T(\A_F^\infty),\C)\longrightarrow\C;\qquad (f_1,f_2)\longmapsto {\rm vol}(K^T_{\infty,+})\cdot\#( F^\times/F^\times_+)\cdot\int_{T(\A_F^\infty)}f_1(z,t)f_2(z,t)d^\times t,
\end{equation}
for any such a $\chi$, any $\Phi\in \cM_{\underline k}(U)$ and any $\varepsilon\in \{\pm 1\}^{\Sigma_B}$ of lowest degree (namely,  $n_\varepsilon=r_B$),
we can consider, 
\[
\cP(\Phi,\varepsilon,\chi):=\varphi({\rm ES}_\varepsilon(\Phi)\cup\chi)\cap\eta_T\in \C.
\]
\begin{remark}\label{algebraicitypairing}
    For our choice of Haar measure given in \S\ref{haarmeasures}, we have that ${\rm vol}(K^T_{\infty,+})=2^{r_1^B}4^{r_2}$. Moreover, for any maximal order $\cO_E\subset E$, we have ${\rm vol}(\hat\cO_E^\times/\hat\cO_F^\times)=|d_FD|^{\frac{-1}{2}}$, where $D$ is the relative discriminant of $E/F$. This implies that the pairing \eqref{pairing2CC} is the extension of scalars of a pairing
    \[
    C^0(T(\A_F^\infty),\Q)\times C_c^0(T(\A_F^\infty),\Q)\longrightarrow |d_FD|^{\frac{-1}{2}}\Q,
    \]
    since any open compact subgroup $\cO\subset T(\A_F^\infty)$ must be a finite index subgroup of $\hat\cO_E^\times/\hat\cO_F^\times$.
\end{remark}

In order to state the main result of this section we will make the following simplifying assumption (but see \cite{preprintsanti2} for the result in a more general situation).
\begin{assumption}\label{Assonchi}
    As above, let $N$ be the level of the automorphic representation $\pi$ and let $c$ be the conductor of  $\chi$. We will assume the following:
    \begin{itemize}
        \item For all finite places $v$, either $\ord_v(c)\geq \ord_v(N)$ or $\ord_v(c)=0$ with $\ord_v(N)\leq 1$ if $v$ is non-split in $E$.
        
        \item  For all finite places $v$, all local root numbers satisfy $\epsilon(1/2,\pi_v,\chi_v)=\psi_{T,v}(-1)\epsilon(B_v)$, where $\epsilon(B_v)=1$ if $B_v$ is a matrix algebra and $\epsilon(B_v)=-1$ otherwise.
    \end{itemize}
\end{assumption}
\begin{theorem}\label{mainTHMintro}
Let $\chi:T(\A_F)/T(F)\rightarrow\C^\times$ be a character of conductor $c$ satisfying Assumption \ref{Assonchi} with  
\[
\chi\mid_{T(F_\infty)}(t)=\chi_0(t)t^{\underline m},\qquad \underline{m}=(m_{\tilde\sigma})\in \Z^{[F:\Q]},
\] 
for some locally constant character $\chi_0$ and some $\underline{m}$ with $\frac{2-k_{\nu}}{2}\leq m_{\nu}\leq\frac{k_{\nu}-2}{2}$.  Then there exists an  Eichler order $\cO_N\subset B$ of discriminant $N$
 such that, for any $\Phi\in \cM_{\underline k}(U_N)$, where $U_N=\hat\cO_N^\times$,
we have
\[
\cP(\Phi,\varepsilon,\chi)\cdot \cP(\Phi,\varepsilon,\chi^{-1})=\frac{2^{\#S_D}L_{c}(1,\psi_{T})^2
C(\underline k,\underline m)}{|c^2 D|^{\frac{1}{2}}}\cdot L^S(1/2,\Pi,\chi)\cdot\frac{\langle \Phi,\Phi\rangle}{\langle \Psi,\Psi\rangle}\cdot\frac{{\rm vol}(U_0(N))}{{\rm vol}(U_N)},
\]
where $\varepsilon$ is the lowest degree sign such that $\chi_{0,\sigma}(-1)=\varepsilon_\sigma$, $S:=\{v\mid (N,c)\}$, $S_D:=\{v\mid (N,D)\colon \ord_v(c)=0\}$, $L^S(s,\Pi,\chi)$ is the L-function with the local factors at places in $S\cup\Sigma_F$ removed, $L_{c}(s,\psi_{T})$ is the product of the local factors at places dividing $c$, and
    \[
    C(\underline k,\underline m)=(-1)^{\left(\sum_{\sigma\not\in\Sigma_B}\frac{k_{\sigma}-2}{2}\right)}4^{r_{1,B}+2r_2}\pi^{-r_1^B}\prod_{\nu:F\hookrightarrow\C}\frac{\Gamma(\frac{k_{\nu}}{2}-m_{\nu})\Gamma(\frac{k_{\nu}}{2}+m_{\nu})}{(-1)^{m_{\nu}}(2\pi)^{k_{\nu}}}.
    \]
\end{theorem}
\begin{proof}
This is a particular case of \cite[Theorem 1.2]{preprintsanti2}. We can also give the following sketch of an alternative proof, using the machinery introduced in the present article. Similarly as in the proof of Proposition \ref{ThmintMS}, the embedding $T\hookrightarrow G$ provides the subspace
\[
s_\cO:=T(F)\backslash (M\times T(\A_F^\infty)/\cO)\subseteq S_U=G(F)\backslash (\mathbb{H}\times G(\A_F^\infty)/U);\qquad \cO=U\cap\I_F^\infty,
\]
where $M\simeq \R_+^{r_B}$ is the connected component of $T(F_\infty)$. 
Since $K^T_{\infty,+}=K_{\infty,+}\cap T(F_\infty)_+$ and $\mathfrak{g}^T_\infty/\cK^T_\infty=\bigoplus_\sigma \R D_\sigma$, where $\cK^T_\infty={\rm Lie}(K^T_{\infty,+})$, $\mathfrak{g}^T_\infty={\rm Lie}(T(F_\infty))$ and $D_\sigma$ corresponds to $\big(\begin{smallmatrix}
    1&\\&
\end{smallmatrix}\big)$ under the fixed isomorphism ${\rm Lie}(G(F_\sigma))\simeq {\rm Lie}(\PGL_2(F_\sigma))$,
this implies that the differential associated to ${\rm ES}_\varepsilon(\Phi)_{\mid s_\cO}$ is
\[
\Omega\big({\rm ES}_\varepsilon(\Phi)_{\mid s_\cO}\big)=\int_0^\infty \iota\big(\Phi\big( c_\varepsilon\big(\bigwedge_{\sigma\mid\infty} D_\sigma\big)\big)\big)_{\mid T(\A_F)}\bigwedge_{\sigma\mid\infty} dD_\sigma.
\]
Following the same steps as those in the proof of Proposition \ref{ThmintMS}, we deduce that
\[
\cP(\Phi,\varepsilon,\chi)=\varphi({\rm ES}_\varepsilon(\Phi)\cup\chi)\cap \eta_{T}=\int_{T(\A_F)/T(F)}\chi(t)\cdot \Phi(\underline{\delta s}_{\varepsilon_\R}(\mu_{\underline{m}}))(t)d^\times t.
\]
Finally, Waldspurger formula provides {\small$\left(\int_{T(\A_F)/T(F)}\chi(t)\cdot \Phi(\underline{\delta s}_{\varepsilon_\R}(\mu_{\underline{m}}))(t)d^\times t\right)\cdot \left(\int_{T(\A_F)/T(F)}\chi^{-1}(t)\cdot \Phi(\underline{\delta s}_{\varepsilon_\R}(\mu_{\underline{m}}))(t)d^\times t\right)$} in terms of the local constituents of the automorphic form $\Phi(\underline{\delta s}_{\varepsilon_\R}(\mu_{\underline{m}}))$. In \cite{preprintsanti2}, the factors associated to such local constituents are computed using the archimedean local theory of \S \ref{LocArchRep} and the explicit versions of Waldspurger formula of \cite{CST}. The result follows from such a computation (see \cite[Theorem 1.2]{preprintsanti2}).
\end{proof}

\begin{corollary}\label{AlgebraicityofL}
    Let $\chi:T(\A_F)/T(F)\rightarrow\pm 1$ be a quadratic character satisfying Assumption \ref{Assonchi}. Then 
    \begin{equation}\label{formcorollGF}
\frac{L^S(1/2,\Pi,\chi)}{(-1)^{\left(\sum_{\sigma\not\in\Sigma_B}\frac{k_\sigma-2}{2}\right)}\cdot 2^{\#S_D}\cdot\pi^{r_1^B}\cdot\pi^{\underline k}\cdot(\Omega^\pi_\varepsilon)^{2}\cdot|cd_F|\cdot| D|^{-1/2} \cdot \alpha^{\frac{\underline{k}-2}{2}}}\,\text{ is a square in }\, L_\Pi,  
\end{equation}
where $\varepsilon$ is the unique lowest degree sign such that $\chi_{\sigma}(-1)=\varepsilon_\sigma$, and $\alpha\in F$ is such that $E=F(\sqrt{\alpha})$.
\end{corollary}
\begin{proof}
For the chosen Eichler order $\cO_N$, we consider the normalized form $\Phi_0\in \cM_{\underline k}(U_N)$ satisfying \eqref{condnormalization}. We obtain by Theorem \ref{mainTHMintro} (notice that $\chi=\chi^{-1}$ and $m_{\tilde\sigma}=0$),
\[
\left(\frac{\cP(\Phi_0,\varepsilon,\chi)}{\Omega^\pi_\varepsilon}\right)^2=\frac{2^{\#S_D}L_{c}(1,\psi_{T})^2
C(\underline k,\underline 0)}{|c^2 D|^{\frac{1}{2}}(\Omega^\pi_\varepsilon)^2}\cdot L^S(1/2,\Pi,\chi).
\]
By Lemma \ref{rationalinvariants2}, $\chi$ corresponds to an element $\chi\in \sqrt{\alpha^{\frac{\underline{k}-2}{2}}}H^0(T(F),C^0(T(\A_F^\infty),L_\Pi)\otimes V(\underline{k}-2)_{L_\Pi})$.
Since the periods $\Omega^\pi_\varepsilon$ are defined so that $(\Omega^\pi_\varepsilon)^{-1}{\rm ES}_\varepsilon(\Phi_0)$ has coefficients in $L_\Pi$, we obtain by Remark \ref{algebraicitypairing}
\[
\frac{\cP(\Phi_0,\varepsilon,\chi)}{\Omega^\pi_\varepsilon}=\varphi\left(\frac{{\rm ES}_\varepsilon(\Phi_0)}{\Omega^\pi_\varepsilon}\cup\chi\right)\cap\eta_T\in \sqrt{\alpha^{\frac{\underline{k}-2}{2}}}|d_FD|^{\frac{-1}{2}}L_\Pi.
\]
Thus, the result follows from the fact that $L_{c}(1,\psi_{T})^2\in\Q^2$.
\end{proof}
\begin{remark}
    If the weight is parallel $\underline k=(k,\cdots,k)$, then the formula \eqref{formcorollGF} simplifies to
    \[
    \frac{L^S(1/2,\Pi,\chi)}{2^{\#S_D}\cdot\pi^{r_1^B}\cdot\pi^{\underline k}\cdot(\Omega^\pi_\varepsilon)^{2}\cdot|cd_F| \cdot |D|^{\frac{1-k}{2}}}\,\text{ is a square in }\, L_\Pi.
    \]
\end{remark}

\subsection{A Petersson product formula}

Given any sign $\varepsilon=(\varepsilon_\sigma)_\sigma\in \{\pm1\}^{\Sigma_B}$, we can consider its opposite $-\varepsilon:=(-\varepsilon_\sigma)_\sigma\in \{\pm1\}^{\Sigma_B}$. The corresponding character is given by
\[
(-\varepsilon_\R)(g_\sigma)=\frac{\det(g_\sigma)}{|\det(g_\sigma)|}\varepsilon_\R(g_\sigma);\qquad F_\sigma=\R.
\]
Moreover, it is easy to compute that 
$n_{-\varepsilon}=r_{1,B}+\sum_{\sigma\in\Sigma_B^\C}\frac{3+\varepsilon_\sigma}{2}=2r_B-n_\varepsilon+r_2$.
Thus, for any pair $\Phi_1,\Phi_2\in \cM_{\underline k}(U)$, 
\[
{\rm ES}_\varepsilon(\Phi_1)\in H^{n_\varepsilon}(G(F)_+,\cA^\infty(V(\underline k-2))^U);\qquad {\rm ES}_{-\varepsilon}(\Phi_2)\in H^{2r_B-n_\varepsilon+r_2}(G(F)_+,\cA^\infty(V(\underline k-2))^U).
\]
Hence, by means of the natural $G(F)_+$-equivariant pairing:
\[
\kappa:\cA^\infty(V(\underline k-2))^U\times \cA^\infty(V(\underline k-2))^U\longrightarrow C^0(G(\A_F^\infty),\C);\qquad \kappa(\phi_1,\phi_2)=\phi_1\phi_2(\Upsilon),
\]
where $\Upsilon\in V(\underline k-2)^{\otimes 2}$ is as in \eqref{defUpsilon}, we can consider the cup product
\[
\kappa\left({\rm ES}_\varepsilon(\Phi_1)\cup {\rm ES}_{-\varepsilon}(\Phi_2)\right)\in H^{r_B+r_2}(G(F)_+,C^0(G(\A_F^\infty),\C)).
\]
Recall that we have constructed in \S\ref{fundClasses} the fundamental class
\[
\eta_G\in H_{r_B+r_2}(G(F)_+,C_c^0(G(\A_F^\infty),\Q)),
\]
and we have a natural pairing between $C_c^0(G(\A_F^\infty),\Q)$ and $C^0(G(\A_F^\infty),\C)$ provided by the Haar measure $dg_f$ of $G(\A_F^\infty)$.
The following result computes the cap product of the above cohomology and homology classes in terms of the Petersson product:
\begin{theorem}
    For any $\varepsilon\in \{\pm1\}^{\Sigma_B}$, we have that 
    \[
    \kappa\left({\rm ES}_\varepsilon(\Phi_1)\cup {\rm ES}_{-\varepsilon}(\Phi_2)\right)\cap\eta_G=K\cdot\langle \Phi_1,\Phi_2\rangle;\qquad K=\frac{3^{r_2}(2i)^{r_{1,B}}}{(2\pi^2)^{r_1^B+r_2}\pi^{r_{1,B}}}.
    \]
\end{theorem}
\begin{proof}
Notice that $\kappa\left({\rm ES}_\varepsilon(\Phi_1)\cup {\rm ES}_{-\varepsilon}(\Phi_2)\right)$ corresponds to the differential form
\[
\Omega\left({\rm ES}_\varepsilon(\Phi_1)\cup {\rm ES}_{-\varepsilon}(\Phi_2)\right)=(\Phi_1\Phi_2)((c_\varepsilon\cup c_{-\varepsilon})(\Upsilon))(W)dW
\]
for any basis $W\in \bigwedge^{r_B+r_2}\mathfrak{g}_\infty/\cK_\infty$, once we identify $(\Phi_1\Phi_2)((c_\varepsilon\cup c_{-\varepsilon})(\Upsilon))\in H^{r_B+r_2}((\mathfrak{g}_\infty,K_{\infty,+}),\cA(U)^{G(F)})$ with a cocycle in $\Hom_{K_{\infty,+}}(\bigwedge^{r_B+r_2}\mathfrak{g}_\infty/\cK_\infty,\cA(U)^{G(F)})$ representing it. By Proposition \ref{propcupprod1}, Remark \ref{remarkonhaarmeasurecocycle1}, Proposition \ref{propcupprod2} and Remark \ref{remarkonhaarmeasurecocycle2},
\[
\Omega\left({\rm ES}_\varepsilon(\Phi_1)\cup {\rm ES}_{-\varepsilon}(\Phi_2)\right)=K\cdot \int_{K_{\infty,+}}(f_1f_2)(\underline{\delta s}_1(\underline \Upsilon))(k,g_i,k,g_i)d^\times g_\infty.
\]

 If we write $G(F_\infty)_0=G(F_\infty)/K_\infty=G(F_\infty)_+/K_{\infty,+}$, we obtain (see \S \ref{IndepU}) 
    \begin{eqnarray*}
    \kappa\left({\rm ES}_\varepsilon(\Phi_1)\cup {\rm ES}_{-\varepsilon}(\Phi_2)\right)\cap\eta_G&=&\sum_{g_i\in {\rm Pic}_G(U)}\frac{K}{\#\cG_{g_i}}\int_{g_iU}\int_{\Gamma_{g_i}\backslash G(F_\infty)_0}\int_{K_{\infty,+}}(\Phi_1\Phi_2)(\underline{\delta s}_1(\underline \Upsilon))(k,g_i,k,g_i)d^\times g_\infty d^\times g_f\\
    &=&K\sum_{g_i\in {\rm Pic}_G(U)}\int_{g_iU}\int_{\Gamma_{g_i}\backslash G(F_\infty)_0}\int_{\cG_{g_i}\backslash K_{\infty,+}}(\Phi_1\Phi_2)(\underline{\delta s}_1(\underline \Upsilon))(g,g)d^\times g\\
    &=&K\int_{G(F)\backslash G(\A_F)}(\Phi_1\Phi_2)(\underline{\delta s}_1(\underline \Upsilon))(g,g)d^\times g=K\langle \Phi_1, \Phi_2\rangle,
    \end{eqnarray*}
    since $G(F)\backslash G(\A_F)\simeq G(F)_+\backslash (G(F_\infty)_+\times G(\A^\infty_F))\simeq\bigsqcup_{g_i\in \Pic_G(U)}(\Gamma_{g_i}\backslash G(F_\infty)_+)\times g_iU$.
\end{proof}
\begin{remark}
    Similarly as in Remark \ref{algebraicitypairing}, the pairing induced by $dg_f$ restricts to
    \[
    C^0(T(\A_F^\infty),\Q)\times C_c^0(T(\A_F^\infty),\Q)\longrightarrow {\rm vol}(U_N)\Q.
    \]
\end{remark}

As in previous sections, let $\pi$ be an automorphic representation for $G$ of weight $\underline k$ and level $N$, and let $\Pi$ be it Jacquet-Langlands lift to $\PGL_2$. In lowest and highest degree situations, we can consider the periods $\Omega_\varepsilon^\pi$ obtained by means of normalized forms generating $\pi$. The above theorem provides a nice relation between such periods:
\begin{corollary}\label{cor3}
    For any $\varepsilon\in \{\pm1\}^{\Sigma_B}$ of lowest degree,
    \[
\frac{L(1,\Pi,{\rm ad})}{\Omega^\pi_\varepsilon\cdot \Omega^\pi_{-\varepsilon}\cdot \pi^{2r_1^B+r_2}\cdot(\pi i)^{r_{1,B}}\cdot\pi^{\underline k} }\quad\text{belongs to}\; L_\Pi^\times. 
\]

\end{corollary}
\begin{proof}
With the notation of \S\ref{normforms}, let $\Phi_0\in \cM_{\underline k}(U_N)$ be a normalized form. By definition, 
\[
\frac{{\rm ES}_{\pm\varepsilon}(\Phi_0)}{\Omega^\pi_{\pm\varepsilon}}\in H^{n_{\pm\varepsilon}}(G(F)_+,\cA^{\infty}(V(\underline k-2)_{L_\Pi})^U).
\]
The existence of $\Upsilon\in \left(V(\underline k-2)^{\otimes 2}\right)^{G(F)}$ is equivalent to the isomorphism \eqref{dualVP}. Since by Remark \ref{dualVkmodelpairing}, such isomorphism is defined over $L_\Pi$, we deduce that $\Upsilon\in V(\underline k-2)_{L_\Pi}^{\otimes 2}$. This implies that
\[
\frac{\kappa\left({\rm ES}_\varepsilon(\Phi_0)\cup {\rm ES}_{-\varepsilon}(\Phi_0)\right)}{\Omega^\pi_\varepsilon \Omega^\pi_{-\varepsilon}}\in H^{r_B+r_2}(G(F)_+,C^0(G(\A_F^\infty),L_\Pi)).
\]
Thus, we obtain by \eqref{condnormalization} and the above remark
\[
\frac{K}{\Omega^\pi_\varepsilon\Omega^\pi_{-\varepsilon}}\cdot  \frac{\langle\Psi,\Psi\rangle}{{\rm vol}(U_0(N))}=\frac{K}{\Omega^\pi_\varepsilon\Omega^\pi_{-\varepsilon}}\cdot  \frac{\langle\Phi_0,\Phi_0\rangle}{{\rm vol}(U_N)}=\frac{\kappa\left({\rm ES}_\varepsilon(\Phi_0)\cup {\rm ES}_{-\varepsilon}(\Phi_0)\right)\cap\eta_G}{\Omega^\pi_\varepsilon \Omega^\pi_{-\varepsilon}{\rm vol}(U_N)}\in L_\Pi^\times.
\]
On the other side, by \cite[Proposition 2.1]{CST}, given decomposable $\mathfrak{f}_1,\mathfrak{f}_2\in \Pi$
    \begin{equation}\label{defpetersson}
    \int_{\PGL2(F)\backslash \PGL_2(\A_F)}\mathfrak{f}_1(g)\mathfrak{f}_2(g)d^\times g=2\Lambda(1,\Pi,{\rm ad})\cdot \Lambda_F(2)^{-1}\cdot\prod_v\alpha_v(W_{\mathfrak{f}_1,v},W^-_{\mathfrak{f}_2,v}),
\end{equation}
where $\Lambda$ stands for the completed L-function, the elements of the Whittaker model
\begin{eqnarray*}
    \prod_vW_{\mathfrak{f}_i,v}&=&W_{\mathfrak{f}_i}=\int_{\A_F/F}\mathfrak{f}_i\left(\left(\begin{array}{cc}1&x\\&1\end{array}\right)g\right)\psi(-x)dx;\qquad \prod_vW^-_{\mathfrak{f}_i,v}=W^-_{\mathfrak{f}_i}=\int_{\A_F/F}\mathfrak{f}_i\left(\left(\begin{array}{cc}1&x\\&1\end{array}\right)g\right)\psi(x)dx;\\
\end{eqnarray*}
 and the pairings $\alpha_v(W_{\mathfrak{f}_1,v},W^-_{\mathfrak{f}_2,v})$ are given by
\begin{equation*}\label{calcWhitt}
    \alpha_v(W_{\mathfrak{f}_1,v},W^-_{\mathfrak{f}_2,v})=\frac{\zeta_v(2)\cdot\langle W_{\mathfrak{f}_1,v},W^-_{\mathfrak{f}_2,v} \rangle_v}{\zeta_v(1)\cdot L(1,\Pi_v,{\rm ad})};\qquad \langle W_{\mathfrak{f}_1,v},W^-_{\mathfrak{f}_2,v} \rangle_v=\int_{F_v^\times}W_{\mathfrak{f}_2,v}\left(\begin{array}{cc}a&\\&1\end{array}\right)W^-_{\mathfrak{f}_2,v}\left(\begin{array}{cc}a&\\&1\end{array}\right)d^\times a.
\end{equation*}
We write $\mathfrak{f}_0:=\Psi(\underline{\delta s}_1(\mu_{\underline 0}))=\bigotimes_v\mathfrak{f}_{0,v}$. By \cite[Lemma 3.4]{preprintsanti2} and \cite[Proposition 3.11]{CST}
\begin{eqnarray*}
  \alpha_v(W_{\mathfrak{f}_0,v},W^-_{\mathfrak{f}_0,v})=\left\{\begin{array}{l}
      |d_F|_v^{-1/2}\ \  \text{ if } v\mbox{ is non-archimedean and }\pi_v\mbox{ unramified};  \\
    |d_F|_v^{-1/2}\zeta_v(2)\zeta_v(1)^{-1} L(1,\Pi_v,{\rm ad})^{-\delta_v}\ \  \text{ if } v\mbox{ is non-archimedean and }\pi_v\mbox{ ramified};\\
    \zeta_v(2) L(1,\Pi_v,{\rm ad})^{-1}2(4\pi)^{-k_v}\Gamma(k_v)\ \ \text{ if }F_v=\R;\\
    \zeta_v(2) L(1,\Pi_v,{\rm ad})^{-1}4(2\pi)^{1-k_{v_1}-k_{v_2}}\Gamma\left(\frac{k_{v_1}+k_{v_2}}{2}\right)^2\frac{\Gamma(k_{v_1})\Gamma(k_{v_2})}{\Gamma(k_{v_1}+k_{v_2})}\ \ \text{ if }F_v=\C;
  \end{array}\right. 
\end{eqnarray*}
where $\delta_v\in\{0,1\}$ and equals 0 when $\Pi_v$ is Steinberg.
By \cite[Lemmas 4.12, 4.14, 4.29, and 4.30]{preprintsanti2}
\begin{eqnarray*}
    \langle\Psi,\Psi\rangle&=&\int_{\PGL2(F)\backslash \PGL_2(\A_F)}\mathfrak{f}(g)\mathfrak{f}(g)d^\times g\prod_{\sigma\mid\infty}\frac{\langle{\delta_\sigma s}_1(\Upsilon)\rangle_{\sigma}}{\langle{\delta_\sigma s}_1(\mu_0),{\delta_\sigma s}_1(\mu_0)\rangle_\sigma}\\
    &=&2^{1-r_1}|d_F|^{1/2}(2\pi)^{-\underline k}\Gamma(\underline k) L(1,\Pi,{\rm ad})\cdot \zeta_F(2)^{-1}\prod_{F_\sigma=\C}\frac{4\pi}{3}(-1)^{\frac{\underline k_\sigma}{2}}\prod_{v\nmid\infty,v\,{\rm ram.}}\zeta_v(2)\zeta_v(1)^{-1} L(1,\Pi_v,{\rm ad})^{-\delta_v},
\end{eqnarray*}
where $\Gamma({\underline k})=\prod_{\nu:F\hookrightarrow\C}\Gamma(k_\nu)$. The result follows from the fact that ${\rm vol}(U_0(N))=\frac{{\rm vol}(U_0(1))}{[U_0(1):U_0(N)]}=\frac{\zeta_F(2)^{-1}|d_F|^{-3/2}}{[U_0(1):U_0(N)]}$ and $L(1,\Pi_v,{\rm ad})\in L_\pi$ for any finite place $v$.
\end{proof}

\appendix
\section{Beilinson's conjectures}\label{ApdxBC}

In this appendix, we explain a simplified version of Beilinson's conjectures and make them explicit in the special case of an elliptic curve defined over a number field and its corresponding adjoint motive. In order to do that, we will introduce some notation: For any vector space $V$, we write $\det(V):=\bigwedge^{\dim V}V$. If $V$ is a $\C$-vector space, we say that a $\Q$-vector space $W\subset V$ is a \emph{$\Q$-structure of $V$} if $W\otimes_\Q\C=V$. Given an isomorphism of $\C$-vector spaces $\alpha: V_1 \to V_2$, and chosen $\Q$-structures $W_1 \subset V_1$ and $W_2 \subset V_2$, we define $\det(\alpha) \in \C^\times / \Q^\times$ to be the class of the determinant of any matrix representing $\alpha$ with respect to any choice of bases of $W_1$ and $W_2$. On the other side, given an exact sequence of \( \mathbb{C} \)-vector spaces
\[
0 \longrightarrow V_1 \stackrel{\alpha}{\longrightarrow} V_2 \stackrel{\beta}{\longrightarrow} V_3 \longrightarrow 0,\qquad d=\dim(V_3),
\]
and \( \mathbb{Q} \)-structures \( W_1 \subseteq V_1 \) and \( W_2 \subseteq V_2 \), we can naturally define a \( \mathbb{Q} \)-structure \( D_3 \subseteq \det(V_3) \) by the following rule: an element \( w_3 \in \det(V_3) \) lies in \( D_3 \) if and only if, for every lift \( \tilde{w}_3 \in \bigwedge^{d} V_2 \) of \( w_3 \), there exists \( w_1 \in \det(W_1) \) such that \( \tilde{w}_3 \wedge \alpha(w_1) \in \det(W_2) \). Note that, if $\alpha$ is an isomorphism (hence, $V_3=0$), we have that $D_3=\det(\alpha)^{-1}\Q\subset \C=\det (V_3)$. In the general setting, the choice of a basis $B=\{b_3^i\}_{i=1,\cdots,d}\subset V_3$ defines a $\Q$-structure $W_3^{B}$ of $V_3$. Given any section $s:V_3\rightarrow V_2$ of $\beta$, one can compute $D_3=\Q\det(\alpha\oplus s)^{-1}\bigwedge_{i=1}^{d}b_3^i$, where the determinant is taken with respect to the $\Q$-structure $W_1\oplus W_3^B$ of $V_1\oplus V_3$.
The same formalism applies when the $V_i$ are $\R$-vector spaces rather than $\C$-vector spaces.

 Let $\mfM$ be a motive over $\Q$ of weight $w\leq-1$, and write $\mfM_B$ and $\mfM_{\rm dR}$ for its Betti and de Rham realizations.  Notice that, under the comparison map
\[
I_\infty: \mfM_{\rm dR}\otimes_\Q\C\longrightarrow \mfM_{B}\otimes_\Q\C,
\]
$\mfM_{\rm dR}\otimes_\Q\R$ corresponds to $(\mfM_{B}^+\otimes_\Q\R)\oplus(\mfM_{B}^-\otimes_\Q\R(-1))$, where $(\cdot)^\pm$ stands for the subspace where the action of complex conjugation $C_\infty$  on $\mfM_B$ acts as $\pm1$. 
Thus, the natural projection provides a morphism 
\begin{equation}\label{eqtildepi1}
\tilde \pi_1:F^0\mfM_{\rm dR}\otimes_\Q\R\hookrightarrow\mfM_{\rm dR}\otimes_\Q\R\stackrel{I_\infty}{\simeq}(\mfM_{B}^+\otimes_\Q\R)\oplus(\mfM_{B}^-\otimes_\Q\R(-1)){\longrightarrow} (\mfM_{B}^-\otimes_\Q\R(-1)),
\end{equation}
that turns out to be injective. Deligne cohomology $H^1_\cD(\mfM_\R)$ can be computed as the cokernel of the morphism $\tilde \pi_1$, namely, 
\[
0\longrightarrow F^0\mfM_{\rm dR}\otimes_\Q\R\stackrel{\tilde \pi_1}{\longrightarrow} \mfM_B^-\otimes_\Q\R(-1)\longrightarrow H^1_\cD(\mfM_\R)\longrightarrow 0
\]
The $\Q$-structures $F^0\mfM_{\rm dR}$ and $\mfM_B^{-}\otimes\Q(-1)$ provide a $\Q$-structure $\cR$ on $\det(H^1_\cD(\mfM_\R))$. Belinson's conjectures describe the determinant of the motivic cohomology $H^1_{\cM}(\mfM)$ of $M$ in terms of $\cR$. In this note, we use a simplified version of them:
\begin{conjecture}[Beilinson]\label{Beilconj}
Assume that $L(\mfM,s)$ has no poles at $s=0$.
    If $w< -1$ then Beilinson's regulator map defines an isomorphism
    \[
    r:H^1_{\cM}(\mfM)\otimes_\Q\R\stackrel{\simeq}{\longrightarrow}H^1_\cD(\mfM_\R);\quad\mbox{such that}\quad \det\left(r\left(H^1_{\cM}(\mfM)\right)\right)=(2\pi i)^{\dim(\mfM_B^-)}L(0,\mfM)^*\det(I_\infty)\cR,
    \]
    where $L(0,\mfM)^*$ denotes the leading term of $L(s,\mfM)$ at $s=0$.
    If $w=-1$ and $L(0,\mfM)\neq 0$, then we have
    \[
    (2\pi i)^{\dim(\mfM_B^-)}L(0,\mfM)\det(I_\infty)\cR=\Q. 
    \]
\end{conjecture}

In the remainder of the section we will compute $\cR_i\subset \det(H^1_\cD(\mfM^i_\R))$ for the motives $\mfM^1=h^1(A)_\Q(1)$, of weight $w_1=-1$, and $\mfM^2={\rm Ad}(h^1(A)_\Q)(1)$, of weight $w_2=-2$. Let us decompose $\Sigma_F=\Sigma_F^\R\sqcup\Sigma_F^\C$ by the real and complex places. For any archimedean place $\sigma$, we will fix an embedding $\sigma:F\hookrightarrow \C$ representing it and, if $\sigma\in\Sigma_F^\C$, we will denote by $\bar\sigma:F\hookrightarrow \C$ its composition with complex conjugation. For any embedding $\nu:F\hookrightarrow\C$, we can describe $A_{\nu}=A\times_{\nu}\C$ as a complex torus satisfying
\[
A_{\nu}(\C)\sim \C/\Lambda_{\nu},\qquad \Lambda_{\nu}=\Z\Omega_{{\nu},1}+\Z\Omega_{{\nu},2},\qquad \left\{\begin{array}{ll}
   \Omega_{\sigma,1}\in\R;\quad \Omega_{\sigma,2}\in \R i,\quad \tau_\sigma:=\Omega_{\sigma,1}/\Omega_{\sigma,2},  &\sigma\in\Sigma_F^\R;  \\
    \overline{\Omega_{\sigma,1}}=\Omega_{\bar\sigma,1},\quad\overline{\Omega_{\sigma,2}}=\Omega_{\bar\sigma,2}, & \sigma\in\Sigma_F^\C.
\end{array}\right.
\]
Notice that our chosen basis for $\Lambda_{\nu}=\Z\Omega_{{\nu},1}+\Z\Omega_{{\nu},2}$ provides the identification
    \[
    \mfM^1_B=(2\pi i)h^1(A)_{B}=\bigoplus_{\nu:F\hookrightarrow\C}2\pi i\Hom(\Lambda_{\nu},\Q)=\bigoplus_{\nu:F\hookrightarrow\C}2\pi i\Q^2.
    \]
    We will assume that a $F$-rational invariant differential $\omega_A$ is identified with the differential $dz$ of $\C/\Lambda_\sigma$, for every $\sigma\in\Sigma_F$. This implies that a choice of a $F$-basis $\{\omega_A,\eta_A\}$ provides an identification $\mfM^1_{\rm dR}=H^1_{\rm dR}(A)=F^2$. The comparison isomorphism between $\mfM^1_{\rm dR}\otimes_\Q\C$ and $\mfM^1\otimes_\Q\C$ is provided by the injection 
    \[
    I_{\infty}^1:\mfM^1_{\rm dR}\hookrightarrow \mfM^1_{B}\otimes_\Q\C;\qquad (e,f)\longmapsto \left(({\nu}(e)\Omega_{1,{\nu}}+{\nu}(f)\lambda_{1,{\nu}}\Omega_{1,{\nu}},{\nu}(e)\Omega_{2,{\nu}}+{\nu}(f)\lambda_{2,{\nu}}\Omega_{2,{\nu}})_\nu\right).
    \]
    for some $\lambda_{1,{\nu}},\lambda_{2,{\nu}}\in \C$.
    To compute $\det(I_{\infty})$, notice that $I_{\infty}$ is the composition of the morphisms
    \[
    I^1_{\infty}:\mfM^1_{\rm dR}\simeq F^2\stackrel{\alpha_1}{\longrightarrow}\bigoplus_{\nu:F\hookrightarrow\C}\C^2\stackrel{\alpha_2}{\longrightarrow}\bigoplus_{\nu:F\hookrightarrow\C}\C^2\simeq \mfM^1_B\otimes_\Q\C,
    \]
    where $\alpha_1(e,f)=(\nu(e),\nu(f))_\nu$ and $\alpha_2(x_\nu,y_\nu)_\nu=((x_\nu\Omega_{1,{\nu}}+y_\nu\lambda_{1,{\nu}}\Omega_{1,{\nu}},x_\nu\Omega_{2,{\nu}}+y_\nu\lambda_{2,{\nu}}\Omega_{2,{\nu}}))_\nu$. If we consider the $\Q$-structure provided by the canonical basis in the middle space and the $\Q$-basis of $F^2$ provided by an integral basis of $\cO_F$, it is clear that $\det(\alpha_1)=\Delta_F$, the discriminant of $F$, and $\det(\alpha_2)=(2\pi i)^{-2d}\prod_{\nu}\Omega_{1,{\nu}}\Omega_{2,{\nu}}\prod_{{\nu}}(\lambda_{2,{\nu}}-\lambda_{1,{\nu}})$, hence,
    \begin{eqnarray*}
    \det (I_{\infty}^1)=\det(\alpha_1)\det(\alpha_2)=\frac{\Delta_F}{(2\pi i)^{2d}}\prod_{\nu}\Omega_{1,{\nu}}\Omega_{2,{\nu}}\prod_{{\nu}}(\lambda_{2,{\nu}}-\lambda_{1,{\nu}}).
    \end{eqnarray*}
    
Using the Weil pairing, one can deduce that $\wedge^{2d}\mfM^1\simeq \Q(d)$, hence,
    \begin{equation}\label{relationlambdaomega}
\det( I^1_{\infty})\Q^\times=(2\pi i)^{-2d}\prod_{\nu}\Omega_{1,{\nu}}\Omega_{2,{\nu}}\prod_{{\nu}}(\lambda_{2,{\nu}}-\lambda_{1,{\nu}})\Q^\times=(2\pi i)^{-d}\Q^\times.
    \end{equation}
For any $\sigma\in\Sigma_F$, let $\mfM^1_{{\rm dR},\sigma}$ be the component of $\mfM^1_{\rm dR}\otimes_\Q\R$ corresponding to $\sigma$ under the identification $F\otimes_\Q\R\simeq\bigoplus_{\sigma\in\Sigma_F^\R}\R\oplus\bigoplus_{\sigma\in\Sigma_F^\C}\C$. The previously introduced  $I^1_{\infty}$ provides morphisms:
    \[
    \left\{\begin{array}{lll}\mfM^1_{\rm dR,\sigma}\simeq \R^2\longrightarrow\C^2\simeq \mfM^1_{B,\sigma};& \big(\begin{smallmatrix}a\\b\end{smallmatrix}\big)\longmapsto\big(\begin{smallmatrix}
        \Omega_{1,\sigma}&\lambda_{1,\sigma}\Omega_{1,\sigma}\\\Omega_{2,\sigma}&\lambda_{2,\sigma}\Omega_{2,\sigma}
    \end{smallmatrix}\big)\big(\begin{smallmatrix}a\\b\end{smallmatrix}\big);&\sigma\in\Sigma_F^\R;\\
    \mfM^1_{\rm dR,\sigma}\simeq \C^2\longrightarrow\C^2\oplus\C^2\simeq \mfM^1_{B,\sigma};&\big(\begin{smallmatrix}a\\b\end{smallmatrix}\big)\longmapsto\left(\big(\begin{smallmatrix}
        \Omega_{1,\sigma}&\lambda_{1,\sigma}\Omega_{1,\sigma}\\\Omega_{2,\sigma}&\lambda_{2,\sigma}\Omega_{2,\sigma}
\end{smallmatrix}\big)\big(\begin{smallmatrix}a\\b\end{smallmatrix}\big),\big(\begin{smallmatrix}
\overline{\Omega_{1,\sigma}}&\overline{\lambda_{1,\sigma}}\overline{\Omega_{1,\sigma}}\\\overline{\Omega_{2,\sigma}}&\overline{\lambda_{2,\sigma}}\overline{\Omega_{2,\sigma}}
\end{smallmatrix}\big)\big(\begin{smallmatrix}\bar a\\\bar b\end{smallmatrix}\big)\right);&\sigma\in\Sigma_F^\C.\end{array}\right.
    \]
    By the above description of the lattices $\Lambda_\nu$, the action of $C_\infty$  on $\mfM^1_{B,\sigma}$ is given by multiplication by $\big(\begin{smallmatrix}
        -1&0\\0&1
    \end{smallmatrix}\big)$, if $\sigma\in\Sigma_F^\R$, and first swapping the two components isomorphic to $\C^2$ and then multiplying by $-1$, if $\sigma\in\Sigma_F^\C$. This implies that 
    \[
   \mfM^{1-}_B\otimes_\Q \R(-1)=\bigoplus_{\sigma\in\Sigma_F}\mfM^{1-}_{B,\sigma};\qquad \left\{\begin{array}{ll}
        \mfM^{1-}_{B,\sigma}=\left\{\big(\begin{smallmatrix}a\\ 0\end{smallmatrix}\big);\;a\in\R\right\}\subset \mfM^{1}_{B,\sigma}; &\sigma\in\Sigma_F^\R;  \\
         \mfM^{1-}_{B,\sigma}=\left\{\left(\big(\begin{smallmatrix}a\\ b\end{smallmatrix}\big),\big(\begin{smallmatrix}a\\ b\end{smallmatrix}\big)\right);\;a,b\in\R\right\}\subset \mfM^{1}_{B,\sigma};& \sigma\in\Sigma_F^\C.
    \end{array}\right.
    \]
    Moreover, the morphism $\tilde\pi_1$ of \eqref{eqtildepi1} in this setting is provided by
    \[
    \tilde \pi_1:(F\otimes_\Q\R)\omega_A\simeq F\otimes_\Q\R\simeq \bigoplus_{\sigma\in\Sigma_F}F_\sigma\stackrel{\bigoplus_\sigma\tilde\pi_\sigma}{\longrightarrow} \bigoplus_{\sigma\in\Sigma_F}\mfM^{1-}_{B,\sigma};\qquad \tilde\pi_\sigma(x_\sigma)=\left\{\begin{array}{ll}
         \Omega_{1,\sigma}\big(\begin{smallmatrix}x_\sigma\\ 0\end{smallmatrix}\big); &\sigma\in\Sigma_F^\R;  \\
        \left(\big(\begin{smallmatrix}{\rm Re}(x_\sigma\Omega_{1,\sigma})\\ {\rm Re}(x_\sigma\Omega_{2,\sigma})\end{smallmatrix}\big),\big(\begin{smallmatrix}{\rm Re}(x_\sigma\Omega_{1,\sigma})\\ {\rm Re}(x_\sigma\Omega_{2,\sigma})\end{smallmatrix}\big)\right);& \sigma\in\Sigma_F^\C.
    \end{array}\right.
    \]
    Thus, $\tilde\pi_1$ is an isomorphism and $\cR^1=\det(\tilde\pi_1)^{-1}\Q$. For convenience, we consider the $\Q$-basis of $F^0\mfM^1_{\rm dR}\simeq F\omega_A\simeq F$ provided by an integral basis of $\cO_F$, and the basis $\{e_\sigma,\;\sigma\in\Sigma_F^\R,\;e_\sigma^1,e_\sigma^2,\sigma\in\Sigma_F^\C\}\subset\mfM_{B}^{1-}\otimes\Q(-1)$, where
    \[
    e_\sigma=\big(\begin{smallmatrix}1\\ 0\end{smallmatrix}\big),\quad\sigma\in\Sigma_F^\R;\qquad e_\sigma^1=\left(\big(\begin{smallmatrix}1\\ 0\end{smallmatrix}\big),\big(\begin{smallmatrix}1\\ 0\end{smallmatrix}\big)\right);\quad e_\sigma^2=\left(\big(\begin{smallmatrix}0\\ 1\end{smallmatrix}\big),\big(\begin{smallmatrix}0\\ 1\end{smallmatrix}\big)\right),\quad\sigma\in\Sigma_F^\C.
    \]
    Notice that we can interpret the restriction $\tilde\pi_1\mid_F$ as the composition
    \begin{eqnarray*}
    &&\tilde\pi_1\mid_F:F\stackrel{\alpha_1}{\longrightarrow}\bigoplus_{\nu:F\hookrightarrow\C}\C\stackrel{\beta_2}{\longrightarrow}\bigoplus_{\sigma\in\Sigma_F}(\mfM^{1-}_{B,\sigma}\otimes_\R\C);\qquad \alpha_1(a)=(\nu(a))_{\nu:F\hookrightarrow\C}\\
    &&\beta_2\left((x_\nu)_{\nu:F\hookrightarrow\C}\right)=\sum_{\sigma\in\Sigma_F^\R}\Omega_{1,\sigma}x_\sigma e_\sigma+\sum_{\sigma\in\Sigma_F^\R}\left(\frac{x_\sigma\Omega_{1,\sigma}+x_{\bar\sigma}\overline{\Omega_{1,\sigma}}}{2}\right)e_\sigma^1+\left(\frac{x_\sigma\Omega_{2,\sigma}+x_{\bar\sigma}\overline{\Omega_{2,\sigma}}}{2}\right)e_\sigma^2.
    \end{eqnarray*}
    If we consider the $\Q$-structure on $\bigoplus_{\nu:F\hookrightarrow\C}\C$ provided by the canonical basis, we deduce using Brill's theorem
    \[
\det(\tilde\pi_1)=\det(\alpha_1)\det(\beta_2)=\Delta_F^{1/2}\prod_{\sigma\in\Sigma_F^\R}\Omega_{1,\sigma}\prod_{\sigma\in\Sigma_F^\C}\frac{i{\rm Im}(\Omega_{1,\sigma}\overline{\Omega_{2,\sigma}})}{2}=|d_F|^{1/2}\prod_{\sigma\in\Sigma_F^\R}\Omega_{1,\sigma}\prod_{\sigma\in\Sigma_F^\C}\frac{{\rm Im}(\Omega_{1,\sigma}\overline{\Omega_{2,\sigma}})}{2}.
    \]
    Thus, Conjecture \ref{Beilconj} (see Conjecture \ref{BSDBeilinson}) predicts that, when $L(0,\mfM^1)=L(1,A)\neq0$,
    \begin{equation*}
    L(1,A)=L(0,\mfM^1)\in(2\pi i)^{-\dim(\mfM_B^{1-})}\det(I^1_\infty)^{-1}\cR_1^{-1}=\det((\tilde\pi_1))\Q=|d_F|^{1/2}\prod_{\sigma\in\Sigma_F^\R}\Omega_{1,\sigma}\prod_{\sigma\in\Sigma_F^\C}{\rm Im}(\Omega_{1,\sigma}\overline{\Omega_{2,\sigma}})\Q.
    \end{equation*}

Let us consider now $\mfM^2={\rm Ad}(h^1(A)_\Q)(1)$. In this setting, the comparison morphism between $\mfM^2_{\rm dR}\otimes_\Q\C$ and $\mfM^2_B\otimes_\Q\C$ is provided by the injection
\begin{eqnarray*}
I^2_{\infty}:\mfM^2_{\rm dR}=\M_2(F)_0&\hookrightarrow& \bigoplus_{{\nu}:F\hookrightarrow\C}\M_2(\C)_0=\mfM^2_B\otimes_\Q\C;\\ 
\beta&\mapsto& \left(g_{\nu} {\nu}(\beta)g_{\nu}^{-1}\right)_{\nu},\qquad\qquad\qquad g_{\nu}=\left(\begin{array}{cc}\Omega_{1,{\nu}}&\lambda_{1,{\nu}}\Omega_{1,{\nu}}\\\Omega_{2,{\nu}}&\lambda_{2,{\nu}}\Omega_{2,{\nu}}\end{array}\right).
\end{eqnarray*}
To compute $\det(I^2_\infty)$, we consider the following basis of $\M_2(F)_0$ and $\mfM_B^2=\bigoplus_{\nu:F\hookrightarrow\C}(2\pi i)\M_2(\Q)_0$:
\begin{eqnarray*}
\left\{b_1=\big(\begin{smallmatrix}1&0\\ 0&-1\end{smallmatrix}\big),b_2=\big(\begin{smallmatrix}0&1\\ 0&0\end{smallmatrix}\big),b_3=\big(\begin{smallmatrix}0&0\\ 1&0\end{smallmatrix}\big)\right\}&\subset&\M_2(F)_0,\\
 \left\{b_\nu^1=(2\pi i)\big(\begin{smallmatrix}1&0\\ 0&-1\end{smallmatrix}\big),b_\nu^2=(2\pi i)\big(\begin{smallmatrix}0&1\\ 0&0\end{smallmatrix}\big),b_\nu^3=(2\pi i)\big(\begin{smallmatrix}0&0\\ 1&0\end{smallmatrix}\big)\right\}_{\nu:F\hookrightarrow\C}&\subset&\bigoplus_{\nu:F\hookrightarrow\C}(2\pi i)\M_2(\Q)_0.
\end{eqnarray*}
Then it is clear that $I_\infty$ is provided by the composition
\begin{eqnarray*}
I_\infty^2:\M_2(F)_0&\stackrel{\alpha_1}{\longrightarrow}&\bigoplus_{\nu:F\hookrightarrow\C}\M_2(\C)_0\stackrel{\delta_2}{\longrightarrow}\bigoplus_{\nu:F\hookrightarrow\C}\M_2(\C)_0;\\ \alpha_1(f_1b_1+f_2b_2+f_3b_3)&=&\left(\nu(f_1)({2\pi i})^{-1}b_\nu^1+\nu(f_2)({2\pi i})^{-1}b_\nu^2+\nu(f_3)({2\pi i})^{-1}b_\nu^3\right)_\nu;\\
\delta_2\left(\left((b_\nu^1, b_\nu^2, b_\nu^3)\left(\begin{smallmatrix}x_\nu\\y_\nu\\z_\nu\end{smallmatrix}\right)\right)_\nu\right)&=&
\left((b_\nu^1, b_\nu^2, b_\nu^3)\left(\begin{smallmatrix}\frac{\lambda_{2,\nu}+\lambda_{1,\nu}}{\lambda_{2,\nu}-\lambda_{1,\nu}}&\frac{-1}{\lambda_{2,\nu}-\lambda_{1,\nu}}&\frac{\lambda_{2,\nu}\lambda_{1,\nu}}{\lambda_{2,\nu}-\lambda_{1,\nu}}\\\frac{-2\tau_\nu\lambda_{1,\nu}}{\lambda_{2,\nu}-\lambda_{1,\nu}}&\frac{\tau_\nu}{\lambda_{2,\nu}-\lambda_{1,\nu}}&\frac{\tau_\nu\lambda_{1,\nu}^2}{\lambda_{2,\nu}-\lambda_{1,\nu}}\\\frac{2\tau_\nu^{-1}\lambda_{2,\nu}}{\lambda_{2,\nu}-\lambda_{1,\nu}}&\frac{-\tau_\nu^{-1}}{\lambda_{2,\nu}-\lambda_{1,\nu}}&\frac{\tau_\nu^{-1}\lambda_{2,\nu}^2}{\lambda_{2,\nu}-\lambda_{1,\nu}}\end{smallmatrix}\right)\left(\begin{smallmatrix}x_\nu\\y_\nu\\z_\nu\end{smallmatrix}\right)\right)_\nu,\qquad\qquad\qquad\qquad\quad
\end{eqnarray*}
       where $\tau_\nu=\Omega_{1,\nu}\Omega_{2,\nu}^{-1}$. An easy calculation shows that the determinant of the matrix defining $\delta_2$ is 1, thus,
\[
\det(I^2_\infty)=\det(\alpha_1)\det(\delta_2)=\Delta_F^{3/2}(2\pi i)^{-3d}.
\]

The action of $C_\infty$ on $\mfM^2_B=\bigoplus_{\sigma\in \Sigma_F^\R}(2\pi i)\M_2(\Q)_0\oplus \bigoplus_{\sigma\in\Sigma_F^\C}(2\pi i)\M_2(\Q)_0\oplus(2\pi i)\M_2(\Q)_0)$ is given as follows: on the components indexed by real places in $\sigma\in\Sigma_F^\R$, it acts by multiplication by $-1$, followed by conjugation by $\big(\begin{smallmatrix}
    1&0\\0&-1
\end{smallmatrix}\big)$; on the components indexed by complex places $\sigma\in\Sigma_F^\C$, it interchanges the two summands isomorphic to $(2\pi i)\M_2(\Q)_0$ and then multiplies by $-1$. This implies that 
    \[
    \mfM^{2-}_B\otimes_\Q \R(-1)=\bigoplus_{\sigma\in\Sigma_F}\mfM^{2-}_{B,\sigma};\qquad \left\{\begin{array}{ll}
        \mfM^{2-}_{B,\sigma}=(2\pi i)^{-1}\R b^1_\sigma\subset \mfM^{2}_{B,\sigma}; &\sigma\in\Sigma_F^\R;  \\
         \mfM^{2-}_{B,\sigma}=\R\frac{b_\sigma^1+b_{\bar\sigma}^1}{2\pi i}+\R\frac{b_\sigma^2+b_{\bar\sigma}^2}{2\pi i}+\R\frac{b_\sigma^3+b_{\bar\sigma}^3}{2\pi i}\subset \mfM^{2}_{B,\sigma};& \sigma\in\Sigma_F^\C.
    \end{array}\right.
    \]
    In this situation, our previous choice of a $F$-basis $\{\omega_A,\eta_A\}$ of $H^1_{\rm dR}(A)$ provides an  identification 
\begin{equation}\label{dRfil}
0\simeq F^1\mfM^2_{\rm dR}\subset F^0\mfM^2_{\rm dR}=Fb_2\subset F^{-1}\mfM^2_{\rm dR}=Fb_1+Fb_2\subset F^{-2}\mfM^2_{\rm dR}=\mfM^2_{\rm dR}\simeq \M_2(F)_0.
\end{equation}
Thus, the Hodge decomposition of $\mfM^2_B$ is given by
\begin{equation}\label{HdgDeceq}
  \begin{small}
    (\mfM^2_B)^{-1,-1}=\bigoplus_{{\nu}:F\hookrightarrow\C}\C\left(\begin{smallmatrix}-{\rm Re}(\tau_\nu)&|\tau_\nu|^2\\-1&{\rm Re}(\tau_\nu)\end{smallmatrix}\right);\quad  (\mfM^2_B)^{0,-2}=\bigoplus_{{\nu}:F\hookrightarrow\C}\C\left(\begin{smallmatrix}1&-\tau_\nu\\\tau_\nu^{-1}&-1\end{smallmatrix}\right); \quad  (\mfM^2_B)^{-2,0}=\overline{(\mfM^2_B)^{0,-2}}.\end{small}
    \end{equation}
 Moreover, the morphism $\tilde\pi_1$ of \eqref{eqtildepi1} is provided  by the composition 
  \[
  \tilde \pi_1:F^0\mfM^2_{\rm dR}\otimes_\Q\R= (F\otimes_\Q\R)b_2\stackbin[\simeq]{\alpha_1\mid_{F^0}}{\longrightarrow} \bigoplus_{\sigma\in\Sigma_F}F_\sigma (2\pi i)^{-1}b_{\sigma}^2\stackbin{\bigoplus_\sigma\tilde\pi_\sigma}{\longrightarrow} \bigoplus_{\sigma\in\Sigma_F}\mfM^{2-}_{B,\sigma},
  \]
  where
    \begin{equation*}\label{tildepisigma}
     \tilde\pi_\sigma(x_\sigma(2\pi i)^{-1}b_{\sigma}^2)=\left\{\begin{array}{ll}
         -x_\sigma(\lambda_{2,\sigma}-\lambda_{1,\sigma})^{-1}(2\pi i)^{-1}b_\sigma^1; &\sigma\in\Sigma_F^\R;  \\
        {\rm Re}\left(\frac{-x_\sigma}{\lambda_{2,\sigma}-\lambda_{1,\sigma}}\right)\frac{b_\sigma^1+b_{\bar\sigma}^1}{2\pi i}+{\rm Re}\left(\frac{x_\sigma\tau_\sigma}{\lambda_{2,\sigma}-\lambda_{1,\sigma}}\right)\frac{b_\sigma^2+b_{\bar\sigma}^2}{2\pi i}+{\rm Re}\left(\frac{-x_\sigma\tau_\sigma^{-1}}{\lambda_{2,\sigma}-\lambda_{1,\sigma}}\right)\frac{b_\sigma^3+b_{\bar\sigma}^3}{2\pi i};& \sigma\in\Sigma_F^\C.
    \end{array}\right.
    \end{equation*}
Hence,  we can identify 
\[
H^1_{\cD}(\mfM^2_\R)=\bigoplus_{\sigma\in\Sigma_F^\C}\M_2(\R)_0/\R\left(\begin{smallmatrix}1&-{\rm Re}\left(\tau_{\sigma}\right)\\{\rm Re}\left(\tau_{\sigma}^{-1}\right)&-1\end{smallmatrix}\right)\oplus \R\left(\begin{smallmatrix}0&{\rm Im}\left(\tau_{\sigma}\right)\\-{\rm Im}\left(\tau_{\sigma}^{-1}\right)&0\end{smallmatrix}\right).
\]
It is convenient to choose the following representatives of the above quotients:
\begin{equation}\label{defHsigma}
H_\sigma:=\frac{-{\rm Re}(\tau_\sigma)}{{\rm Im}(\tau_\sigma)}\frac{b_\sigma^1}{2\pi i}+\frac{|\tau_\sigma|^2}{{\rm Im}(\tau_\sigma)}\frac{b_\sigma^2}{2\pi i}-\frac{1}{{\rm Im}(\tau_\sigma)}\frac{b_\sigma^3}{2\pi i}\in\M_2(\R)_0/\R\left(\begin{smallmatrix}1&-{\rm Re}\left(\tau_{\sigma}\right)\\{\rm Re}\left(\tau_{\sigma}^{-1}\right)&-1\end{smallmatrix}\right)\oplus \R\left(\begin{smallmatrix}0&{\rm Im}\left(\tau_{\sigma}\right)\\-{\rm Im}\left(\tau_{\sigma}^{-1}\right)&0\end{smallmatrix}\right).
\end{equation}
It is easy to show that $\{H_\sigma\}_{\sigma\in\Sigma_F^\C}$ form a basis for $H^1_{\cD}(\mfM^2_\R)$. Thus, 
$\cR_2=\Q\det(\tilde\pi_1\oplus s)^{-1}\bigwedge_{\sigma\in\Sigma_F^\C} H_\sigma$,  for any section $s:H^1_{\cD}(\mfM^2_\R)\rightarrow\mfM^{2-}_B\otimes_\Q\R(-1)$, and the determinant is taken with respect to the $\Q$-structures $Fb_2\oplus\bigoplus_{\sigma\in\Sigma_F^\C}\Q H_\sigma$ and $\mfM^{2-}_B\otimes_\Q\Q(-1)$. We can choose $\tilde\pi_1\oplus s$ to be induced by the composition
\[
F b_2\oplus \bigoplus_{\sigma\in\Sigma_F^\C}\Q H_\sigma\stackrel{\alpha_1\oplus{\rm Id}}{\longrightarrow}\bigoplus_{\nu:F\hookrightarrow\C}\C \frac{b_\nu^2}{2\pi i}\oplus\bigoplus_{\sigma\in\Sigma_F^\C}\Q H_\sigma\stackrel{\gamma_2}{\longrightarrow}\mfM^{2-}_B\otimes_\Q\C,
\]
where  $\gamma_2$ is given by
\begin{small}
\begin{eqnarray*}
\gamma_2\left(\sum_\nu x_\nu \frac{b^2_\nu}{2\pi i}+\sum_\sigma y_\sigma H_\sigma\right)=\qquad\qquad\qquad\qquad\qquad\qquad\qquad\qquad\qquad\qquad\qquad\qquad\qquad\qquad\qquad\qquad\qquad\qquad\\
 =\sum_{\sigma\in\Sigma_F^\R}\frac{-x_\sigma}{\lambda_{2,\sigma}-\lambda_{1,\sigma}} \frac{b^1_\sigma}{2\pi i}+\sum_{\sigma\in\Sigma_F^\R}\left(\frac{b^1_\sigma}{2\pi i},\frac{b^2_\sigma}{2\pi i},\frac{b^3_\sigma}{2\pi i}\right)\left(\begin{array}{ccc}
    \frac{-2^{-1}}{\lambda_{2,\sigma}-\lambda_{1,\sigma}} & \frac{-2^{-1}}{\bar\lambda_{2,\sigma}-\bar\lambda_{1,\sigma}} &\frac{-{\rm Re}(\tau_\sigma)}{{\rm Im}(\tau_\sigma)}\\
    \frac{\tau_\sigma2^{-1}}{\lambda_{2,\sigma}-\lambda_{1,\sigma}} &\frac{\bar\tau_\sigma2^{-1}}{\bar\lambda_{2,\sigma}-\bar\lambda_{1,\sigma}} &\frac{|\tau_\sigma|^2}{{\rm Im}(\tau_\sigma)}\\
    \frac{-\tau_\sigma^{-1}2^{-1}}{\lambda_{2,\sigma}-\lambda_{1,\sigma}}&\frac{-\bar\tau_\sigma^{-1}2^{-1}}{\bar\lambda_{2,\sigma}-\bar\lambda_{1,\sigma}}&\frac{-1}{{\rm Im}(\tau_\sigma)}
\end{array}\right)\left(\begin{array}{c}x_\sigma\\x_{\bar\sigma}\\y_\sigma\end{array}\right).
\end{eqnarray*}
\end{small}
Since $\{(2\pi i)^{-1}b_\sigma^k\}$ defines a basis for $\mfM^{2-}_B\otimes_\Q(-1)$, this implies that 
\begin{small}
\[
\det(\tilde\pi_1\oplus s)=\Delta_F^{1/2}\prod_{\sigma\in\Sigma_F^\R}\frac{-1}{\lambda_{2,\sigma}-\lambda_{1,\sigma}}\prod_{\sigma\in\Sigma_F^\C}\left|\begin{smallmatrix}
    \frac{-2^{-1}}{\lambda_{2,\sigma}-\lambda_{1,\sigma}} & \frac{-2^{-1}}{\bar\lambda_{2,\sigma}-\bar\lambda_{1,\sigma}} &\frac{-{\rm Re}(\tau_\sigma)}{{\rm Im}(\tau_\sigma)}\\
    \frac{\tau_\sigma2^{-1}}{\lambda_{2,\sigma}-\lambda_{1,\sigma}} &\frac{\bar\tau_\sigma2^{-1}}{\bar\lambda_{2,\sigma}-\bar\lambda_{1,\sigma}} &\frac{|\tau_\sigma|^2}{{\rm Im}(\tau_\sigma)}\\
    \frac{-\tau_\sigma^{-1}2^{-1}}{\lambda_{2,\sigma}-\lambda_{1,\sigma}}&\frac{-\bar\tau_\sigma^{-1}2^{-1}}{\bar\lambda_{2,\sigma}-\bar\lambda_{1,\sigma}}&\frac{-1}{{\rm Im}(\tau_\sigma)}
\end{smallmatrix}\right|=\frac{\Delta_F^{1/2}}{(-1)^{r_1}}\frac{\prod_{\sigma\in\Sigma_F^\C}|\tau_\sigma|^{-2}{\rm Im}(\tau_\sigma)^2}{\prod_{\nu}(\lambda_{2,\nu}-\lambda_{1,\nu})}.
\]
\end{small}
Thus, by relation \eqref{relationlambdaomega},
\begin{equation*}
    \cR_2=\Q\det(\lambda_s)^{-1}\bigwedge_{\sigma\in\Sigma_F^\C} H_\sigma=\Q\Delta_F^{1/2}(2\pi i)^d\prod_{\sigma\in\Sigma_F^\R}\Omega_{1,{\sigma}}^{-1}\Omega_{2,{\sigma}}^{-1}\prod_{\sigma\in\Sigma_F^\C}{\rm Im}(\Omega_{1,\sigma}\overline{\Omega_{2,\sigma}})^{-2}\bigwedge_{\sigma\in\Sigma_F^\C} H_\sigma,
\end{equation*}
and by Conjecture \ref{Beilconj} (see Conjecture \ref{R2calc})
\[
 \det\left(r\left(H^1_{\cM}(\mfM^2)\right)\right)= L(0,\mfM^2)^*(2\pi i)^{-r_1-r_2}\prod_{\sigma\in\Sigma_F^\R}\Omega_{1,{\sigma}}^{-1}\Omega_{2,{\sigma}}^{-1}\prod_{\sigma\in\Sigma_F^\C}{\rm Im}(\Omega_{1,\sigma}\overline{\Omega_{2,\sigma}})^{-2}\left(\bigwedge_{\sigma\in\Sigma_F^\C} H_\sigma\right)\Q.
\]

\section{Strongly admissible automorphic representations}\label{sec:apendix B} Let $A$ be an elliptic curve over a number field $F$ of conductor $N$. Suppose that $A$ corresponds to an automorphic representation $\Pi$ of $\PGL_2$.  Recall that we call a sign vector $\varepsilon\in\{\pm 1\}^{\Sigma_F}$ of lowest degree if $\varepsilon_v=1$ for all complex places $v$. Following, and slightly generalizing, a terminology introduced by Oda in \cite{oda83}, we say that $\Pi$ is \emph{strongly admissible} if for any sign vector of lowest degree $\varepsilon\in \{\pm 1\}^{\Sigma_F}$, there exists a quadratic Hecke character $\rho$ of $F$ of conductor coprime to $N$ such that $L(1,\Pi,\rho)\neq 0$ and $\rho_v(-1)=\varepsilon_v$ for all $v\in\Sigma_F$.

Denote by $\Pi_\rho$ the twist of $\Pi$ by a quadratic character $\rho$ of $F$ of conductor coprime to $N$. The signs (root numbers) of $\Pi$ and $\Pi_\rho$ are related by the formula
\begin{equation}\label{signs}
  \sgn(\Pi) \cdot \sgn(\Pi_\rho) = \sign(\rho)\cdot \rho(N),
\end{equation}
where $\sign(\rho)=\prod_{v\in\Sigma_F}\rho_v(-1)$ (see, e.g., \cite[p. 338]{rohrlich}).

\begin{proposition}\label{strongadm}
If $N$ is not a square, then $\Pi$ is strongly admissible.
\end{proposition}

\begin{proof}
Let $\varepsilon$ be a sign vector of lowest degree. Let $\fp $ be a prime dividing $N$ such that $\mathrm{ord}_\fp(N)$ is odd, and let  let $\Sigma$ be a set of places of $F$ defined either as
  \begin{equation}\label{sigma11}
\Sigma=    \{v\in\Sigma_F \text{ such that } \varepsilon_v = -1\},
  \end{equation}
  or as
  \begin{equation}\label{sigma12}
   \Sigma= \{v\in\Sigma_F \text{ such that } \varepsilon_v = -1\}\cup \{\fp\}.
  \end{equation}
  Given $\Pi$ and $\varepsilon$, we choose between definitions \eqref{sigma11} and \eqref{sigma12} in such a way that $\Sigma$ has even cardinality if $\sgn(\Pi)=1$ and odd cardinality if $\sgn(\Pi)=-1$. For every place $v$ which is either archimedean or such that it divides $N$, choose a local character $\theta_v$ of $F_v^\times$ as follows:
  \begin{enumerate}
  \item if $v\in \Sigma_F$ and  $\varepsilon_v = -1$, define $\theta_v$ to be the nontrivial quadratic character;
  \item  if $v\in\Sigma_F$ and  $\varepsilon_v = 1$, define  $\theta_v=1$;
  \item if $v\mid N$ and $v\neq \fp$, define $\theta_v=1$;
    \item if $v = \fp$ and $\fp\not\in\Sigma$, define $\theta_v = 1$;
  \item if $v=\fp$ and $\fp\in \Sigma$, define $\theta_v$ to be the nontrivial unramified character of $F_v^\times$.
    \end{enumerate}

By the Grunwald-Wang Theorem (see, e.g., \cite[Theorem 2.4]{milneCFT}) there exists a quadratic Hecke character $\theta$ of $F$ that locally coincides with $\theta_v$ at all archimedean $v$ and all  $v\mid N$. By our choice of the $\theta_v$'s we have that $\theta$ has conductor coprime to $N$ and
  \begin{equation*}
    \sgn(\Pi) \cdot \sgn(\Pi_\theta) = (-1)^{|\Sigma|},
  \end{equation*}
which, by the choice of $\Sigma$, implies that $\sgn(\Pi_\theta)=1$. Moreover, we have that $\theta_v(-1) = \varepsilon_v$ for $v\in\Sigma_F$.

Now we apply Waldspurger's theorem \cite[Th. 4]{Wa91} to $\Pi_\theta$: there exists $\xi\in F$ such that $\xi_v>0$ for all real $v$ such that $L(1,\Pi_\theta,\theta_\xi)\neq 0$, where $\theta_\xi$ denotes the quadratic character associated to $F(\sqrt{\xi})/F$. Define $\rho = \theta\cdot \theta_\xi$. For $v\in\Sigma_F$ we have that $\rho_v(-1)=(\theta\cdot \theta_\xi)_v(-1)=\theta_v(-1)=\varepsilon_v$. The theorem of Waldspurger guarantees that $\xi$ can be taken to satisfy also that
\begin{align}\label{eq: condition walds}
|\xi-1|_v<1 \text{ for all } v\mid N.  
\end{align}
 This shows that $\theta_\xi$ can be taken to be trivial outside $N$. Therefore, $\rho$ is unramified outside $N$, and we have that $L(1,\Pi_\rho)\neq 0$.
\end{proof}
\begin{remark}\label{rk:appendix}
  In condition $(3)$, the character $\theta_v$ can also be defined to be the non-trivial unramified character of $F_v^\times$ at an even number of places $v$ where $\ord_v(N)$ is odd, since this does not change the sign of $\Pi_\theta$. 
\end{remark}
\begin{remark}\label{rk:appendix2}
  Condition \ref{eq: condition walds} implies that, for any $v\mid N$, the character $\theta_{\xi,v}$ is the trivial character. Therefore, $\rho_v=\theta_v$ for $v\mid N$. We will use this property in Proposition \ref{prop:apendix second} below, in which we need to choose $\rho_v$ carefully at primes $v\mid N$.
\end{remark}

As in the main body of the text, we denote by $B$ a quaternion algebra over $B$ and by $G$ the algebraic group associated to $B^\times/F^\times$.
\begin{proposition}\label{prop:apendix second}
  Suppose that $N$ is not a square and that $\Pi$ admits a Jacquet--Langlands lift to $G$. Let $\varepsilon\in\{\pm 1\}^{\Sigma_B}$ be sign vector of lowest degree and let $\lambda\in\{\pm 1\}^{\Sigma_F\setminus\Sigma_B}$ be any sign vector. There exist quadratic Hecke characters $\rho_1,\rho_2:\I_F^\times/F^\times\rightarrow \{\pm 1\}$ with sign vectors $(\varepsilon,\lambda)$ and $(\varepsilon,-\lambda)$, respectively, of conductor coprime to $N$, and such that $L(1, \Pi, {\rho_i})\neq 0$. Moreover, $\rho_1$ and $\rho_2$ can be chosen so that the quadratic extension $E/F$ associated $\rho_1\cdot\rho_2$ admits an embedding into $B$.
\end{proposition}

\begin{proof}
  The first assertion follows directly from Proposition \ref{strongadm}. It remains to see that $\rho_1$ and $\rho_2$ can be chosen in such a way that $E$ admits an embedding into $B$. For this, we will see that all places $v$ that ramify in $B$ are inert or ramify in $E/F$. This is clearly satisfied at all the archimedean places, since by our choice of the sign vectors of $\rho_1$ and $\rho_2$ we have that $(\rho_1\cdot\rho_2)_v(-1)=-1$ at all $v\in\Sigma_F\setminus \Sigma_B$.

  Next, we deal with the condition at finite places. Suppose first that $|\Sigma_F\setminus\Sigma_B|$ is odd; that is, that $B$ ramifies at an odd number of archimedean places. Then $B$ ramifies an odd number of finite places. Since $\Pi$ admits a Jacquet--Langlands lift to $B$, all finite places where $B$ ramifies are places that divide $N$ with valuation 1. In particular, when choosing $\rho_1$ and $\rho_2$, we can take the prime $\fp$ in the proof of Proposition \ref{strongadm} to be a prime where $B$ ramifies. Also, $s_{(\varepsilon,\lambda)}$ and $s_{(\varepsilon,-\lambda)}$ have different parity (i.e., the number of coordinates where the sign is $-1$ in each vector have different parity), and this means that one of the characters $\rho_i$  has the set $\Sigma$ as in \eqref{sigma11} and the other as in \eqref{sigma12}. Therefore, the character $(\rho_1\cdot\rho_2)_\fp$ is the quadratic unramifed character of $F_\fp^\times$. For the remaining primes $\fq$ that ramify in $B$, since there are an even number of them, by Remark \ref{rk:appendix} and Remark \ref{rk:appendix2}, we can take $\rho_{1,\fq}$ to be the quadratic unramified character and $\rho_{2,\fq}$ to be the trivial character. Therefore, $(\rho_1\cdot\rho_2)_\fq$ is the quadratic unramifed character of $F_\fq^\times$.  This implies that, with this choices of $\rho_1$ and $\rho_2$, the extension $E/F$ is inert at all finite primes where $B$ ramifies.

  Suppose now that $|\Sigma_F\setminus\Sigma_B|$ is even; that is, that $B$ ramifies at an even number of archimedean places. Then $B$ ramifies an even number of finite places. If $B$ does not ramify at any finite place, there is no condition to check. If $B$ ramifies at some finite place, then it ramifies at least in two places, say $\fp_1$ and $\fp_2$. In this case $s_{(\varepsilon,\lambda)}$ and $s_{(\varepsilon,-\lambda)}$ have the same parity, so both characters $\rho_1$ and $\rho_2$  have the set $\Sigma$ as in \eqref{sigma11}, or both as in \eqref{sigma12}. If both fall in case \eqref{sigma12}, we can take the prime $\fp$ to be $\fp_1$ for $\rho_1$ and $\fp_2$ for $\rho_2$. In the choice of $\rho_1$, we choose $\rho_{1,\fp_2}$ to be trivial and in the choice of $\rho_2$ we choose $\rho_{2,\fp_1}$ to be trivial. For the other primes $\fq$ where $B$ ramifies (there are an even number of them), we take $\rho_{1,\fq}$ to be the unramified quadratic character and $\rho_{2,\fq}$ to be the trivial character. This implies that at all primes $\fq$ where $B$ is ramified (including $\fp_1$ and $\fp_2$), the local character $(\rho_1\cdot\rho_2)_\fq$ is the quadratic unramified character and therefore $\fq$ is inert in $E/F$. If both characters fall in case \eqref{sigma11}, then we can take $\rho_{1,\fq}$ to be the unramified quadratic character and $\rho_{2,\fq}$ to be the trivial character at all finite primes where $B$ ramifies, and again we have that $\fq$ is inert in $E/F$ for all such primes $\fq$.
\end{proof}
\begin{remark}\label{rk:ass}
  From the choice of $\rho_1$ and $\rho_2$, we see that for all $\fq\mid N$ such that $\fq$ is non-split in $E/F$, we have that $\ord_\fq(N)=1$. This is because $\rho_{i,\fq}$ can only be non-trivial at primes $\fq$ that ramify in $B$, so they divide $N$ exactly.
\end{remark}
\bibliographystyle{amsalpha}
\bibliography{biblio}
\end{document}